\documentclass[10 pt]{amsart}

\usepackage{amsmath, amsthm, amssymb, enumerate,
  graphicx, longtable, xcolor, tikz, comment}

\usepackage{url, hyperref}

\usepackage[margin=3cm]{geometry}

\usepackage[english]{babel}

\usepackage[all]{xy}

\theoremstyle{plain}
\newtheorem{thm}{Theorem}[section]
\newtheorem{pro}[thm]{Proposition}
\newtheorem{lem}[thm]{Lemma}
\newtheorem{cor}[thm]{Corollary}

\newtheorem{claim}[thm]{Claim}

\theoremstyle{definition}
\newtheorem{dfn}[thm]{Definition}
\newtheorem{nota}[thm]{Notation}
\newtheorem{setup}[thm]{Setup}
\newtheorem{rem}[thm]{Remark}
\newtheorem{exa}[thm]{Example}

\numberwithin{equation}{section}

\newcommand{\inv}{^{-1}}
\newcommand{\into}{\hookrightarrow} 
\newcommand{\onto}{\twoheadrightarrow} 
\renewcommand{\setminus}{\smallsetminus} 
\newcommand{\epsi}{\varepsilon}
\newcommand{\Gm}{\mathbb{G}_\mathrm{m}} 
\renewcommand{\emptyset}{\varnothing}
\newcommand{\phiv}{\varphi}

\DeclareMathOperator{\codim}{codim}
\DeclareMathOperator{\coker}{coker}
\DeclareMathOperator{\Newt}{Newt}
\DeclareMathOperator{\Aut}{Aut}

\DeclareMathOperator{\Spec}{Spec}
\DeclareMathOperator{\Proj}{Proj}
\DeclareMathOperator{\uProj}{\underline{Proj}}
\DeclareMathOperator{\Sing}{Sing}
\DeclareMathOperator{\Trop}{\mathcal{T}}
\DeclareMathOperator{\T}{\mathbf{T}}
\DeclareMathOperator{\C}{\mathbf{C}}
\DeclareMathOperator{\depth}{depth}
\DeclareMathOperator{\ob}{ob}
\DeclareMathOperator{\divisor}{div}
\DeclareMathOperator{\interior}{int}
\DeclareMathOperator{\Der}{Der}

\newcommand{\Def}[1]{\mathrm{Def}(#1)}
\newcommand{\catDef}[1]{\mathfrak{Def}(#1)}
\newcommand{\Deff}[2]{\mathrm{Def}(#1, #2)} 
\newcommand{\catDeff}[2]{\mathfrak{Def}(#1, #2)}

\newcommand{\Hom}{\mathrm{Hom}} 
\newcommand{\aut}{\mathfrak{aut}}
\newcommand{\Ext}{\mathrm{Ext}} 
\newcommand{\cHom}{\mathcal{H}om} 
\newcommand{\cExt}{\mathcal{E}xt} 

\newcommand{\conv}[1]{\left[ #1 \right]} 

\newcommand{\cayley}[1]{\mathrm{Cayley}(#1)} 

\newcommand\cA{\mathcal{A}}

\newcommand\cC{\mathcal{C}}
\newcommand\cD{\mathcal{D}}
\newcommand\cE{\mathcal{E}}

\newcommand\cH{\mathcal{H}}

\newcommand\cM{\mathcal{M}}

\newcommand\cO{\mathcal{O}}

\newcommand\cS{\mathcal{S}}
\newcommand\cT{\mathcal{T}}
\newcommand\cU{\mathcal{U}}

\newcommand\cX{\mathcal{X}}
\newcommand\cY{\mathcal{Y}}

\renewcommand\AA{\mathbb{A}}

\newcommand\CC{\mathbb{C}}
\newcommand\DD{\mathbb{D}}

\newcommand\GG{\mathbb{G}}
\newcommand\HH{\mathbb{H}}

\newcommand\LL{\mathbb{L}}

\newcommand\NN{\mathbb{N}}

\newcommand\PP{\mathbb{P}}
\newcommand\QQ{\mathbb{Q}}
\newcommand\RR{\mathbb{R}}

\newcommand\TT{\mathbb{T}}

\newcommand\ZZ{\mathbb{Z}}

\newcommand\rB{\mathrm{B}}

\newcommand\rH{\mathrm{H}}

\newcommand\rL{\mathrm{L}}

\newcommand\rR{\mathrm{R}}

\newcommand\rT{\mathrm{T}}

\newcommand\rmd{\mathrm{d}}

\newcommand{\frakB}{\mathfrak{B}}

\newcommand{\frakD}{\mathfrak{D}}

\newcommand{\frakg}{\mathfrak{g}}

\DeclareMathOperator{\Cone}{Cone}

\title{Smoothing Gorenstein toric Fano \mbox{$3$-folds}}

\author{Alessio Corti}
\address{Department of Mathematics, Imperial College London, 180 Queen's Gate, London, SW7 2AZ, UK}
\email{a.corti@imperial.ac.uk}

\author{Paul Hacking}
\address{Department of Mathematics and Statistics, Lederle Graduate, Research Tower, University of Massachusetts, Amherst, MA 01003-9305, US}
\email{hacking@math.umass.edu}

\author{Andrea Petracci}
\address{Dipartimento di Matematica, 
	Universit\`a di Bologna,
	Piazza di Porta San Donato 5,
	Bologna,
	40126,
	Italy}
\email{a.petracci@unibo.it}

\begin{document}

\renewcommand{\thefootnote}{\fnsymbol{footnote}}

\begin{abstract}
  We introduce \emph{admissible Minkowski decomposition
    data (amd)} for a $3$-dimensional reflexive polytope
  $P$. This notion is defined purely in terms of the combinatorics of
  $P$. Denoting by $X_P$ the Gorenstein toric Fano~\mbox{$3$-fold}
  whose fan is the spanning fan (a.k.a. face fan) of $P$, our first
 result states that amd for $P$ determine a smoothing of
  $X_P$. Our second result amounts to an effective recipe for computing
  the Betti numbers of the smoothing.

In the companion paper~\cite{CdS} we study by computer millions of
amd for the $4,319$ $3$-dimensional
reflexive polytopes, and in particular recover the $98$ families of Fano
$\mbox{3-folds}$ whose anticanonical line bundle is very ample. 
\end{abstract}

\maketitle

\setcounter{tocdepth}{2}
\tableofcontents

\section{Introduction}
\label{sec:introduction}

\subsection{Summary of the paper and short discussion of its context}
\label{sec:summary-paper-short}

Our first result is Theorem~\ref{thm:1}: a sufficient
condition for smoothing a Gorenstein toric Fano \mbox{$3$-fold}
starting from \emph{admissible Minkowski
  decomposition data (amd)}, see
Definition~\ref{dfn:decomposition_data}. Note that, because ampleness is
  an open condition, the smoothing is a Fano
  \mbox{$3$-fold}. Our second result is
Theorem~\ref{thm:topology} --- and the more precise
Theorem~\ref{thm:topologyII} --- implying an effective recipe for computing the Betti
numbers of the smoothing from the amd.


Gorenstein toric Fano \mbox{$3$-folds} are in correspondence with
$3$-dimensional reflexive polytopes, where the polytope $P$
corresponds to the toric variety $X_P$ whose fan is the \emph{spanning
  fan}, also known as the \emph{face fan}, of $P$. Our main
result states that a smoothing of $X_P$ can be constructed from amd
for $P$.


These results are not optimal --- see Remark~\ref{rem:limitations} ---
but they are applicable to a very large, though of course finite, set
of examples.
Our Theorems~\ref{thm:topology} and~\ref{thm:topologyII} allow to compute
the Betti numbers of the smoothing from the amd. In~\cite{CdS}
the technology is applied to construct smoothings of Gorenstein toric
Fano \mbox{$3$-folds} and compute their Betti numbers from millions of
amd.

\smallskip

Our results go some way towards explaining the facts on mirror symmetry for Fano \mbox{$3$-folds}
first observed ``experimentally'' in \cite{quantum_periods_fano_3folds, MR3007265}.
Indeed, starting from amd for a reflexive
polytope $P$, in this paper we construct:
\begin{enumerate}[(A)]
\item A smoothing $X$ of the toric Fano \mbox{$3$-fold} $X_P$. The
  companion paper \cite{CdS}, using Theorems~\ref{thm:topology}
  and~\ref{thm:topologyII}, locates $X$ in the Mori--Mukai
  classification;
\end{enumerate}
On the other hand, as explained in~\cite{MR3007265}, amd
for $P$ give:
\begin{enumerate}[(A)]
\item[(B)] A Laurent polynomial $w$ with Newton polytope $P$, called a
  \emph{Minkowski polynomial}. Paper~\cite{MR3007265} computes the
  classical period of $w$.
\end{enumerate}
Paper~\cite{quantum_periods_fano_3folds} computes the regularized
quantum periods of all the Fano \mbox{$3$-folds} in the Mori--Mukai
list, and it can be seen from the outcome of the computation that the
classical period of $w$ equals the regularized quantum period of the
smoothing.

\smallskip

In other words, we now have constructions for both the A and the B
sides of mirror symmetry starting from the same combinatorial
data.\footnote{We are oversimplifying. The Minkowski polynomial only
  depends on the choice, for all facets $F\leq P$, of (a) an
  admissible Minkowski decomposition of $F$. However, according to
  Definition~\ref{dfn:decomposition_data}, amd for $P$ involve a
  further choice, for all $F$, of (b) a dual tropical arrangement
  satisfying the matching condition (c) on dull edges. Now fix (a) and
  assume that we can also further choose (b) such that (c) holds: we
  don't know how to prove from first principles that these additional
  choices lead to deformation equivalent smoothings. (But it is not
  very difficult to show that the smoothings have the same Betti
  numbers.)}  However, we are still missing a conceptual explanation
of why the two periods are equal. This would follow if we had the
technology for computing the quantum cohomology of the smoothing from
the toric degeneration and the amd. The technology does not yet exist,
but the paper~\cite{MR4498821} proves some encouraging results in this
direction.

\smallskip

This paper is a first step in a program to construct deformations of
toric Fano varieties systematically, where the ultimate goal is
to prove a general form of the Fano/Landau--Ginzburg correspondence,
with application to the classification of $\QQ$-Fano \mbox{$3$-folds} ---
see~\cite{MR4340449} --- and Gorenstein terminal Fano
\mbox{$4$-folds}.

\begin{rem}
  \label{rem:Prince}
  The work of Thomas Prince \cite{prince_lagrangian_fano_3folds,
    prince_cracked_ci, prince_cracked_Fano_3folds} presents a
  different perspective on the questions studied here.
\end{rem}

\subsection{Outline of the Introduction}

In \S~\ref{sec:main-theorem} and \S~\ref{sec:comp-hodge-numb} we state
our results precisely; we also define the notions and explain the 
constructions that enter the statements.

In \S~\ref{sec:some-remarks-proof} we explain some key ideas of the
proof, in \S~\ref{sec:summary-paper} we give an outline of the paper;
\S~\ref{sec:notation} collects notation in use throughout.

\subsection{Smoothing Gorenstein toric Fano \mbox{$3$-folds}}
\label{sec:main-theorem}


\begin{dfn} \label{dfn:A_triangles}
	An \emph{$\text{A}_{-1}$-triangle} is a
	lattice segment of lattice length $1$.
	If $n$ is a non-negative integer, an
	\emph{$\text{A}_n$-triangle} is a lattice polygon
	$\ZZ^2 \rtimes \mathrm{GL}_2(\ZZ)$-equivalent to the triangle
	$\conv{(0,0), (0,1), (n+1,1)}$.  An \emph{$\text{A}$-triangle} is an
	$\text{A}_n$-triangle for some integer $n \geq -1$, see
	Figure~\ref{fig:A_triangles}.

	\begin{figure}[h]
          \begin{tikzpicture}[point/.style={circle, inner sep=0pt, minimum
    size=3pt, fill=black}]
  \draw[thick] (0,0) -- (0,-1) (1,0) -- (2,0) -- (1,-1) -- cycle
  (3,0) -- (5,0) -- (3,-1) -- cycle (6,0) -- (9,0) -- (6,-1) -- cycle;
  \node at (0,0) [point] {} ;  \node at (0,-1) [point] {} ;  \node at (1,0)
  [point] {} ;  \node at (2,0) [point] {} ;  \node at (1,-1) [point] {} ;
  \node at (3,0) [point] {} ; \node at (4,0) [point] {} ; \node at (5,0) [point] {} ;  \node at (3,-1)
  [point] {} ;  \node at (6,0) [point] {} ; \node at (7,0) [point] {} ; \node at (8,0) [point] {} ; \node at (9,0) [point] {} ;  \node at (6,-1) [point] {} ;
\end{tikzpicture}
		\caption{An \texorpdfstring{$\text{A}_{-1}$}{A1}-triangle, an \texorpdfstring{$\text{A}_0$}{A0}-triangle, an \texorpdfstring{$\text{A}_1$}{A1}-triangle
			and an \texorpdfstring{$\text{A}_2$}{A2}-triangle.}
		\label{fig:A_triangles}
	\end{figure}
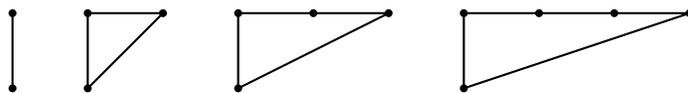
\end{dfn}

\begin{dfn}
  \label{dfn:Minkowski_Dec}
  The \emph{Minkowski sum} of lattice polytopes $F_1, \dots, F_r$ in a
  lattice $L$ is the lattice polytope\footnote{We adopt the convention of \textbf{never}
      putting punctuation at the end of displayed formulas. It is
      impossible to give rules, there are just too many cases. Our
      convention has the double advantage that it is simple and easy
      to use consistently.}    
\[
F_1 + \cdots + F_r = \{ v_1 + \cdots + v_r \mid v_1 \in F_1, \dots, v_r \in F_r \}
\]
A \emph{Minkowski decomposition} of a lattice polytope $F$
is a tuple of lattice polytopes whose Minkowski sum is $F$.
These lattice polytopes are called \emph{Minkowski summands}
of $F$.
\end{dfn}

\begin{dfn}
  \label{dfn:polygon_amd}
	Let $L$ be a lattice of rank 2 and let $F$ be a lattice polygon in
	$L$. 
	An \emph{admissible Minkowski decomposition} of $F$ is a
	Minkowski decomposition where each Minkowski summand is an $\text{A}$-triangle.  
	Minkowski decompositions that differ by reordering and translating
	the summands are considered to be the same.
\end{dfn}

\begin{exa} \label{exa:Minkowski_dec_hexagon} The hexagon
	$ F = \conv{ (1,0),(1,1),(0,1),(-1,0),(-1,-1),(0,-1) } $ in the lattice $\ZZ^2$
	has two admissible Minkowski decompositions, see Figure~\ref{fig:Minkow}: one
	into three $\text{A}_{-1}$-triangles
	\begin{equation} \label{eq:MD_hexagon_3_segments}
	F = \conv{(0,0), (1,0)} 
	+ \conv{(0,0), (0,1)}
	+ \conv{(0,0), (-1,-1)} 
	\end{equation}
	and one into two $\text{A}_0$-triangles
	\begin{equation} \label{eq:MD_hexagon_2_triangles}
	F = \conv{(0,0), (-1,0), (-1,-1)  }
	+ \conv{ (0,0), (1,0), (1,1)} 
	\end{equation} 
\end{exa}

	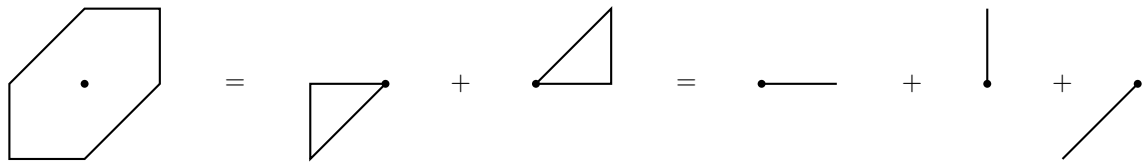
\begin{figure}[h]
          \centering
          \begin{tikzpicture}[point/.style={circle, inner sep=0pt, minimum size=3pt, fill=black}]
  \draw[thick]   (1,0) -- (1,1) -- (0,1) -- (-1,0) -- (-1,-1) --
  (0,-1) -- cycle ;
  \node at (0,0) [point] {} ; \node at  (2,0) {$=$} ;
  \draw[thick]  (4,0) -- (3,0) -- (3,-1) -- cycle ;
  \node at (4,0) [point] {} ; \node at (6,0) [point] {} ; \node at  (5,0) {$+$} ;
  \draw[thick]  (6,0) -- (7,0) -- (7,1) -- cycle ;
  \node at (8,0) {$=$} ; \node at (9,0) [point] {} ; \node at (12,0)
  [point] {} ; \node at (14,0) [point] {} ;
  \draw[thick] (9,0) -- (10,0) ;
  \node at (11,0) {$+$} ;
  \draw[thick] (12,0) -- (12,1) ;
  \node at (13,0) {$+$} ;
  \draw[thick] (13,-1) -- (14,0);
\end{tikzpicture}
	\caption{The two admissible Minkowski decompositions of the hexagon in Example~\ref{exa:Minkowski_dec_hexagon}.}
	\label{fig:Minkow}
\end{figure}

\begin{exa}
  \label{exa:no_amd}
	Some polygons have no admissible Minkowski decompositions,
        e.g.\ the triangle \newline
	$\conv{(-1,-1),(2,-1),(-1,1)}$ and the triangle $\conv{(-1,-1),(1,0),(0,1)}$ in
	$\ZZ^2$, see Figure~\ref{fig:no_amd}.
\end{exa}

\begin{figure}[h]
          \centering
          \begin{tikzpicture}[point/.style={circle, inner sep=0pt, minimum size=3pt, fill=black}]
            \draw[thick]   (-1,-1) -- (2,-1) -- (-1,1) -- cycle ;
            \node at (0,0) [point] {} ;\node at (-1,-1) [point] {} ; \node
            at (2,-1) [point] {} ;\node at (-1,1) [point] {} ; \node at
            (0,-1) [point] {} ; \node at (1,-1) [point] {} ; \node at
            (-1,0) [point] {} ;
            \draw[thick]   (5,-1) -- (7,0) -- (6,1) -- cycle ;
                \node at (5,-1) [point] {} ;\node at (7,0) [point] {} ; \node
            at (6,0) [point] {} ; \node at (6,0) [point] {} ; \node
            at (6,1) [point] {} ;
          \end{tikzpicture}
	\caption{These polygons have no admissible Minkowski decomposition.}
	\label{fig:no_amd}
\end{figure}
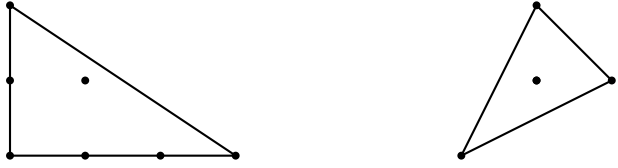

\begin{dfn}
        \label{dfn:dual_tropical_arrangement}
        Fix a plane polygon $F$ and admissible Minkowski decomposition $m=(F=\sum F_j)$. To each
        Minkowski summand 
        $F_j$ is attached a \emph{dual tropical curve} $\Gamma_j$, see Theorem~\ref{thm:cayley_trick} and its proof.
        A \emph{dual tropical arrangement subordinated to $m$} is a
        generic plane arrangement of the $\Gamma_j$. 

        A dual tropical arrangement induces, for every edge
        $e\subset F$, a partition of the set $L_e$ of unit segments of $e$.
        The parts of the partition are indexed by the $\Gamma_j$ so
        that the part corresponding to
        $\Gamma_j$ is the set of unit segments that have a leg of
        $\Gamma_j$ going through them. We call this partition the
        \emph{induced partition}.

        A dual tropical arrangement also induces a lattice polyhedral
        subdivision of $F$, which we call the \emph{induced
          subdivision}.

        These concepts are illustrated in
        Figures~\ref{fig:TA_1},~\ref{fig:TA_2} and~\ref{fig:TA_3}.
\end{dfn}

\begin{rem}
  The tropical curve $\Gamma=\cup \Gamma_j$ has trivalent and
  $4$-valent vertices only.

  The induced subdivision is a \emph{coherent fine mixed subdivision},
  see Theorem~\ref{thm:cayley_trick}, but we don't need this language
  nor the fact now.

  Figures~\ref{fig:TA_1},~\ref{fig:TA_2} and~\ref{fig:TA_3} show examples
  of dual tropical arrangements and induced subdivisions.
  
  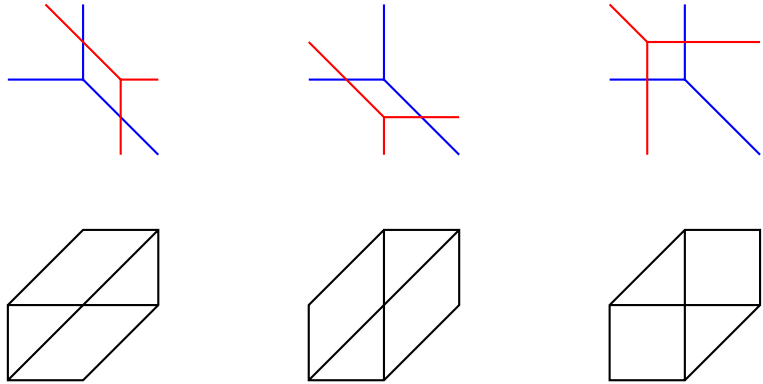
\begin{figure}[ht]
  \centering
\begin{tikzpicture}
  \draw[thick, blue] (0,0) -- (1,-1) (0,0) -- (-1,0) (0,0) -- (0,1)
  (4,0) -- (5,-1) (4,0) -- (3,0) (4,0) -- (4,1)
  (8,0) -- (9,-1) (8,0) -- (7,0) (8,0) -- (8,1);
   \draw[thick, red] (0.5,0) -- (-0.5,1) (0.5,0) -- (1,0) (0.5,0) --
   (0.5,-1)
   (4,-0.5) -- (3,0.5) (4,-0.5) -- (5,-0.5) (4,-0.5) --
   (4,-1)
   (7.5,0.5) -- (7,1) (7.5,0.5) -- (9,0.5) (7.5,0.5) -- (7.5,-1);
   \draw[thick]   (1,-3) -- (1,-2) -- (0,-2) -- (-1,-3) -- (-1,-4) --
  (0,-4) -- cycle (0,-3) -- (1,-3) (0,-3) -- (1,-2) (0,-3) -- (-1,-3) (0,-3)
  -- (-1,-4)
  (5,-3) -- (5,-2) -- (4,-2) -- (3,-3) -- (3,-4) --
  (4,-4) -- cycle (4,-3) -- (4,-2) (4,-3) -- (5,-2) (4,-3) -- (4,-4) (4,-3)
  -- (3,-4)
  (9,-3) -- (9,-2) -- (8,-2) -- (7,-3) -- (7,-4) --
  (8,-4) -- cycle (8,-3) -- (8,-2) (8,-3) -- (9,-3) (8,-3) -- (8,-4) (8,-3)
  -- (7,-3);
\end{tikzpicture}
\caption{Dual tropical arrangements and induced subdivisions for the
    admissible Minkowski decomposition of the hexagon into two $\text{A}_0$-triangles.}
      \label{fig:TA_1}
\end{figure}

\begin{figure}[ht]
  \centering
\begin{tikzpicture}
  \draw[thick, blue] (-0.5,0) -- (-0.5,2) ;
  \draw[thick, red] (-1,0.5) -- (1,0.5) ;
  \draw[thick, green] (-1,2) -- (1,0) ;
  \draw[thick, blue] (3.5,0) -- (3.5,2) ;
  \draw[thick, red] (2,1.5) -- (4,1.5) ;
  \draw[thick, green] (2,2) -- (4,0) ;
   \draw[thick]   (1,-2) -- (1,-1) -- (0,-1) -- (-1,-2) -- (-1,-3) --
   (0,-3) -- cycle ;
   \draw[thick] (0,-2) -- (-1,-2) (0,-2) -- (0,-3) (0,-2) -- (1,-1) ;
     \draw[thick]   (4,-2) -- (4,-1) -- (3,-1) -- (2,-2) -- (2,-3) --
   (3,-3) -- cycle ;
   \draw[thick] (3,-2) -- (4,-2) (3,-2) -- (3,-1) (3,-2) -- (2,-3) ;
\end{tikzpicture}
  \caption{Dual tropical arrangements and induced subdivisions for the
    admissible Minkowski decomposition of the hexagon into three $\text{A}_{-1}$-triangles.}
    \label{fig:TA_2}
\end{figure}
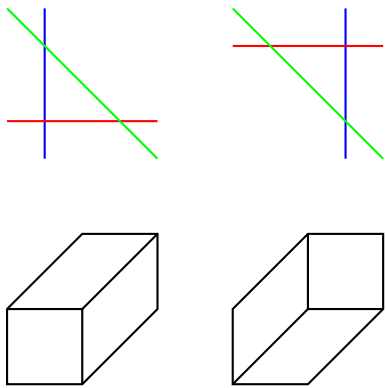

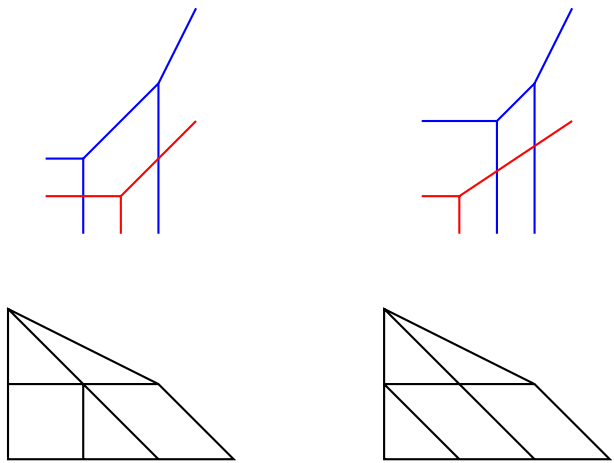
\begin{figure}[ht]
  \centering
\begin{tikzpicture}
  \draw[thick, blue] (0,0) -- (0,-1) (0,0) -- (-0.5,0) (0,0) -- (1,1)
  (1,1) -- (1,-1) (1,1)-- (1.5,2);
  \draw[thick, red] (0.5,-0.5) -- (1.5,0.5) (0.5,-0.5) -- (-0.5,-0.5) (0.5,-0.5) -- (0.5,-1);
  \draw[thick] (-1,-4) -- (2,-4) -- (1,-3) -- (-1,-2) -- cycle
  (0,-3) -- (0,-4) (-1,-2) -- (1,-4) (-1,-3) -- (1,-3);
  \draw[thick, blue] (5.5,0.5) -- (5.5,-1) (5.5,0.5) -- (4.5,0.5) (5.5,0.5) -- (6,1)
  (6,1) -- (6,-1) (6,1)-- (6.5,2);
  \draw[thick, red] (5,-0.5) -- (6.5,0.5) (5,-0.5) -- (5,-1) (5,-0.5)
  -- (4.5,-0.5);
    \draw[thick] (4,-4) -- (7,-4) -- (6,-3) -- (4,-2) -- cycle
  (4,-3) -- (5,-4) (4,-2) -- (6,-4) (4,-3) -- (6,-3);
\end{tikzpicture}
\caption{More dual tropical arrangements and induced subdivisions.}
  \label{fig:TA_3}
\end{figure}
\end{rem}
  
\begin{dfn}
  \label{dfn:sharp_edges}
  Let $P$ be a $3$-dimensional reflexive polytope, and $e\leq P$ an
  edge of $P$.

  The \emph{length} of $e$, denoted by $\ell_e$, is the
  integral length of $e$. The \emph{colength} of $e$, denoted by
  $k_e$, is the length of the dual edge $e^\star\leq P^\star$ of the polar
  polytope $P^\star$.

  The edge $e$ is \emph{dull} if it has colength $k_e=1$. 
\end{dfn}

\begin{dfn}
  \label{dfn:decomposition_data}
 Let $P$ be a $3$-dimensional reflexive polytope. An \emph{admissible
 Minkowski decomposition data (amd) for $P$} consists of the
following:
\begin{enumerate}[(a)]
\item A choice, for all facets $F\leq P$, of an admissible Minkowski decomposition $F=\sum F_j$;
\item A choice, for all facets $F\leq P$, of a dual tropical
  arrangement subordinated to the admissible Minkowski decomposition
  in~(a).
\end{enumerate}
For all dull edges $e\subset P$, the choices above are required to
satisfy the matching condition at $e$ that we describe next.
Let $F,G\leq P$ be the facets incident along $e$. The set $L_e$ of unit segments of $e$
has two induced partitions. The first of these is induced by the
Minkowski decomposition $F=\sum F_i$, and we denote by $L^F_i$ the
part corresponding to $F_i$. The second  is induced by the
Minkowski decomposition $G=\sum G_j$, and we denote by $L^G_j$ the
part corresponding to $G_j$.

The \emph{matching condition at $e$} is the following statement:
\begin{enumerate}[(a)]
  \setcounter{enumi}{2}
\item For all $i,j$, $|L^F_i\cap L^G_j|\leq 1$. 
\end{enumerate}
\end{dfn}

\begin{nota}
  \label{dfn:partial_res}
  Let $P$ be a $3$-dimensional reflexive polytope endowed with
  choices, for all facets, of (a) admissible Minkowski decompositions
  and (b) dual tropical arrangements.
  
  Let $X$ be the Gorenstein toric Fano \mbox{$3$-fold} whose fan is the
  spanning fan of $P$.

  For all facets $F$ of $P$, the dual tropical arrangement induces a
  subdivision of the cone $\sigma_F=\langle F \rangle_+$ and hence a
  toric partial resolution of the corresponding Zariski open subset
  $X_{\sigma_F}\subset X$.

  These subdivisions combine to give a refinement of the fan of $X$
  and hence the corresponding local toric partial resolutions glue to a
  global toric partial resolution of $X$ that we call the
  \emph{induced partial resolution} and denote by
  $\pi \colon Y \to X$.

  Note that, in general, $\pi\colon Y \to X$ is not a projective morphism.
\end{nota}

\begin{dfn}
    \label{dfn:qODP} A \emph{quasi-ordinary double point},
    abbreviated \emph{qODP}, is a Gorenstein toric \mbox{$3$-fold} singularity
    locally analytically --- or, depending on the context, \'etale
    locally --- isomorphic to the germ at the origin of
    \begin{equation*}
    \left( x_1 x_2 - x_3 x_4 = 0 \right) \subseteq \frac{1}{a}(1,-1,b,-b)_{x_1, x_2, x_3, x_4}
    \end{equation*}
    for some positive integers $a,b$ such that
    $\gcd(a,b)=1$.
\end{dfn}

\begin{lem}
  \label{lem:1}
    A Gorenstein toric \mbox{$3$-fold} singularity is a qODP
    if and only if it corresponds to the cone over a parallelogram with edges of unit length.

    The qODP is an ordinary double point (ODP) --- i.e., $a=1$ ---
    if and only if the parallelogram does not have interior points. \qed 
\end{lem}

\begin{lem}
  \label{lem:qODP}
  Let $P$ be a $3$-dimensional reflexive polytope endowed with amd,
  $X$ the corresponding Gorenstein toric Fano \mbox{$3$-fold}, and
  $\pi \colon Y \to X$ the induced partial resolution. Then
  \begin{enumerate}[(1)]
  \item $Y$ is Gorenstein and the morphism $\pi\colon Y \to X$ is
    crepant, that is $K_Y=\pi^\star K_X$;
  \item $Y$ has qODPs;
  \item The pair $(Y,E)$ is isomorphic, Zariski
    locally at all closed points $y\in Y$, either to a normal crossing
    pair, or a qODP with $E= \{ x_1 x_2 = x_3 x_4 = 0 \}$. 
  \end{enumerate}
\end{lem}

\begin{proof}
    The first statement is obvious from the toric geometry: the rays
    of the subdivision are by construction generated by vectors that
    lie on $F$.
    
    It is clear from the tropical curve picture that the cones of the
    fan of $Y$ are either basic simplices, or cones over plane
    parallelograms with edges of lattice length $=1$. 
  \end{proof}

  \begin{rem}
    \label{rem:why_stacks?}
    A toric Gorenstein affine \mbox{$3$-fold} $X_\sigma$ has a qODP
    singularity at the origin if and only if $\sigma$ is the cone over
    a plane integral parallelogram with edges of unit integral length,
    placed at height $1$.
    
  The computations in~\cite[(7.3)~Proposition]{altmann_obstruction}
  show that the natural obstruction space of qODP that are not ODP is
  nontrivial. 

  The deformation theory of the stack $[\text{ODP}/\mu_a]$ is better behaved.
\end{rem}
  
\begin{dfn}
  \label{dfn:qODPpair}
  A \emph{qODP stack} is a Deligne--Mumford (DM) stack $Y$ which is
  locally analytically --- or, depending on the context, \'etale
  locally --- isomorphic to either a smooth scheme or the
  DM stack:
  \begin{equation*}
    [\left( x_1 x_2 - x_3 x_4 = 0 \right)/\mu_a]
  \end{equation*}
  where $x_1,\dots, x_4$ are coordinates on $\CC^4$ and $\mu_a$ acts
  with weights $(1,-1,b,-b)$ for some positive integers $a,b$ such
  that $\gcd(a,b)=1$.

  A \emph{qODP stack pair} $(Y,E)$ is a pair of a \mbox{$3$-fold} DM
  stack $Y$ together with an effective Cartier divisor $E \subset Y$
  which is locally analytically --- or, depending on the context,
  \'etale locally --- isomorphic to: either a
  normal crossing pair, or a qODP stack with
  $E= \{ x_1 x_2 = x_3 x_4 = 0 \}$.
\end{dfn}

\begin{thm}
  \label{thm:1}
  Let $P$ be a $3$-dimensional reflexive polytope endowed with amd,
  $X$ the corresponding toric Fano \mbox{$3$-fold}, and
  $\pi \colon Y \to X$ the induced partial resolution. Then:
  \begin{enumerate}
  \item The pair $(Y,E)$, regarded as a qODP stack in the obvious way,
    is unobstructed and smoothable;
  \item If $\cY_t$ is a general smoothing of $Y$, then $\cY_t$ is a weak Fano
\mbox{$3$-fold} and, denoting by
    \[
R(\cY_t,-K_{\cY_t}) = \bigoplus_{n=0}^\infty \rH^0(\cY_t,-nK_{\cY_t})
    \]
the anticanonical ring of $\cY_t$, the anticanonical morphism
\[
\pi_t \colon \cY_t \to \cX_t = \Proj R (\cY_t, -K_{\cY_t})
\]
  contracts a finite number of disjoint nonsingular rational curves with normal
  bundle $\cO(-1)\oplus \cO(-1)$;
\item $\cX_t$ is a Fano \mbox{$3$-fold} with ODPs as singularities
  and it is a deformation of $X$.
  \end{enumerate}
\end{thm}

\begin{rem}
 \label{rem:1}
 By Namikawa~\cite[Theorem~11]{namikawa_smoothing_fano_3folds} $\cX_t$
 is unobstructed and smoothable; hence $X$ itself is smoothable. In the
 display diagram $\cX_\eta$ is a generic smoothing of $\cX_t$:
\begin{equation*}
\xymatrix{
Y \ar[d]^\pi & \cY_t \ar[d]^{\pi_t} \ar@{~>}[l]& \\
X & \cX_t \ar@{~>}[l] & \cX_\eta \ar@{~>}[l]}
\end{equation*}
\end{rem}

\begin{rem}
  \label{rem:limitations}
  \begin{enumerate}[(1)]
  \item   See~\cite{petracci_almost_flat}
  for some examples and counterexamples when the matching condition is
  not satisfied.
  \item Theorem~\ref{thm:1} is not optimal. For example consider the
  polytope $P$ with vertices the columns of the matrix
  \[
    \begin{pmatrix}
    -1 & 2 & -1 & 0 \\
    -1 & -1& 1 & 0 \\
    -1 & -1 & -1 & 1  
    \end{pmatrix}   
  \]
  $P$ is the convex hull of the polygon $F$ of Fig.~\ref{fig:no_amd}, left,
  placed at height $-1$, and the vector $(0,0,1)$. Because $F$ has no
  admissible Minkowski decomposition, there is no amd for $P$. However, one
  can see that the Fano $X_P$ is smoothable. In fact, the local moduli
  space of $X_P$ has two components, one smoothing to
  $X_{1,1}\subset \PP^2\times \PP^2$ and the other to
  $\PP^1\times \PP^1 \times \PP^1$. See~\cite{MR4381899} for a
  conjectural characterization of smoothing components of Gorenstein
  toric affine \mbox{$3$-folds}.
\item A conjectural characterization of all deformations of
  $\QQ$-Gorenstein toric Fano \mbox{$3$-folds} to $\QQ$-factorial \mbox{$3$-folds}
  with terminal singularities can be found in~\cite{MR4340449}.
  \end{enumerate}
\end{rem}

\subsection{Understanding the topology of the smoothing}
\label{sec:comp-hodge-numb}

Consider a general smoothing $\cY_t$ of $Y$ as provided
by Theorem~\ref{thm:1}. In order to determine the topology of
$\cX_\eta$, it is necessary to study the morphism $\pi_t \colon \cY_t \to \cX_t$.  

\begin{thm}
  \label{thm:topology}
  Let $P$ be a $3$-dimensional reflexive polytope endowed with amd,
  $X$ the corresponding toric Fano \mbox{$3$-fold}, and
  $\pi \colon Y \to X$ the induced toric partial resolution.

  For all edges $e \leq P$, define $n_e\in \NN$ as follows. Let
  $F,G\leq P$ be the facets incident at $e$, $F=\sum F_i$, $G=\sum
  G_j$ their admissible Minkowski decompositions, and $L^F_i$, $L^G_j$
  the corresponding parts of the induced partitions of $L^e$. Define
\[
n_e=k_e \binom{\ell_e}{2} -\sum_i \binom{|L^F_i|}{2}
-\sum_j \binom{|L^G_j|}{2}
 \] 

 Then $\cX_t$ has precisely 
  \[n=\sum_{\ell_e\geq 2} n_e\]
  ODPs. Correspondingly, $\cY_t$ has precisely $n$
  disjoint nonsingular rational curves with normal bundle
  $\cO(-1)\oplus \cO(-1)$.
\end{thm}

Theorem~\ref{thm:topologyII} is a more precise version of the
statement just made.

\subsection{Some ideas of the proof}
\label{sec:some-remarks-proof}

We want to study deformations of $X$. We prefer to work with
deformations of the pair $(X,D)$. These are controlled by
$\Omega_X (\log D)$, see
Definition~\ref{dfn:logarithmic_differential}, and its dual sheaf
$T_X (-\log D)$. A nice thing about pairs is that
$\Aut^0 (X,D)=\TT$ is the torus, and that the sheaf of infinitesimal
automorphisms
$T_X (-\log D)=\cO_X\otimes N$ is the trivial sheaf and hence has
vanishing higher cohomology. The deformation theory of the pair $(X,D)$
is complicated and poorly understood. On the other hand, it is easy to
see that the pair $(Y,E)$, regarded as a qODP stack in the obvious
way, is unobstructed and smoothable. It is also easy to see by
well-understood means that all deformations $(\cY_t,\cE_t), \; t\in T$ of
$(Y,E)$ can be blown down to deformations $(\cX_t,\cD_t),\; t\in T$ of
$(X,D)$. The plan is to work only with the deformations of $(X,D)$
that arise in this way from deformations of $(Y,E)$. To
understand these deformations we need to study the contraction morphism
$\pi_t \colon \cY_t \to \cX_t$. Theorem~\ref{thm:1} follows easily from
Lemma~\ref{lem:fibre_dimensions}, which establishes some simple facts
about the geometry of the contraction morphism. In that lemma, we
prove things about $\pi_t\colon \cY_t \to \cX_t$ for a general $t\in
\cM_{Y,E}$, the miniversal deformation of $(Y,E)$. The key
difficulty is that we do not have an explicit description of
$(\cY_t,\cE_t)$. The result is derived from studying many particular
deformations that one can write down explicitly. These are the
$\TT$-equivariant deformations, also known as homogeneous
deformations, defined over eigenspaces of the torus action on
$\T^1_{Y,E}$.

The proof of Theorem~\ref{thm:topology} is harder; an informal
discussion can be found at the beginning of \S~\ref{sec:flops}. A
recurrent theme of the proof is the correspondence between: (i)
coherent mixed subdivisions of a polytope $F$ subordinated to a
Minkowski decomposition, (ii) regular subdivisions of the Cayley
polytope, and (iii) arrangements of dual tropical hypersurfaces. The
correspondence allows us to study the geometry of the very
high-dimensional and complicated deformation spaces associated with
the Cayley polytope by drawing simple diagrams of tropical curves in
the plane.

\subsection{Summary of the paper}
\label{sec:summary-paper}

In Section~\ref{sec:homog-deform-local} we collect some facts about
toric (also known as homogeneous) deformations. Most of these are
known, see for example the paper of Petracci~\cite{petracci_mavlyutov}
and the references therein, but the statements in the literature are
not always written in a way that they can be applied off-the-shelf to
our situation.

Section~\ref{sec:deform-theory-pair} proves some explicit results
concerning the deformation theory of the pair $(Y,E)$: this is
straightforward combinatorics.

Section~\ref{sec:homog-deform-ii} uses the deformation theory of the
pair $(Y,E)$ to construct a particular homogeneous deformation and
establish some of its properties, thus completing the discussion of
homogeneous deformations started in Section~\ref{sec:homog-deform-local}.

In Section~\ref{sec:proof-theorem} we prove Theorem~\ref{thm:1}, and in
Section~\ref{sec:proof-theorem-2} we prove Theorem~\ref{thm:topology},
after some preliminaries discussed in
Section~\ref{sec:flops}.

In the Appendices we synthesise mostly known facts for which there may
not be a convenient reference to a statement in the literature
that can be applied off-the-shelf to our situation.

In Appendix~\ref{sec:Cayley} we summarize some facts on the
combinatorics of lattice polyhedra, specifically the correspondence
between: (i) coherent mixed subdivisions of a polytope $F$
subordinated to a Minkowski decomposition, (ii) regular subdivisions
of the Cayley polytope, and (iii) arrangements of dual tropical
hypersurfaces. These results are well-known lore in the lattice
polyhedra community but we could not find in the literature a complete
statement that would apply off-the-shelf to our situation.

In Appendix~\ref{sec:appendix} we develop the general theory of
deformations of pairs $(X,D)$ of a scheme and effective Cartier
divisor. This is basically a long exercise in deformation theory. The
results must be well-known to many algebraic geometers but,
surprisingly given the vast amounts of existing literature on
deformation theory, we could not find them written up anywhere.

The purpose of Appendix~\ref{sec:exampl-miniv-deform} is to write down
explicitly some miniversal deformation families of pairs $(Y,E)$ that are
needed in the proof of Theorem~\ref{thm:local_invariance}, which in
turn is the main ingredient in the proof of
Theorem~\ref{thm:topology}. This would be an exercise, were it
not for the fact that several of these families are
infinite-dimensional. We address the infinite-dimensionality carefully
by working in the language of ind-schemes (fortunately, we only need
the definition and none of the theory of ind-schemes).

In Appendix~\ref{sec:equiv-g-struct} we summarize Rim's theory of
(formal) equivariant $G$-structures on versal deformations.

We refer the reader to the beginning of each section for a more
detailed discussion of its contents.

\subsection{Notation and conventions}
\label{sec:notation}

\begin{dfn}
  \label{dfn:toric_pair}
  A \emph{toric pair} is a pair $(X,D)$ of a toric variety $X$ and
  toric boundary divisor $D$.
\end{dfn}

The table and diagrams summarize notation in use throughout the paper. More
notation is spelled out in \S~\ref{sec:notation-1}.

\begin{longtable}{lp{0.85\textwidth}}
$X$ & Gorenstein toric Fano \mbox{$3$-fold} (except in
\S~\ref{sec:altm-deform-gorenst} and~\ref{sec:admiss-decomp-simult})\\
$X$ & Gorenstein toric affine \mbox{$3$-fold} (in \S~\ref{sec:altm-deform-gorenst} and~\ref{sec:admiss-decomp-simult})\\
$\pi \colon Y\to X$ & induced partial resolution of $X$ with qODPs \\
$D$ & toric boundary of $X$ \\
$E$ & toric boundary of $Y$ \\
$\Gamma$ & toric 1-skeleton of $X$ \\
$\Delta$ & toric 1-skeleton of $Y$ \\
$p\colon \Delta \to \Gamma$ & natural morphism \\
$\nu \colon \Delta^\prime\to \Delta$ & the partial normalisation of
                                       $\Delta$ defined in \S~\ref{sec:notation-1}\\
$p^\prime=p\circ \nu \colon \Delta^\prime \to \Gamma$ & natural
                                                        morphism \\
$f\colon (\cX,\cD) \to S$ & deformation of the pair $(X,D)$ over a base scheme $S$\\
$g\colon (\cY,\cE)\to S$ & deformation of the pair $(Y,E)$ over a base
scheme $S$
\end{longtable}

\begin{equation*}
\xymatrix{
\Delta^\prime \ar[dr]_{p^\prime}\ar[r]^\nu & \Delta \ar[d]^p \ar@{^(->}[r] & E \ar[d] \ar@{^(->}[r] & Y \ar[d]^\pi \\
& \Gamma \ar@{^(->}[r] & D \ar@{^(->}[r] & X
}
\end{equation*}

\[
    \xymatrix{ (Y,E)\ar[dd]\ar[dr]^\pi\ar@{^{(}->}[rrr] &   & & (\cY,\cE)\ar[dd]_(.3){g}\ar[dr]^\Pi & \\
                        &(X,D)\ar[dl] \ar@{^{(}-}[rr]&  & \ar@{->}[r]& (\cX,\cD)\ar[dl]^f\\
                        \{0\}  \ar@{^{(}->}[rrr]& & & S &}
\]

\subsection{Leitfaden}

\[
  \xymatrix{A \ar@{=>}[r]& 2\ar@{=>}[dr] & & 3\ar@{=>}[dl] \ar@{=>}[ddd]& B\ar@{=>}[l] \\
    & D \ar@{=>}[r]& 4\ar@{=>}[d] & & \\
    & &    5 \ar@{=>}[d] &    & \\
    & &    6 \ar@{=>}[d] & C \ar@{=>}[l]& \\
    & &    7 &    & }
\]

\subsection*{Acknowledgements} During most of the work AC was
supported by EPSRC Programme Grant EP/N03189X/1. PH was supported by
NSF grants DMS-1901970 and DMS-2200875.  PH and AC thank IHES for
support in Summer 2023 and 2024, when parts of this project were
completed.

\section{Homogeneous deformations I}
\label{sec:homog-deform-local}

This section is a synthesis of mostly known facts due to
Altmann~\cite{altmann_minkowski_sums,MR1798979},
Mavlyutov~\cite{mavlyutov}, Ilten~\cite{MR2958983},
Matsushita~\cite{matsushita_simultaneous}, and
Petracci~\cite{petracci_mavlyutov} (this list of references is not
attempting to be complete). 

\smallskip

The reader who wishes to understand quickly the idea of the proof of
Theorem~\ref{thm:1} is advised to read the statements of purpose
\S~\ref{sec:synopsis}, and then jump to
Section~\ref{sec:proof-theorem}.

\subsection{Statement of purpose}
\label{sec:synopsis}

Fix a $3$-dimensional reflexive polytope $P$. 

\begin{enumerate}[(1)]
\item Let $F\leq P$ be a facet. Fix a Minkowski decomposition
  $F=\sum_{j=1}^r F_j$. Consider the cone $\sigma =\langle F\rangle_+$
  and the corresponding affine toric pair $(X_F,D_F)$. In
  \S~\ref{sec:altm-deform-gorenst} we recall Altmann's construction of
  a deformation
\[
    \xymatrix{
      (X_F,D_F) \ar@{^(->}[r]\ar[d]& (\cX_F,\cD_F) \ar[d]^f \\
\{0\} \ar@{^(->}[r]& \AA^r}
  \]
  that we call the \emph{local Altmann deformation.}
\item In \S~\ref{sec:admiss-decomp-simult} we show that the
  local Altmann deformation can be simultaneously partially
  resolved. In other words, choose a dual tropical arrangement
  subordinated to the Minkowski decomposition of $F$, and
  let $\pi_F\colon (Y_F,E_F)\to (X_F,D_F)$
  be the induced partial resolution. There exists a
  diagram of deformations:
  \[
    \xymatrix{ (Y_F,E_F)\ar[dd]\ar[dr]^\pi\ar@{^{(}->}[rrr] &   & & (\cY_F,\cE_F)\ar[dd]_(.3){g}\ar[dr]^\Pi & \\
      &(X_F,D_F)\ar[dl] \ar@{^{(}-}[rr]&  & \ar@{->}[r]& (\cX_F,\cD_F)\ar[dl]^f\\
      \{0\} \ar@{^{(}->}[rrr]& & & \AA^r&}
  \]
\item In \S~\ref{sec:glob-altm-deform} we show that the local Altmann
  deformation can be globalized. In other words, let $(X_P,D_P)$ be
  the $3$-dimensional Fano toric pair corresponding to the polytope $P$.
  There is a deformation
  \[
    \xymatrix{
      (X_P,D_P) \ar@{^(->}[r]\ar[d]& (\cX,\cD) \ar[d]^f \\
\{0\} \ar@{^(->}[r]& \AA^r}
  \]
  that induces the local Altmann deformation of $(X_F,D_F)$. We call
  this the \emph{global Altmann deformation}.
\item  In \S~\ref{sec:simult-part-resol} we state
  that the analytification of the global Altmann deformation can be
  simultaneously partially resolved. (The statement is proved
  in Section~\ref{sec:homog-deform-ii}.) In other words, fix amd for
  the whole polytope $P$ that restrict to the given one on the facet
  $F$, and let $\pi\colon (Y,E)\to (X_P, D_P)$ be the induced partial resolution.
  
  Denote by $S^{\text{an}}$ the analytic germ of $0\in \AA^r$ and by
  $f^{\text{an}}\colon (\cX^{\text{an}},\cD^{\text{an}}) \to
  S^{\text{an}}$ the analytification of the global Altmann
  deformation. There exists a diagram of complex analytic
  deformations:
 \[
    \xymatrix{ (Y,E)\ar[dd]\ar[dr]^\pi\ar@{^{(}->}[rrr] &   & &
      \bigl(\cY^{\text{an}}, \cE^{\text{an}}\bigr)
      \ar[dd]_(.3){g^{\text{an}}}\ar[dr]^{\Pi^{\text{an}}} & \\
      &(X,D)\ar[dl] \ar@{^(-}[rr]&  & \ar@{->}[r]&
      \bigl(\cX^{\text{an}},\cD^{\text{an}}\bigr)
      \ar[dl]^{f^{\text{an}}} \\
      \{0\} \ar@{^{(}->}[rrr]& & &  S^{\text{an}} & }
  \]
 
  This last deformation will be a crucial ingredient of the proof of
  our main results.  (The existence of this deformation
  $g^{\text{an}} \colon (\cY^{\text{an}},\cE^{\text{an}})\to
  S^{\text{an}}$ is a natural and unsurprising fact.)
\end{enumerate}

\subsection{The local Altmann deformation and its properties}
\label{sec:altm-deform-gorenst}

\begin{nota}
  \label{nota:characters} We work with lattices $N\cong \ZZ^n$,
  $M=\Hom (N, \ZZ)$, and the torus $\TT=\Spec \CC[M]\cong \Gm^n$. When $m\in M$,
  we denote by $\chi^m\in \CC[M]$ the corresponding character of $\TT$.
\end{nota}

\begin{pro}
  \label{pro:Altmann_deformation}
  Let $\overline{N}$ be a lattice of rank $n-1$ and let $F\subset
  \overline{N}$ be a $(n-1)$-dimensional lattice polytope endowed with a
  Minkowski decomposition
\[
F = F_1 + \cdots + F_r
\]

Consider the Gorenstein cone
$\sigma = \langle F \times \{e_0\} \rangle_+ \subset N = \overline{N}
\oplus \ZZ e_0$ and the corresponding $n$-dimensional affine toric
pair $(X_F, D_F)$.

  Consider the commutative diagram:
\[
    \xymatrix{
      (X_F,D_F) \ar@{^(->}[r]\ar[d]& (\cX,\cD) \ar[d]^f \\
\{0\} \ar@{^(->}[r]& \AA^r}
  \]
where 
\begin{enumerate}[(i)]
\item Write
  $\widetilde{M}=\Hom(\widetilde{N},\ZZ)=\overline{M}\oplus \ZZ
  e_0^\star \oplus \ZZ e_1^\star \oplus \cdots \oplus \ZZ e_r^\star$.
$\cX = \Spec \CC[\widetilde{\sigma}^\vee \cap \widetilde{M}]$ is the $(n+r)$-dimensional affine toric variety corresponding to the cone
\begin{equation*}
  \widetilde{\sigma} = \langle e_0, F_1 + e_1, \dots, F_r + e_r\rangle_+\subset
  \widetilde{N} = \overline{N} \oplus \ZZ e_0 \oplus \ZZ e_1 \oplus \cdots \oplus \ZZ e_r
\end{equation*}
and $\cD= \left(\chi^{e_0^\star}=0\right)$.
\item 
  The morphism $f\colon \cX \to \AA^r_\CC$ is given by the $r$
  functions
  $\chi^{e_1^\star} - \chi^{e_0^\star}, \dots, \chi^{e_r^\star} -
  \chi^{e_0^\star}$.
\item The closed embedding $X \into \cX$ is the toric morphism induced by
the lattice homomorphism
 \[
 N = \overline{N} \oplus \ZZ \into \widetilde{N} = \overline{N} \oplus \ZZ e_0 \oplus \ZZ e_1 \oplus \cdots \oplus \ZZ e_r
 \]
defined by $(v,m) \mapsto v + m(e_0 + e_1 + \cdots + e_r)$.
\end{enumerate}
Then:
\begin{enumerate}[(1)]
\item $\cX$ is Gorenstein; indeed its toric boundary divisor is
  \[\bigl(\chi^{e_0^\star
      + e_1^\star + \cdots + e_r^\star}=0\bigr)\subset \cX
  \]
\item The morphisms $f\colon \cX \to \AA^r_\CC$ and its restriction
$\cD \to \AA^r_\CC$ are flat of relative dimension $n$ and $n-1$, and
the fibres over the origin are $X$ and $D$.
\end{enumerate}
\end{pro}

\begin{dfn}
  \label{dfn:local_Altmann_deformation}
  The diagram of Proposition~\ref{pro:Altmann_deformation} is the
  \emph{local Altmann deformation} of the pair $(X_F,D_F)$ associated with
  the given Minkowski decomposition of $F$.
\end{dfn}

\begin{proof}[Proof of Proposition~\ref{pro:Altmann_deformation}]
  The key point is to show that the morphisms
  $f\colon \cX \to \AA^r_\CC$ and its restriction $\cD \to \AA^r_\CC$
  are flat of relative dimension $n$ and $n-1$.  Because all toric
  varieties are Cohen--Macaulay, this follows from the fact that the
  $r$ functions
  $\chi^{e_1^\star} - \chi^{e_0^\star}, \dots, \chi^{e_r^\star} -
  \chi^{e_0^\star}$ form a regular sequence, see~\cite[Exercise
  18.18]{MR1322960}; alternatively,
  see~\cite[\href{https://stacks.math.columbia.edu/tag/00HT}{Tag
    00HT}]{stacks-project} and~\cite[Theorem~23.1]{MR1011461}.
\end{proof}

\begin{dfn}
  \label{dfn:transverse_singularity}
  Let $(V,B)$ be an affine toric pair. A variety $X$ has
  {$V$-singularities} along a Zariski closed subset $Z\subset X$ if there
  exist Zariski open subsets $Z\subset U$ in $X$ and $B\subset W$ in
  $V$ such that
  \[
U\cong W
  \]
\end{dfn}

\begin{pro} \label{pro:singularities_general_fibre_Altmann}

  Let $\overline{N}$ be a lattice of rank $n-1$,
  $F=F_1 + \cdots + F_r \subset \overline{N}$ a lattice polytope
  endowed with a Minkowski decomposition, 
 $\sigma = \langle F \times \{1\} \rangle_+ \subset N = \overline{N}
\oplus \ZZ$, $(X_F,D_F)$ the corresponding affine toric pair, and
  \[
    \xymatrix{
      (X_F,D_F) \ar@{^(->}[r]\ar[d]& (\cX,\cD) \ar[d]^f \\
\{0\} \ar@{^(->}[r]& \AA^r}
  \]
the local Altmann deformation.
  
  Let $\CC \subset K$ be a field extension and 
\[
a\colon \CC[t_1, \dots, t_r] \to K
\]
a $\CC$-algebra homomorphism such that the $a_j=a(t_j)\in K$ are
pairwise distinct and non-zero. Consider the base change
\[
\xymatrix{
(X_a, D_a )\ \ar[d] \ar@{^{(}->}[r] & (\cX,\cD)\ar[d]^f \\
\Spec K \ar@{^{(}->}[r]_a& \AA^r
}
\]
  
 For all $1\leq j \leq r$ let $\sigma_j=\langle F_j\times
  \{1\}\rangle_+\subset N$ and
  \[
    V_j=\Spec \CC[\sigma_j^\vee \cap M]
  \]
  \begin{enumerate}[(1)]
  \item The singular locus of $X_a$ consists of a disjoint union of
    connected components, where $X_a$ intersects transversely the
    singular locus of $\cX$. These components are in 
    bijective correspondence with the $V_j$ that are singular, and $X_a$
  has $V_j$-singularities along them.
\item The divisor $D_a = \cD \cap X_a \subset X_a$ is nonsingular and
  disjoint from the singular locus of $X_a$.
  \end{enumerate}
\end{pro}

\begin{rem}
  The key cases are $K=\CC$ (general fibre) and
  $K=\CC(t_1,\dots, t_r)$ (generic fibre).
\end{rem}

\begin{proof}[Proof of Proposition~\ref{pro:singularities_general_fibre_Altmann}]
We work over $K$, setting
\[
  \cX_K=\cX \times_{\Spec \CC} \Spec K,\quad \cD_K=\cD\times_{\Spec \CC} \Spec K
\]
etcetera, and we view $a\in \AA^r(K)$ as a $K$-valued point.

\begin{claim}
  \label{cla:subcones}
   For all $j = 1, \dots, r$ consider the subcone $\widetilde{\sigma}_j = \langle F_j
+ e_j\rangle_+ \subset \widetilde{\sigma}$ and the Zariski open affine
toric subscheme
\[
  U_j = \Spec K[\widetilde{\sigma}_j^\vee \cap \widetilde{M}] \subset \cX_K
\]
with its toric boundary $D_j\subset U_j$. Also consider the subcone
$\widetilde{\sigma}_0 = \langle e_0\rangle_+ \subset
\widetilde{\sigma}$ and the Zariski open
$U_0 = \Spec K[\widetilde{\sigma}_0^\vee \cap \widetilde{M}] \subset
\cX_K$ with toric boundary $D_0\subset U_0$.
  \begin{enumerate}[(1)]
  \item For all $j=0,\dots, r$ write $u_j=e_0^\star + \dots +
    \widehat{e_j^\star} + \dots + e_r^\star \in
    \widetilde{M}$. We have
    \[
      U_j=\{\chi^{u_j}\neq 0\} \subset \cX_K \quad \text{and} \quad
      D_j=\bigl(\chi^{e_j^\star}=0\bigr)\subset U_j
    \]
  \item For all $j=0,\dots, r$ write $N_j=\overline{N}\oplus \ZZ e_j$
    and $M_j=\Hom (N_j, \ZZ)=\overline{M}\oplus \ZZ e_j^\star$ so
    that, for example
    \[
\widetilde{M} = M_j  \oplus \ZZ e_0^\star \oplus \cdots \oplus
\widehat{\ZZ e_j^\star} \oplus \cdots \oplus \ZZ e_r^\star
\]
We will need the $r$-dimensional tori:
\[
\TT_j=\Spec K[\ZZ e_0^\star \oplus \cdots \oplus \widehat{\ZZ e_j^\star} \oplus \cdots \oplus \ZZ e_r^\star] 
\]
For $j\geq 1$, let
$\sigma_j=\langle F_j\times \{e_j\}\rangle_+\subset N_j$, and let
$\sigma_0=\langle e_0\rangle_+\subset N_0$. Then
\begin{equation*}
  (U_j,D_j)=(V_j\times \TT_j,B_j\times \TT_j) \quad \text{where} \quad
  V_j=\Spec K[\sigma_j^\vee \cap M_j]\; \text{and}\;B_j\subset V_j \; \text{is the
    toric boundary}
\end{equation*}
  \end{enumerate}
\end{claim}

Note that the notation in the claim is compatible, by a small abuse,
with the statement of the proposition. The proof of the claim is straightforward.

\begin{claim}
  \[
    X_a\subset U_0 \cup U_1 \cup \cdots \cup U_r
  \]
\end{claim}

We prove the claim. Since $X_a$ is the closed subscheme of $\cX_K$
defined by the equations $\chi^{e_i^\star} - \chi^{e_0^\star} = a_i$
($i=1, \dots, r$), it follows from the assumptions on the $a_i$ that
on $X_a$ no pair of the $\chi^{e_0^\star},
\chi^{e_1^\star}, \dots, \chi^{e_r^\star}$ have a common zero.  
This implies that, as was to be shown:
\begin{align*}
X_a &\subset \bigcap_{0 \leq i<j \leq r} \left( \{ \chi^{e_i^\star} \neq 0 \} \cup \{ \chi^{e_j^\star} \neq 0 \} \right) = \bigcup_{i=0}^r \bigcap_{j \neq i} \{ \chi^{e_j^\star}\neq 0 \} \\
&= \bigcup_{i=0}^r \{\chi^{u_i} \neq 0 \} = U_0 \cup U_1 \cup \cdots \cup U_r
\end{align*}

\begin{claim}
  For all $j=0,\dots, r$, $X_a\cap U_j\cong W_j$ where $W_j\subset
  V_j$ is the Zariski open neighbourhood of the toric boundary given
  by the conditions
  \begin{align*}
 \chi^{e_j^\star} - a_j &\neq 0, \\
\chi^{e_j^\star} + a_i - a_j &\neq 0 \quad  \forall \,i \geq 1, i \neq j
\end{align*}
\end{claim}
We prove the claim. Our assumption on the $a_i$ imply that $W_j$ is
contains the boundary $B_j\subset V_j$. The variety $X_a \cap U_j$ is given in $U_j$ by
the equations
\begin{align*}
\chi^{e_0^\star} &= \chi^{e_j^\star} - a_j, \\
\chi^{e_i^\star} &= \chi^{e_j^\star} + a_i - a_j \qquad \forall \; i \geq 1, i \neq j.
\end{align*}
The statement follows by solving for $\chi^{e_0^\star}$ and the
$\chi^{e_i^\star}$ ($i \geq 1, i \neq j$) in terms of
$\chi^{e_j^\star}$.

\smallskip

We finish with the proof of the proposition.
From the equations of $X_a \cap U_0$ in $U_0$ we see that $X_a \cap
U_0$ is isomorphic to an open subscheme of $\AA^1_K \times_K \Spec
K[\overline{M}]$ and hence it is smooth.
If $0 \leq i< j \leq r$, then $X_a \cap U_i \cap U_j$ is smooth, as
$U_i \cap U_j$ is the big torus of $\widetilde{X}_K$: this shows the first
part of the proposition. Finally, it is obvious 
that $D_a \cong \Spec K[ \overline{M}]$, therefore it is smooth.
\end{proof}

\begin{rem}
  It might be possible to give a proof of
  Proposition~\ref{pro:singularities_general_fibre_Altmann} by using
  the theory of T-varieties \cite{ilten_vollmert, geometryTvarieties,
    altmann_hausen_affineTvarieties}.
\end{rem}

\subsection{Simultaneous resolutions of the local Altmann deformation}
\label{sec:admiss-decomp-simult}

\begin{pro} \label{pro:admissible_decomposition_Matsushita} Let
  $\overline{N}$ be a $2$-dimensional lattice,
  $F=F_1+\cdots + F_r\subset \overline{N}$ a polygon endowed with an
  admissible Minkowski decomposition,
  $\sigma = \langle F \times \{ 1\} \rangle_+ \subset N = \overline{N}
  \oplus \ZZ$, $(X_F,D_F)$ the corresponding $3$-dimensional affine toric pair,
  and $f\colon (\cX,\cD) \to \AA^r$ the local Altmann deformation.
  Fix a dual tropical arrangement and let $\pi \colon (Y,E) \to (X,D)$
  be the induced partial resolution.

There exists a commutative diagram
\[
    \xymatrix{ (Y,E)\ar[dd]\ar[dr]^\pi\ar@{^{(}->}[rrr] &   & & (\cY,\cE)\ar[dd]_(.3){g}\ar[dr]^\Pi & \\
      &(X,D)\ar[dl] \ar@{^{(}-}[rr]&  & \ar@{->}[r]& (\cX,\cD)\ar[dl]^f\\
      \{0\} \ar@{^{(}->}[rrr]& & & \AA^r &}
  \]
where all squares are Cartesian and:
\begin{enumerate}[(i)]
\item The morphism $\Pi \colon \cY \to \cX$ is toric, birational,
  projective, and crepant;
\item The toric variety $\cY$ is Gorenstein, $\QQ$-factorial, and terminal;
\item The morphisms $g=f \circ \Pi \colon \cY \to \AA^r$ and its
  restriction $\cE=\Pi^{-1}(\cD)\to \AA^r$ are flat;
\item Let $\CC \subset K$ be a field extension,
  $a\colon \CC[t_1, \dots, t_r] \to K$ a $\CC$-algebra homomorphism
  such that the $a_j=a(t_j)\in K$ are pairwise distinct and non-zero,
  and consider the base change:
  \[
    \xymatrix{ (Y_a,E_a)\ar[dd]\ar[dr]^{\pi_a}\ar@{^{(}->}[rrr] &   & & (\cY,\cE)\ar[dd]_(.3){g}\ar[dr]^\Pi & \\
      &(X_a,D_a)\ar[dl] \ar@{^{(}-}[rr]&  & \ar@{->}[r]& (\cX,\cD)\ar[dl]^f\\
      \Spec K \ar@{^{(}->}[rrr]& & & \AA^r &}
  \]
  The singular locus of $X_a$ is a disjoint union of curves with
  transverse type $\text{A}$ singularities and
  $\pi_a\colon Y_a \to X_a$ is the minimal resolution.
\end{enumerate}
\end{pro}

\begin{proof}
  By construction, $\cX$ is the toric variety corresponding to the cone
  \[
    \widetilde{\sigma} = \langle e_0, F_1+e_1, \dots, F_r+e_r
    \rangle_+ \subset \widetilde{N}=\overline{N}\oplus \ZZ e_0\oplus
    \cdots \oplus \ZZ e_r
  \]
  This is the cone over the $(r+2)$-dimensional polytope
\begin{equation*}
\widetilde{F} = \conv{e_0, F_1 + e_1, \dots, F_r + e_r}
\end{equation*}
which is none other than the Cayley polytope --- see
Definition~\ref{dfn:cayley_polytope} --- of the Minkowski decomposition.
According to Theorem~\ref{thm:cayley_trick}, the tropical
arrangement determines a coherent fine mixed subdivision of $F$
induced by a regular triangulation of $\widetilde{F}$. 

This regular triangulation induces a subdivision of
$\widetilde{\sigma}$. The cones of this subdivision are the elements of a fan
$\Sigma$ in $\widetilde{N}$ and we denote by $\cY=\cY_\Sigma$ the
corresponding toric variety and by $\Pi \colon \cY \to \cX$ the
induced partial resolution. Because the subdivision
arises from a triangulation of $\widetilde{F}$, this morphism is
crepant.

Because $\Pi$ is crepant and comes from a triangulation, the variety
$\cY$ is Gorenstein and $\QQ$-factorial. By construction, every lattice
point of $\widetilde{F}$ is the primitive generator of some ray of the
fan defining $\cY$, hence $\cY$ has terminal singularities. 

\smallskip

The key point is to show that the morphisms
$g=f \circ \Pi \colon \cY \to \AA^r$ and its restriction
$\cE=\Pi^{-1}(\cD)\to \AA^r$ are flat, and this is what we do
next.

Exactly as in the proof of Proposition~\ref{pro:Altmann_deformation},
it is enough to show that
\begin{equation}
  \label{eq:regular_seq}
\chi^{e_1^\star} - \chi^{e_0^\star}, \cdots, \chi^{e_r^\star} - \chi^{e_0^\star}, \chi^{e_0^\star}  
\end{equation}
is a regular sequence in $\cO_{\cY}$. We prove this by showing
that~(\ref{eq:regular_seq}) is a regular sequence in every Zariski open
toric subset of $\cY$. For this purpose, pick a maximal --- that is,
$(r+3)$-dimensional --- cone $\lambda$ of the subdivision and denote by
\[
\cY_\lambda \subset \cY
\]
the corresponding Zariski open toric subset. We need to show
that~(\ref{eq:regular_seq}) is a regular sequence in
$\cO_{\cY_\lambda}$. 

The cone $\lambda$ is generated by vectors $v_1,\dots, v_{r+3}\in \widetilde{N}$
such that for all $k\in \{1,\dots, r+3\}$, there is a unique $i\in
\{1,\dots,r\}$ such that $v_k\in e_i+F_i$. For all $k\in \{1, \dots, r+3\}$, denote by $x_k$ the
Cox coordinate corresponding to the generator $v_k$. Then in Cox
coordinates we have, for all $i=0,\dots,r$
\[
\chi^{e_i^\star} = \prod_{k\colon v_k\in e_i+F_i} x_k
\]
We see that for all
$0\leq i,j\leq r$, $i\neq j$, the two functions
$\chi^{e_i^\star}, \chi^{e_j^\star}$ are expressions in two disjoint sets of Cox
coordinates, namely the set of $x_k$ such that $v_k\in e_i+F_i$ and
the set of $x_k$ such that $v_k\in e_j+F_j$. 
It follows from~\cite[Lemma~4.3]{petracci_mavlyutov}
that~(\ref{eq:regular_seq})
is a regular sequence in the Cox ring of $\cY_\lambda$. As the
elements in \eqref{eq:regular_seq} belong to
$\cO_{\cY_\lambda}$, which is the degree zero part of the
$\mathrm{Cl}(\cY_\lambda)$-grading on the Cox ring of $\cY_\lambda$,
it follows that ~(\ref{eq:regular_seq}) is a regular sequence in
$\cO_{\cY_\lambda}$.

\smallskip

The statement concerning the singular set of $X_a$ in Part~(iv)
follows immediately from
Proposition~\ref{pro:singularities_general_fibre_Altmann},
Part~(1). Because $\cY$ has terminal singularities, and terminal
singularities are nonsingular in codimension two, the morphism
$\Pi\colon \cY \to \cX$ is a crepant resolution of the singular strata of $\cX$ that
cause, according to Part~(1) of Proposition~\ref{pro:singularities_general_fibre_Altmann},
the singularities of $X_a$, and hence $\pi_a\colon Y_a\to X_a$ is the
minimal resolution.
\end{proof}

\subsection{The global Altmann deformation}
\label{sec:glob-altm-deform}

\begin{pro} \label{prop:altmann_global_Fano}
  Let $N$ be a $3$-dimensional lattice, $P\subset N$ a reflexive
  polytope, and $F\leq P$ a facet of $P$ endowed with a Minkowski decomposition
	\begin{equation} \label{eq:mink_deco_altmann_global}
		F = F_1 + \cdots + F_r
        \end{equation}

Let $(X_P,D_P)$ be the corresponding Fano toric pair, and denote by
$(X_F, D_F) \subseteq (X_P,D_P)$ the Zariski open affine toric pair
corresponding to $F$.

There exists a deformation
  \[
    \xymatrix{
      (X_P,D_P) \ar@{^(->}[r]\ar[d]& (\cX,\cD) \ar[d]^f \\
\{0\} \ar@{^(->}[r]& \AA^r}
  \]
  of the pair $(X_P,D_P)$ that induces the local Altmann deformation
  of the pair $(X_F,D_F)$.
\end{pro}

\begin{dfn}
  \label{dfn:global_Altmann_deformation}
  The diagram of Proposition~\ref{prop:altmann_global_Fano} is the
  \emph{global Altmann deformation} of the pair $(X_P,D_P)$ associated with
  the given Minkowski decomposition of the facet $F\leq P$.
\end{dfn}

\begin{proof}[Proof of Proposition~\ref{prop:altmann_global_Fano}]
  The proof is a straightforward application of the general
  construction given in~\cite[Theorem~6.1]{petracci_mavlyutov}.

  Denote by $v\in Q=P^\star\subset M$ the vertex of the polar polytope
  corresponding to $F$; then tautologically $F\subset \{v=-1\}$. We
  choose the following data to interpret
  Equation~\ref{eq:mink_deco_altmann_global}:
  \begin{enumerate}
  \item Lattice polygons $F_i\subset \{v=0\}$;
  \item A lattice vector $f_0\in \{v=-1\}$ --- for
    example we can choose $f_0$ to be a vertex of $F$ --- such that
    \[
F=\{f_0\} +  F_1+\cdots + F_r
\]
    holds in $N$. We write $F_0=\{f_0\}\subset N$, a
    $0$-dimensional polytope.
  \end{enumerate}
  
  We have that $(F, F_0, F_1, \dots, F_r, v)$ is a
  $\partial$-deformation datum, in the language of
  \cite[Definition~3.1]{petracci_mavlyutov}, for $(N, \sigma_F)$, where
  $\sigma_F$ is the cone over $F$. It is a simple exercise to show
  that the construction of~\cite[Theorem~4.1]{petracci_mavlyutov}
  applied to $(N, \sigma_F)$ starting from this $\partial$-deformation
  datum produces the local Altmann deformation of $(X_F, D_F)$.

  \smallskip
  
  Now consider the $4$-dimensional lattice
  $N_\dagger = N \oplus \ZZ e_\dagger$ and the cone
  $\sigma_\dagger=\langle P\times \{e_\dagger\}\rangle_+\subset
  N_\dagger$.  This is the cone corresponding to the polarised projective
  toric pair $(X_P, D_P)$ in the language
  of~\cite[Lemma~2.3]{petracci_mavlyutov}, and
  $(F + e_\dagger, F_0+e_\dagger, F_1, \dots, F_r, v)$ is a
  $\partial$-deformation datum for $(N_\dagger,
  \sigma_\dagger)$.\footnote{The direct sum decomposition $N_\dagger =
    N \oplus \ZZ e_\dagger$ induces a dual direct sum decomposition
    $M_\dagger = M \oplus \ZZ e^\star_\dagger$ (the notation is
    self-explanatory) by which we think of 
    $v\in M$ as of an element of $M_\dagger$ such that $\langle v,
    e_\dagger \rangle =0$.} The
  construction of~\cite[Theorem~6.1]{petracci_mavlyutov} applied to
  $(N_\dagger, \sigma_\dagger)$ starting from this $\partial$-deformation
  datum produces a global deformation of the Fano
  toric pair $(X_P, D_P)$ over $\AA^r$.

  \smallskip

  To finish the proof we verify that the restriction of this global
  deformation to $(X_F,D_F)$ is the local Altmann
  deformation. This, in fact, is a straightforward exercise in
  unpacking the construction of the deformation associated to a
  $\partial$-deformation datum described in the statement
  of~\cite[Theorem~3.5]{petracci_mavlyutov}. We sketch the
  verification leaving some of the details to the reader.

  The global deformation is constructed as follows. Consider the
  lattices:
  \[
    \widetilde{N}=N\oplus \ZZ e_1 \oplus \cdots \oplus \ZZ e_r, \quad
    \text{and} \quad
   \widetilde{N}_\dagger=N_\dagger \oplus \ZZ e_1 \oplus \cdots \oplus \ZZ e_r
  \]
  and the cone:
  \[
 C = \big\langle \sigma_\dagger , \{f_0+e_\dagger\}-e_1-\cdots -e_r,
F_1+e_1, \dots, F_r + e_r \big\rangle_+ \subset \widetilde{N}_\dagger
\]
The element $e_\dagger\in \interior C$ defines a $\ZZ$-grading on the ring $\CC[C^\vee
\cap \widetilde{M}_\dagger]$ and hence a projective toric variety:
\[
\widetilde{X} = \Proj \CC[C^\vee \cap \widetilde{M}_\dagger]
 \]
We now describe the total space of the global 
deformation as a subvariety of $\widetilde{X} \times \AA^r$ defined by
explicit equations in Cox coordinates. First note that
$\widetilde{X}$ is the toric variety of the fan
$\widetilde{\Sigma}\subset \widetilde{N}$ whose cones are the
images of the proper faces of $C$ under the obvious projection
\[
\widetilde{N}_\dagger \to \widetilde{N}=\widetilde{N}_\dagger/\ZZ e_\dagger
\]
Let us denote by $\widetilde{\Sigma}(1)$ the set of primitive generators of
the rays of $\widetilde{\Sigma}$. The Cox coordinates of $\widetilde{X}$ are
indexed by the elements $\rho \in \widetilde{\Sigma}(1)$ and we denote
by $x_\rho$ the corresponding coordinate, and by $t_j$, $j=1,\dots,r$
the standard coordinate functions on $\AA^r$. The total space of the
global deformation is the subvariety $\cX\subset
\widetilde{X}\times \AA^r$ defined by the homogeneous ideal generated
by the following trinomials in Cox coordinates:
\begin{equation}
  \label{eq:global_Altmann_equations}
\prod_{\langle e_j^\star , \rho \rangle>0} x_\rho^{\langle e_j^\star ,
  \rho \rangle}
- \prod_{\langle e_j^\star , \rho \rangle<0}x_\rho^{-\langle e_j^\star ,
  \rho \rangle}
-t_j\prod_{\rho} x_\rho^{\langle v,\rho\rangle}
\prod_{\langle e_j^\star , \rho \rangle<0} x_\rho^{-\langle
  e_j^\star,\rho\rangle}  
\end{equation}
for $j=1, \dots, r$. We need to compare these equations with the
equations that define the local Altmann deformation. Now, on the face of it, we don't understand
all the cones of $\widetilde{\Sigma}$. But we do know one maximal
cone, namely the cone:
\[
\widetilde{\sigma}=\big\langle \{f_0\}-e_1-\cdots -e_r, F_1+e_1,\dots,
F_r+e_r \big\rangle_+ \subset \widetilde{N}
\]
that is used in the construction of the local Altmann
deformation. Indeed denote by
\[
  \widetilde{X}_F=\Spec \CC[\widetilde{\sigma}^\vee \cap
  \widetilde{M}] \subset \widetilde{X}
\]
the corresponding affine open subset. It is immediate to check that
the restriction of Equations~\eqref{eq:mink_deco_altmann_global} to
$\widetilde{X}_F$ (for all $\rho \in \widetilde{\Sigma}$, if $\rho \not
\in \widetilde{\sigma}$ set $x_\rho =1$) give the equations of the
local Altmann deformation in $\widetilde{X}_F\times \AA^r$.   
\end{proof}

\subsection{Simultaneous resolution of the global Altmann deformation}
\label{sec:simult-part-resol}

The proof of the following Proposition will be given in \S~\ref{sec:proof_defoII}.

\begin{pro}
  \label{pro:TdefoII}
Let $N$ be a $3$-dimensional lattice, $P\subset N$ a reflexive
polytope, and $(X_P,D_P)$ be the corresponding Fano toric pair.

Let $F\leq P$ be a facet endowed with a Minkowski decomposition
\[
F=F_1+\cdots + F_r
\]
Denote by $S^{\text{an}}$ the analytic germ of $0\in \AA^r$ and by
\[
  \xymatrix{
    (X_P,D_P) \ar@{^(->}[r] \ar[d] & (\cX^{\text{an}}, \cD^{\text{an}}) \ar[d]^{f^{\text{an}}}\\ 
    \{0\} \ar@{^(->}[r] & S^{\text{an}}}
\]
the analytification of the global Altmann deformation constructed in
Proposition~\ref{prop:altmann_global_Fano}. 

Assume that $P$ is endowed with amd and let $\pi \colon (Y,E)\to
(X_P, D_P)$ be the induced partial resolution. There exists a
diagram of complex analytic deformations:
  \[
    \xymatrix{ (Y,E)\ar[dd]\ar[dr]^\pi\ar@{^{(}->}[rrr] &   & &
      \bigl(\cY^{\text{an}}, \cE^{\text{an}}\bigr)
      \ar[dd]_(.3){g^{\text{an}}}\ar[dr]^{\Pi^{\text{an}}} & \\
      &(X_P,D_P)\ar[dl] \ar@{^(-}[rr]&  & \ar@{->}[r]&
      \bigl(\cX^{\text{an}},\cD^{\text{an}}\bigr)
      \ar[dl]^{f^{\text{an}}} \\
      \{0\} \ar@{^{(}->}[rrr]& & &  S^{\text{an}} & }
  \]
\end{pro}

\begin{rem}
  \label{rem:global_altmann_resolution}
  \begin{enumerate}[(1)]
  \item We prove the result in \S~\ref{sec:proof_defoII} using
    deformation theory. An algebraic version of the statement probably holds, and it
    may be be possible to prove it by the methods of the proof of
    Proposition~\ref{prop:altmann_global_Fano}.
  \item In the statement it is enough to choose, for all facets of
    $P$, an admissible Minkowski decomposition and a dual tropical
    arrangement. It is not necessary to assume the compatibility
    condition along dull edges. We will not need this more general
    statement.
  \end{enumerate}
\end{rem}

\subsection{Some results on infinitesimal deformations}
\label{sec:some-remarks-infin}

We collect some facts about infinitesimal deformations of
the pair $(X,D)$ that we will need in \S~\ref{sec:proof_defoII} when
we prove Proposition~\ref{pro:TdefoII}. The reader can skip this
material and return to it when it is needed. 

\begin{nota}
  \label{sec:notation_reps}
  If $\T$ is a representation of $\TT =\Spec \CC[M]$ and $v\in M$, we
  denote by $\T(v)\subset \T$ the invariant summand on which $\TT$
  acts with character $v$.

  We also denote by $\CC(v)$ the one-dimensional representation of
  $\TT$ on which $\TT$ acts with character $v$.

  If $X$ is a scheme over a field $k$, we denote by $\T^1_X$ the
  tangent space of the deformation functor $\Def{X}$; if $D\subset X$
  is an effective Cartier divisor, we denote by $\T^1_{X,D}$ the
  tangent space of the deformation functor $\Def{X,D}$ of the pair
  $(X,D)$. See Appendix~\ref{sec:appendix} for a discussion of these
  matters.
\end{nota}

\begin{lem}
  \label{lem:T1(XD)and(X)}
    Let $(X,D)$ be a $3$-dimensional Gorenstein Fano toric pair. Denote by $Q\subset M$ be
    the moment polyhedron of $D$.

  There is an exact sequence:
  \[
    \rH^0(X,T_X) \to \rH^0\left(D, N_D X\right) \to \T^1_{X,D} \to
    \T^1_X \to \rH^1\left(D,N_D X\right) 
  \]
  where:
  \begin{enumerate}
  \item $\rH^1\left(D,N_D X\right) = (0)$;
  \item Denote by $Q^1$ the union of the edges of $Q$, including the
    vertices.
    \[
     \text{Write}\quad \C =   \coker \Bigl[  \rH^0(X,T_X) \to \rH^0\left(D, N_D
        X\right)\Bigr] ,\quad \text{then:} \quad  
      \C= \bigoplus_{v\in Q^1\cap M} \CC(v)
    \]
    as a representation of $\TT$;
  \item For all $v\in  Q^1\cap M$, $\T^1_X(v)=(0)$ and hence $\T^1_{X,D}(v)=\C(v)$.
  \end{enumerate}
\end{lem}

\begin{rem}
  \label{rem:T1(X,D)}
  \begin{enumerate}[(1)]
  \item   The statement applies both to the case $X$ compact and $X$
  affine. If $X$ is compact, then $X$ corresponds to a reflexive
  polytope $P\subset N$ and $Q=P^\star$ is the polar polytope. If $X$
  is affine then $X=X_\sigma$ for some Gorenstein cone $\sigma \subset
  N$; there is a vector $u\in \sigma^\vee$ such that $\divisor
  (\chi^u)=D$ and $Q=-u+\sigma^\vee$.
\item The proof follows easily from the difficult explicit description of
  $\T^1_{X_\sigma}$, for $\sigma \subset N$ a $3$-dimensional
  Gorenstein cone, given in Altmann's paper~\cite[\S~4.3]{MR1798979}.
  \end{enumerate}
\end{rem}

\begin{proof}
By Lemma~\ref{lem:two_short_exact_sequence_with_log_differentials} we
have an exact sequence:
\[
(0) \to \Omega_X \to \Omega_X(\log D) \to \cO_D \to (0)
\]
and by Proposition~\ref{pro:deformations_pairs_X_D}
$\T^1_{X,D}=\Ext^1 \bigl(\Omega_X(\log D),\cO_X\bigr)$. Taking
the long exact sequence of $\Ext$ we get
\[
(0) \to \rH^0\bigl( X,T_X(-\log D) \bigr)  \to \rH^0(X,T_X) \to
\Ext^1 \left(\cO_D ,\cO_X\right) \to \T^1_{X,D} \to
    \T^1_X \to \Ext^2\left(\cO_D,\cO_X\right) 
  \]
but then $\Ext^1 \left(\cO_D ,\cO_X\right)=\rH^0\left(X, \cExt^1(\cO_D
  ,\cO_X)\right)$ and $\cExt^1\left(\cO_D,\cO_X\right)=\cO_D(D)=N_D
X$. Similarly, $\Ext^2 \left(\cO_D ,\cO_X\right)=\rH^1\left(\cExt^1(\cO_D
  ,\cO_X)\right)=\rH^1\left(D,N_D X\right)$. Note $\cExt^2(\cO_D,\cO_X)=0$ since $D$ is Cartier. 
This establishes the exact sequence.

\smallskip

From the exact sequence $(0)\to \cO_X \to \cO_X(D)\to N_D X \to (0)$ we
get an exact sequence
\[
\rH^1\left(X, \cO_X(D)\right) \to \rH^1 (D, N_D X) \to \rH^2 (X,\cO_X) 
\]
where the groups on the left and right are $(0)$ by Kawamata Viehweg vanishing (using $D=-K_X$ ample) and hence the group
in the middle also is $(0)$.

\smallskip

To prove Part~(2), consider first a toric cone $\sigma=\langle
\rho_1,\dots \rho_r \rangle_+\subset N$, and let
$R_\sigma=k[\sigma^\vee\cap M]$. We work with the affine case $X=\Spec
R_\sigma$. It is stated for example
in~\cite[Theorem~3.2]{2018arXiv180909070S} that
\begin{equation}
  \label{eq:derivations}
  \Der_k(R_\sigma, R_\sigma)=\bigl(N\otimes R_\sigma \bigr)\oplus k[P_1]D_1 \oplus
\cdots \oplus k[P_r]D_r
\end{equation}
where:
\begin{enumerate}[(i)]
\item For all $v\in N$ $D_v\colon R_\sigma\to R_\sigma$ is the
  derivation defined by $D_v(x^m)=\langle v,m\rangle x^m$;
\item  The morphism $N\otimes R_\sigma\to \Der_k(R_\sigma, R_\sigma)$
  takes $v\otimes f$ to $fD_v$;
\item For all $j=1,\dots,r$,
  \[
    P_j=\bigl\{m\in M \mid \langle\rho_j, m \rangle=-1 \; \text{and} \; \forall
    i\neq j\; \langle\rho_i, m \rangle\geq 0\bigr\}
  \]
$k[P_j]$ is a $R_\sigma$-module in the obvious way, and $D_j=D_{\rho_j}$.
\end{enumerate}
In Equation~\eqref{eq:derivations}, the summand $N\otimes R_\sigma$
corresponds to the log derivations $T_X (-\log D)=N\otimes \cO_X$.
Part~(2) follows immediately in the affine case. Part~(2) in the
general case is proved by considering the toric affine cover.

Part~(3) in the affine case is an immediate consequence of the explicit description of $\T^1_{X_\sigma}$,
for $\sigma \subset N$ a $3$-dimensional Gorenstein cone, given in the
Theorem in~\cite[\S~4.3]{MR1798979}.  Part~(3) in the
general case is proved by considering the toric affine cover.
\end{proof}

\begin{lem}
  \label{lem:T1(X,D)}
  Let $P\subset N$ be a reflexive polytope, and $(X,D)$
  the corresponding toric pair. Denote by $Q\subset M$ be the polar
  polytope.
  
  Denote by $\Sigma$ the spanning fan of $P$, and by $\Sigma^\prime$ the
  set of its maximal cones. For all $\sigma \in \Sigma^\prime$,
denote by $(X_{\sigma},D_{\sigma})$ the corresponding affine toric
pair.
  
  \begin{enumerate}
  \item The natural homomorphism \begin{equation}
      \label{eq:T1injectivity}
  \T^1_{X,D} = \rH^0(X, \cT^1_{X,D})
  \to \bigoplus_{\sigma \in \Sigma^\prime} \T^1_{X_\sigma, D_\sigma}  
\end{equation}
 is injective.
\item  Let $v\in Q$ be a vertex. Denote by
$\sigma_v\subset \Sigma^\prime$ the cone corresponding to $v$. Equation~\ref{eq:T1injectivity} induces an isomorphism:
  \[
    \T^1_{X,D}(v) \cong \T^1_{X_{\sigma_v}, D_{\sigma_v}} (v) 
  \]
  \end{enumerate}  
\end{lem}

\begin{proof}
  We use freely the content of Appendix~\ref{sec:appendix}, especially
  Proposition~\ref{pro:deformations_pairs_X_D}. Computing
  $\T^1_{X,D}=\Ext^1_{\cO_X}\left(\Omega_X (\log D),\cO_X\right)$ with
  the local-to-global spectral sequence for $\Ext$ we get an exact
  sequence:
\[
 \rH^1 \left( X, T_X(-\log D)\right) \to \T^1_{X,D} \to
 \rH^0(X,\cT^1_{X,D}) \to \rH^2\left(X,T_X(-\log D)\right)
\]
But now by Remark~\ref{rem:log_tangents} $T_X(-\log D)=N\otimes
\cO_X \cong \cO_X^3$
and hence we get~(1).

Part~(2) follows from Lemma~\ref{lem:T1(XD)and(X)}.
\end{proof}

\section{Infinitesimal deformations of \texorpdfstring{$(Y, E)$}{(Y,E)}}
\label{sec:deform-theory-pair}

\subsection{Notation}
\label{sec:notation-1}

We summarize the notation that will be in force throughout
Sections~\ref{sec:deform-theory-pair}--\ref{sec:proof-theorem-2}. From
this point on in the paper, the reader should come back to this
section when in doubt about notation.  Most of this notation is
illustrated and summarised in Fig.~\ref{fig:excep_notation}. The
figure depicts a pair $(Y,E)$ above an affine neighbourhood $(X,D)$ of
a closed point $x\in X$ in the $0$-stratum. The blue surfaces are the
strict transforms of the components of $D\subset X$; the green
surfaces dominate a $1$-stratum $x\in \Gamma \subset X$; the pink
surfaces are contracted to the closed point $x\in X$. We invite the
reader to contemplate the picture before reading the following text.

\begin{figure}[ht]
  \centering
\begin{tikzpicture}
  \draw[thick, blue] (0,0) -- (2,-2) -- (2,-6) -- (0,-8)
  (2,-2) -- (4,-2) (2,-6) -- (4,-6);
  \draw[thick, red] (1,0) -- (1,-4) -- (-3,-7) -- (-5,-8)
  (1,-4) -- (4,-4) (-3,-7) -- (4,-7);
  \draw[thick, black] (3.5,-2) -- (3.5,-4) ;
  \fill[green!5!white] (2.1,-6.1) -- (4,-6.1) -- (4,-6.9) -- (1.1, -6.9)
  -- cycle ;
  \fill[green!5!white] (2.1,-4.1) -- (4,-4.1) -- (4,-5.9) -- (2.1, -5.9)
  -- cycle ;
  \fill[green!5!white]  (2.1,-2.1) -- (3.4,-2.1) -- (3.4,-3.9) -- (2.1, -3.9)
  -- cycle ;
  \fill[green!5!white]  (3.6,-2.1) -- (4,-2.1) -- (4,-3.9) -- (3.6, -3.9)
  -- cycle ;
  \fill[red!5!white]  (1,-4.1) -- (1.9,-4.1) -- (1.9,-6) -- (1,-6.9) -- (-2.7, -6.9)
  -- cycle ;
   \fill[red!5!white]  (1.1,-1.2) -- (1.9,-2) -- (1.9,-3.9) -- (1.1,-3.9) 
   -- cycle ;
   \fill[blue!5!white]  (1.1,0) -- (1.1,-1) -- (2,-1.9) -- (4,-1.9)  -- cycle ;
   \fill[blue!5!white]  (0.9,0) -- (0.9,-0.8) -- (0.1,0) -- cycle ;
    \fill[blue!5!white]  (-0.1,0) -- (0.9,-1) -- (0.9,-4) -- (-3,-6.9)
  -- (-5,-7.9) -- cycle ;
   \fill[blue!5!white]  (-3,-7.1) -- (0.8,-7.1) -- (-0.1,-8) --
   (-4.9,-8) -- cycle ;
   \fill[blue!5!white]  (1.1,-7.1) -- (4,-7.1) -- (0.1,-8) -- cycle ;
  \node at (2.75,-3) {$E_3^e$};
  \node at (2.75,-5) {$E_2^e$};
  \node at (2.75,-6.5) {$E_1^e$};
  \node at (4.5,-7) {$\Delta_1^e$};
  \node at (4.5,-6) {$\Delta_2^e$};
  \node at (4.5,-4) {$\Delta_3^e$};
  \node at (4.5,-2) {$\Delta_4^e$};
  \node at (3.75,-2.5) {$C_3^e$};
\end{tikzpicture}
\caption{Exceptional surfaces over transverse $\rm{A}_3$.} 
  \label{fig:excep_notation}
\end{figure}
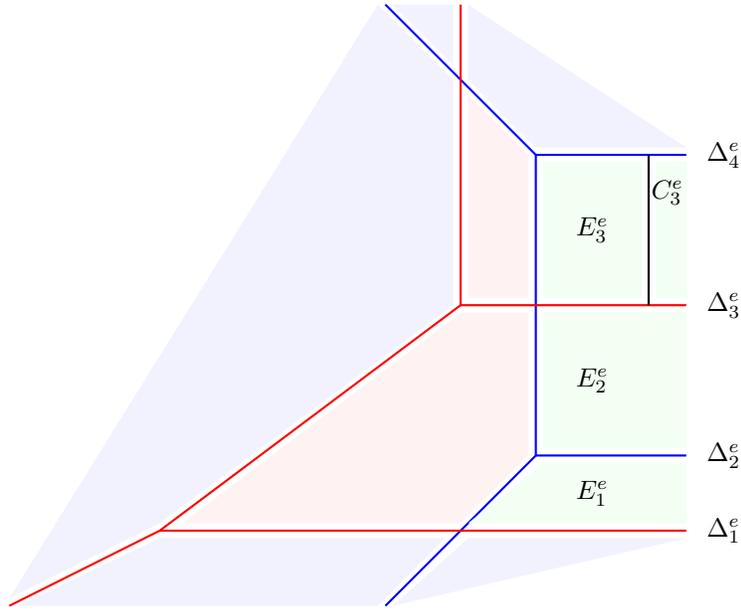

\smallskip

Fix a $3$-dimensional reflexive polytope $P$ endowed with
amd, and denote by $(X,D)=(X_P,D_P)$ the corresponding Fano toric
pair and by  $\pi \colon (Y,E) \to (X,D)$ the induced partial resolution.

\smallskip

By Lemma~\ref{lem:qODP}, the pair $(Y,E)$ has qODP singularities, that
is, it is simple normal crossing outside a finite set of
quasi-ordinary double points $x\in E \subset X$, \'etale-locally
isomorphic to
 \begin{equation*}
0 \in (x_1x_2=0) \subset  \left( x_1 x_2 - x_3 x_4 = 0 \right)/\mu_a
  \end{equation*}

\smallskip

We denote by $\Delta \subset E$ the singular locus of
$E$. At a qODP point, $\Delta$ is the union of the $4$
coordinate axes, as pictured on the right in Fig.~\ref{fig:axes_normalization}.
We denote by $\nu \colon \Delta^\prime \to \Delta$ the partial
normalization of $\Delta$ that is an isomorphism outside the qODPs, and
that at the qODPs is as pictured in Fig.~\ref{fig:axes_normalization}.
\begin{figure}[ht]
  \centering
\begin{tikzpicture}[point/.style={circle, inner sep=0pt, minimum
    size=3pt, fill=black}] 
  \draw[thick, blue, ->] (-1,0) -- (-1,-2);
  \draw[thick, blue, ->] (-1,0) -- (-1,2);
  \node at (-0.75,2) {$x_4$};
  \node at (-1.25,-2) {$x_3$};
  \node at (-1,0) [point] {} ;
  \draw[thick, red, ->] (3,0) -- (5,0);
  \draw[thick, red, ->] (3,0) -- (1,0);
  \node at (1,0.25) {$x_1$};
  \node at (5,-0.25) {$x_2$};
  \node at (3,0) [point] {};
  \node at (0,0) {$\coprod$};
  \node at (6.5,0) {$\overset{\nu}{\xrightarrow{\hspace{1truecm}}}$};
  \draw[thick, blue, ->] (10,0) -- (10,-2);
  \draw[thick, blue, ->] (10,0) -- (10,2);
  \node at (10,0) [point] {} ;
  \draw[thick, red, ->] (10,0) -- (12,0);
  \draw[thick, red, ->] (10,0) -- (8,0);
  \node at (8,0.25) {$x_1$};
  \node at (12,-0.25) {$x_2$};
  \node at (10.25,2) {$x_4$};
  \node at (9.75,-2) {$x_3$};
\end{tikzpicture}
\caption{The partial normalization $\nu \colon \Delta^\prime \to \Delta$.} 
  \label{fig:axes_normalization}
\end{figure}
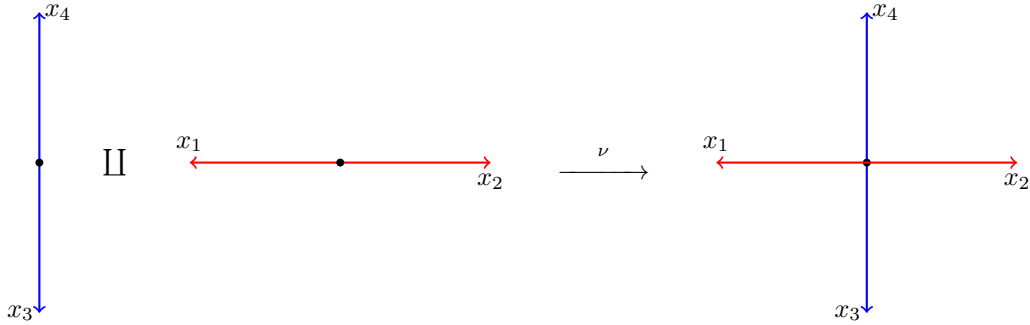

\smallskip

We denote by $Q=P^\star \subset M$ the polar polytope of $P$: $Q$ is
the moment polytope of $X$.

\smallskip

For all vertices $v\leq Q$, we denote by 
$F^v\leq P$ the corresponding facet of $P$, and by $r(v)$ the number of
Minkowski factors of $F^v$. We choose a labeling of those factors and write: 
\[
F^v=\sum_{i=1}^{r(v)} F^v_i
\]
We denote by $x^v\in X$ the $0$-stratum corresponding to $v$.

\smallskip

For all edges $e \leq P$, we denote by $\ell_e$ the (lattice) length of $e$, by
$e^\star=[v,w] \leq Q$ the dual edge, and by $k_e$ the colength of
$e$, that is, the length of $e^\star$.

We denote by $\Gamma^e\subset X$ the closure of the $1$-dimensional torus orbit
corresponding to $e$. Note that $X$ has transverse $\rm{A}_{\ell_e-1}$
singularities generically along
$\Gamma^e$, and that $-K_X\cdot \Gamma^e=k_e$.

\smallskip

The set $L^e$ of unit lattice segments of $e$ is in bijective
correspondence with the $1$-dimensional torus orbits of $Y$ that
dominate $\Gamma^e$. We choose an
orientation of $e$ identifying $L^e$ with the set
$[\ell_e]=\{1,2,\dots,\ell_e\}$ and hence denote by
\[
\Delta^e_1, \dots, \Delta^e_{\ell_e}
\]
the closures of the  $1$-dimensional torus orbits of $Y$ that dominate
$\Gamma^e$. We label the exceptional divisors of $\pi
\colon Y \to X$ that dominate $\Gamma^e$ in such a
way that
\[
\Delta^e_1 \subset E^e_1; \quad \forall \, 2\leq j\leq
\ell_e-1,\;\Delta^e_j=E^e_{j-1}\cap E^e_j; \quad \Delta^e_{\ell_e}
\subset E^e_{\ell_e-1}
  \]
  In this setting, we denote by $C^e_j\subset E^e_j$ a general fibre
  of $\pi\colon E^e_j\to \Gamma^e$. These labelling and numbering
  conventions are illustrated in Fig.~\ref{fig:excep_notation}.

\smallskip

  For all edges $e\leq P$ and dual edge $e^\star=[v,w]\leq Q$ and all
  $1\leq j \leq \ell_e$, the polarized curve
  $\bigl(\Delta^e_j, -K_Y\vert_{\Delta^e_j}\bigr)$ is canonically isomorphic
  to $\bigl(\PP^1,\cO (k_e)\bigr)$ and we denote by
  $[z_v:z_w]$ the homogeneous coordinates such that $x^v=[1:0]$,
  $x^w=[0:1]$.

\smallskip

Let $e\subset P$ be an edge and $e^\star=[v,w]$ the dual edge.
The set $L^e$
is endowed with a partition whose parts correspond to
Minkowski factors $F_i^v$. We denote by
\[
L^{v}_i\subset L^e\quad \text{the part corresponding to} \quad F^v_i
\] 
Note that the $F^v_i$ are in bijective correspondence with the
connected components of $(\pi \circ \nu)^{-1} (x^v)$.

\smallskip

The notation is summarized in Fig.~\ref{fig:excep_notation}, which
shows a typical configuration above a vertex-edge flag $v\in
e^\star\subset Q$ where
$F_v$ has two Minkowski factors, with
tropical curves pictures in blue and red, and $\ell_e=4$. In the
figure, $L^e=\{1,2,3,4\}$ and the partition is
$\bigl\{\{1,3\},\{2,4\}\bigr\}$. The exceptional divisors of $\pi
\colon Y \to X$ that map to $x^v$ are shaded in pink; those that
dominate $\Gamma^e$ are shaded in green; and the strict transform of
boundary divisors in $X$ are shaded in blue.

\subsection{Statement of purpose}
\label{sec:statement-purpose}

By Lemma~\ref{lem:qODP}, the pair
$(Y,E)$ has qODP singularities.

In Appendix~\ref{sec:appendix} we develop a general theory of
deformations and obstructions of pairs of a variety or stack and
effective Cartier divisor, and specialize the general theory to the
case of qODP pairs and qODP stack pairs. In particular, we construct
an obstruction theory for deformations of qODP pairs and qODP stack
pairs.

In this section we prove some facts about deformations
of the qODP stack pair $(Y,E)$. In particular we compute the tangent space
$\T^1_{Y,E}$ to the deformation functor as a
representation of the torus, and we show that $\T^2_{Y,E}=(0)$.

\subsection{Deformation theory of qODP stack pairs}
\label{sec:deform-theory-qodp}

\begin{lem} \label{lem:deformations_3fold_with_qODPs}
	Let $Y$ be a qODP stack and let $\Sing Y$ be the singular locus of $Y$ equipped with the reduced structure. Then
	$\cExt^1(\Omega_Y, \cO_Y) = \cO_{\Sing Y}$ and
	$\cExt^2(\Omega_Y, \cO_Y)=(0)$.  \qed
\end{lem}

\begin{lem}
  \label{lem:qODPdeformations}
  Let $(Y,E)$ be a qODP stack pair and let $\Omega_Y(\log E)$ be the
  coherent sheaf of Definition~\ref{dfn:logarithmic_differential}. Let
  $\Delta = \Sing E$ and let $\nu \colon \Delta^\prime \to \Delta$ be
  the partial normalisation of Notation~\ref{sec:notation-1}. Then:
\[
\cT^i_{Y,E} = \cExt^i _{\cO_Y}\left( \Omega_Y(\log E),\cO_Y \right) =
\begin{cases}
  T_{Y}(-\log E) \quad & \text{if } i=0,\\
\nu_\star \nu^\star (\cO_\Delta(E)) \quad & \text{if } i=1,\\
(0) \quad & \text{if } i=2.
\end{cases}
\]
\end{lem}

\begin{rem}
  \label{rem:log_str} The sheaf $\cT^1_{Y,E}$ is the sheaf
  $\mathcal{LS}_E$ of log structures on $E$ as defined
  in~\cite[Definition~3.19]{MR2213573}; see
  also~\cite{2023arXiv231213867C}. 
\end{rem}

\begin{proof}[Proof of Lemma~\ref{lem:qODPdeformations}]
	The case $i=0$ is obtained in Remark~\ref{rem:ext_sheaves_Omegalog}.
	By Remark~\ref{rem:ext_sheaves_Omegalog} we have $\cT^2_{Y,E} \simeq \cT^2_Y$. Therefore the case $i=2$ follows from Lemma~\ref{lem:deformations_3fold_with_qODPs}.
	The case $i=1$ follows from Proposition \ref{prop:Omegalog_for_X_smooth}(2) and Example \ref{exa:ODP_deformations}  (working intrinsically, cf. \cite{Friedman83}).
	
\end{proof}

\begin{lem}
\label{lem:toricqODPdeformations}
If $(Y, E)$ is a proper toric qODP stack pair, then 
\[
\T^i_{Y, E} = \Ext^i_{\cO_Y}(\Omega_Y(\log E), \cO_Y) = 
\begin{cases}
  \CC \otimes N \quad & \text{if }i=0,\\
  \rH^0\left(Y, \cT^1_{Y,E}\right)=\rH^0\left(\Delta^\prime, \nu^\star \left( -K_Y\right) \right)\quad & \text{if } i=1, \\
  \rH^1\left(Y, \cT^1_{Y,E}\right)=\rH^1\left(\Delta^\prime, \nu^\star \left(-K_Y\right) \right) \quad &\text{if } i=2.
\end{cases}
\]
\end{lem}

\begin{proof}
We consider the spectral sequence
\[
E_2^{p,q} = \rH^p(Y, \cT^q_{Y,E}) \Rightarrow \Ext^{p+q} (\Omega_Y(\log E), \cO_Y) = \T^{p+q}_{Y,E}.
\]
Recalling that $\cO_Y$ does not have higher cohomology, we easily conclude by Lemma~\ref{lem:qODPdeformations} and Remark~\ref{rem:log_tangents}.
\end{proof}

\subsection{Explicit computation of
  \texorpdfstring{$\T^1_{Y,E}$}{T1(Y,E)}} 
\label{sec:structureTT1}

We build an explicit basis of $\T^1_{Y,E}=\rH^0\left(Y,
  \cT^1_{Y,E}\right)$ as a representation of the torus $\TT=\Spec
\CC[M]$ and we show that $\T^2_{Y,E}=(0)$. As a corollary we get the
first part of Theorem~\ref{thm:1}, stating that $Y$ is smoothable. 

\begin{lem} 
  \label{lem:T1Y}
  Let $P\subset N$ be a reflexive polytope endowed with choices, for
  all facets, of admissible Minkowski decompositions and dual tropical arrangements.

  Let $(X,D)$ be the corresponding Fano toric pair and let $\pi\colon
  (Y,E)\to (X,D)$ the induced partial resolution.

  Assume the notation set out in \S~\ref{sec:notation-1}.

The following sections $s^v_i$ and $s^e_{j,m}$ give a
basis of $\T^1_{Y,E}=\rH^0\left(Y, \cT^1_{Y,E}\right)$:
\begin{itemize}
\item For all vertices $v$ of $Q$, and for every Minkowski factor
  $F^v_i$, $s=s^v_i$ is the unique section such that:
  \begin{enumerate}[(a)]
  \item $s$ takes constant value $1$ on the
  connected component of $(\nu \pi)\inv \left(x_v\right)$
  corresponding to $F^v_i$; and $s$ takes value $0$ everywhere else
  on $(\nu \pi)\inv \left(x_v\right)$;
\item If $e\subset F^v$ is an edge: if $j\in L^{v}_i$ then
  $s\vert_{\Delta^e_j}=z_v^{k_e}$; if $j\not \in L^{v}_i$ then $s\vert_{\Delta^e_j}=0$;
 \item If $e$ is an edge of $P$ that is not an edge of $F^v$, then for all
   $j\in L^e$ $s\vert_{\Delta^e_j}=0$.
  \end{enumerate}
\item For all edges $e \leq P$ such that $k_e\geq 2$,
  $j=1, \dots, \ell_e$, $1\leq m \leq k_e-1$, and dual edge
  $e^\star=[v,w]$, $s=s^e_{j,m}$ is the
  unique section such that: $s\vert_{\Delta^e_j}=z_v^mz_w^{k_e-m}$, and $s=0$
  everywhere else.
\end{itemize}

If $\T$ is a representation of $\TT$, and $m\in M$, denote by $\T(m)$
the invariant summand on which $\TT$ acts with character $m$. Denote
by $Q^1$ the union of the edges of $Q$, including the vertices. The
set of $m \in M$ such that $\T^1_{Y,E}(m)\neq (0)$ is $Q^1\cap
M$. More precisely,
\[
  s^v_i\in \T^1_{Y,E}(v) \quad \text{and} \quad s^e_{j,m}\in
  \T^1_{Y,E}(v+m(w-v))
\]
thus $\dim \T_{Y,E}^1(v)=r(v)$ (the number of Minkowski summands of
$F_v$) and $\dim \T_{Y,E}^1(m)=\ell_e$.
\end{lem}

\begin{proof}
  Straightforward combinatorics.
\end{proof}

\begin{cor}
  \label{cor:TT1nice_sections}
 Assume that the matching condition holds. Consider an edge $e^\star
 \leq Q$ with $\ell_e>1$. 
 There is a section $s\in \T^1_{Y,E}$ such that $s\vert_{\Delta^e_{j_1}}\neq s\vert_{\Delta^e_{j_2}}$ for
 $j_1\neq j_2$ .
\end{cor}

\begin{proof}
  If $k_e>1$ then the statement is easy. Assume that $k_e=1$, and
  let $v$ and $w$ be the vertices of $e^\star$. We
  argue that for general $\lambda_h, \mu_i\in \CC$
\[
s = \sum \lambda_h s^v_h + \sum \mu_i s^w_i
\] 
has the wished-for property. Indeed, assume for a contradiction that
there exist $j_1\neq j_2$ such that
$s\vert_{\Delta^e_{j_1}}=s\vert_{\Delta^e_{j_2}}$. There are $h_1, i_1$ such that
$j_1\in L^{v}_{h_1}\cap L^{w}_{i_1}$ and,
similarly, there are $h_2, i_2$ such that
$j_2\in L^{v}_{h_2}\cap L^{w}_{i_2}$. By assumption
$\lambda_{h_1}=\lambda_{h_2}$ and $\mu_{i_1}=\mu_{i_2}$. Since these
constant are generic, this implies further that $L^{v}_{h_1}=L^{v}_{h_2}$
and $L^{w}_{i_1}=L^{w}_{i_2}$. The matching condition states that for
all $h,i$ $|L^w_h \cap L^v_i|\leq 1$ and it implies that $j_1=j_2$.
\end{proof}

\begin{thm} \label{thm:3}
  Let $P\subset N$ be a reflexive polytope endowed with amd, $(X,D)$
  the corresponding Fano toric pair, and $\pi\colon (Y,E)\to (X,D)$
  the induced partial resolution.

  Regard the pair $(Y,E)$ as a qODP stack pair in the obvious 
  way. Then
  
  \begin{enumerate}[(1)]
\item The sheaf $\cT^1_{Y,E}$ is generated by global sections and
$\rH^1\left(Y, \cT^1_{Y,E}\right)=(0)$;
\item For $i>0$ $\rH^i\left(Y, T_Y(-\log E)\right)=(0)$;
\item $\T^2_{Y,E}=(0)$; that is, the pair $(Y,E)$ is unobstructed and
  hence smoothable.
  \end{enumerate}
\end{thm}

\begin{proof}
  For Part~(1) consider the partial normalization $\mu \colon
  \widetilde{\Delta} \to \Delta^\prime$ that pries out all of the
  $\Delta^e_j$ from the body of $\Delta^\prime$; then we have an exact
  sequence 
\[
(0)\to \cO_{\Delta^\prime} (-K_Y)\to \mu_\star \mu^\star
\cO_{\Delta^\prime} (-K_Y) \to \delta \to (0)
\]
and one sees right away that the homomorphism
\[
\rH^0\bigl(\Delta^\prime, \mu_\star \mu^\star
\cO_{\Delta^\prime} (-K_Y) \bigr) \to \delta
\]
is surjective. It is also clear that $\rH^1\bigl(\widetilde{\Delta}, \mu^\star
\cO_{\Delta^\prime} (-K_Y) \bigr) = (0)$.

\smallskip

For the rest of the proof we use freely the content of
Appendix~\ref{sec:appendix}.

Part~(2) is an immediate consequence of Remark~\ref{rem:log_tangents}
stating that $T_Y(-\log
E)=N\otimes_\ZZ\cO_Y\cong \cO_Y^3$.

We now prove Part~(3). By Lemma~\ref{lem:qODPdeformations}
\[
\cT^1_{Y,E} =\cExt^1_{\cO_Y}\left(\Omega_Y(\log E),\cO_Y\right) 
 \]
and $\T^2_{Y,E}=\Ext^2_{\cO_Y}\left(\Omega_Y (\log
  E),\cO_Y\right)$. Computing this group with the
local-to-global spectral sequence for $\Ext$ we get an exact sequence:
\[
\rH^2 \left( Y, T_Y(-\log E)\right) \to \T^2_{Y,E} \to
\rH^1\left(Y,\cT^1_{Y,E}\right) 
\]
Now conclude with the vanishing proved in Parts~(1) and~(2).
\end{proof}

\subsection{Blowing down deformations}
\label{sec:blow-down-deform}

\begin{lem}
  \label{lem:blow_down_defo} Let $P\subset N$ be a $3$-dimensional reflexive polyhedron
  endowed with choices, for all facets, of admissible Minkowski decompositions
  and dual tropical arrangements. Let $(X,D)$ the corresponding Fano
  toric pair and $\pi\colon (Y,E)\to (X,D)$ the induced partial
  resolution.
  
  Let $g\colon (\cY,\cE) \to S$ be a deformation of the pair $(Y,E)$
  (either algebraic or complex analytic); then
  \begin{enumerate}[(1)]
  \item For all $n\geq 0$ and all $i\geq 1$
    $R^i g_\star \cO_{\cY}(-nK_{\cY})=(0)$. For all $n\geq 0$, $g_\star \cO_{\cY}(-nK_{\cY})$
    is a vector bundle.
  \item There is a diagram of deformations:
  \[
    \xymatrix{ (Y,E)\ar[dd]\ar[dr]^\pi\ar@{^{(}->}[rrr] &   & &
    \bigl(\cY, \cE \bigr)
      \ar[dd]_(.3){g}\ar[dr]^{\Pi} & \\
      &(X,D)\ar[dl] \ar@{^(-}[rr]&  & \ar@{->}[r]&
      \bigl(\cX,\cD\bigr)
      \ar[dl]^{f} \\
      \{0\} \ar@{^{(}->}[rrr]& & &   S& }
  \]
where
  \[
    \cX = \uProj  \oplus_{n\geq 0} g_\star \cO_{\cY}
    (-nK_{\cY})
  \]
\item Every fibre of $f$ is a Fano \mbox{$3$-fold} with
  Gorenstein canonical singularities.
  \end{enumerate}
\end{lem}

\begin{proof}

The birational morphism $\pi$ is crepant and toric so $R^i\pi_*\cO_Y(-nK_Y) = R^i\pi_*\cO_Y = 0$ for $i>0$ since toric varieties have rational singularities, see e.g. \cite{Fulton93}, p.~76.
Then  also $\rH^i(Y,-nK_Y)=(0)$ for $i>0$ by Leray and Kawamata--Viehweg vanishing for $X$. Part~(1) then follows from cohomology and base change.

Part~(2) now follows from the fact that $X= \Proj  \oplus_{n\geq 0}
\rH^0(Y, -nK_Y)$, see also~\cite[Theorem~1.4]{Wahl}.

Next we show Part~(3). First, it follows from Part~(2) that, for
all $t\in S$, the fibre $\cX_t$ is Gorenstein and
the contraction morphism $\pi_t \colon \cY_t
\to \cX_t$ is crepant. For all $t$, the fibre
$\cY_t$ has qODP, hence terminal singularities, and hence
$\cX_t$  has at worst canonical singularities.
\end{proof}

\begin{dfn}
  \label{dfn:blow_down}
  The deformation $f\colon (\cX,\cD) \to S$ constructed in
  Lemma~\ref{lem:blow_down_defo} is the \emph{blow-down} of
  the deformation $g\colon (\cY, \cE)\to S$.
\end{dfn}
\section{Homogeneous deformations II}
\label{sec:homog-deform-ii}

\subsection{Statement of purpose}
\label{sec:statement-purpose-1}

The purpose of this section is to prove Proposition~\ref{pro:TdefoII}.
The statement is natural and unsurprising: if you are
willing to take it on trust, we suggest that you jump to
Section~\ref{sec:proof-theorem} and see how it is used to complete the
proof of Theorem~\ref{thm:1}.

Our proof uses abstract deformation
theory. Specifically, we need: (a)  the theorem of
Kuranishi--Douady--Grauert~\cite{MR0141139, MR0382729, MR0346194} on
the existence of the complex analytic versal deformation of a compact
complex analytic space, and (b) the work of Rim~\cite{MR549162} on
$G$-equivariant structures on formal versal families, which for
convenience we recall in Appendix~\ref{sec:equiv-g-struct}.


\subsection{Proof of
  \texorpdfstring{Proposition~\ref{pro:TdefoII}}{Proposition 4.1}}
\label{sec:proof_defoII}

\begin{setup}
  \label{sec:setup_sec4}
  This setup will be in force throughout this section.
  Let $\widetilde{f}^{\text{an}} \colon
 (\widetilde{\cX}^{\text{an}},\widetilde{\cD}^{\text{an}})
 \to \cM_{X,D}^{\text{an}}$ be a miniversal
  complex analytic deformation of the pair $(X,D)$. 

  Let $f^{\text{an}}\colon 
  (\cX^{\text{an}},\cD^{\text{an}}) \to \cM_{Y,E}^{\text{an}}$ be the
  blow-down deformation.

  Choose a complex analytic morphism $\Pi_\star \colon \cM_{Y,E}^{\text{an}}\to
  \cM_{X,D}^{\text{an}}$ and a fibre square:
  \begin{equation}
    \label{eq:blow_dn_diagram}
        \xymatrix{(\cX^{\text{an}},\cD^{\text{an}}) \ar[r]
      \ar[d]_{f^{\text{an}}}&
      (\widetilde{\cX}^{\text{an}},\widetilde{\cD}^{\text{an}}) 
      \ar[d]^{\widetilde{f}^{\text{an}}} \\
      \cM_{Y,E}^{\text{an}}  \ar[r]_{\Pi_\star}&  \cM_{X,D}^{\text{an}}}
  \end{equation}

  (The fibre square exists  by the versal property of
  $\widetilde{f}^{\text{an}}$.)

\smallskip  

Let
$f^{\text{an}}_v \colon (\cX_v^{\text{an}} ,\cD_v^{\text{an}})\to S_v
^{\text{an}}$ be the analytification of the global Altmann family.
Choose a complex analytic morphism
$j\colon S_v ^{\text{an}}\to \cM_{X,D}^{\text{an}}$ and a fibre square
\begin{equation}
  \label{eq:Altmann_diagram}
    \xymatrix{
       \bigl(\cX_v^{\text{an}}
       ,\cD_v^{\text{an}}\bigr)\ar[d]_{f^{\text{an}}_v}\ar[r]&
       (\widetilde{\cX}^{\text{an}} , \widetilde{\cD}^{\text{an}}) \ar[d]^{\widetilde{f}^{\text{an}}}\\
       S_v ^{\text{an}} \ar[r]_j & \cM_{X,D}^{\text{an}}}
\end{equation}

    (The existence of such a fibre square follows from the versal property of
   $\widetilde{f}^{\text{an}}\colon
   (\widetilde{X}^{\text{an}},\widetilde{D}^{\text{an}} )\to
   \cM_{X,D}^{\text{an}}$.)

   \smallskip

   In the course of the proof, we identify the vertex $v\in M$ with the
  corresponding character of the torus and denote by $\T^1_{Y,E}(v)$
  the invariant summand on which $\TT$ acts with character $v$.
  \end{setup}
   
   \begin{lem}
     \label{lem:injectivities}
     Consider the diagram of morphisms and their derivatives, also
     known as Kodaira--Spencer maps 
     \[
       \xymatrix{ & \cM_{Y,E}^{\text{an}} \ar[d]^{\Pi_\star}\\
         S_v^{\text{an}} \ar[r]_j & \cM_{X,D}^{\text{an}}
       }
       \quad
          \xymatrix{ & \T^1_{Y,E} \ar[d]^{D\Pi_\star=\kappa}\\
         \CC^r=T_0 S_v \ar[r]_{Dj=\kappa^\dagger} & \T^1_{X,D}
       }
       \]
       Then $\kappa^\dagger$ is injective, $\kappa$ is injective on
       $\T^1_{Y,E}(v)$, and $Dj(\CC^r)=D\Pi_\star (\T^1_{Y,E}(v))$.
     \end{lem}

     \begin{proof}
By construction, the global Altmann
deformation is $\TT$-equivariant and the tangent space $T_{0}S_v=\CC^r$
is $v$-isotypic. Both $\kappa$ and $\kappa^\dagger$
are $\TT$-equivariant and it follows that both images are contained in
the direct summand $\T^1_{X,D}(v)\subset \T^1_{X,D}$.

Denote by $\Sigma$ the spanning fan of $P$, and by $\Sigma^\prime$ the
set of its maximal cones. Among these cones, there is the cone
$\sigma_v$ corresponding to $v$. For all $\sigma \in \Sigma^\prime$,
denote by $(X_{\sigma},D_{\sigma})$ the corresponding affine toric
pair. Also write $Y_\sigma=\pi^{-1}(X_\sigma), E_\sigma=\pi^{-1}(D_\sigma)$.

Lemma~\ref{lem:T1(X,D)} states that the localization map
  \[
    \T^1_{X,D}(v) \to \T^1_{X_{\sigma_v}, D_{\sigma_v}}(v)
  \]
is an isomorphism (we only need that it is injective). We prove
the result by composing with the localization map and studying
images in $\T^1_{X_{\sigma_v}, D_{\sigma_v}}$.

It follows from the explicit
description of Lemma~\ref{lem:T1Y} that
$\T^1_{Y,E}(v) =\T^1_{Y_{\sigma_v},E_{\sigma_v}}(v)$. Also, we have a
commutative diagram where the notation is self-explanatory:
\begin{equation}
  \label{eq:CD_A}
  \xymatrix{
    \T^1_{Y,E} (v) \ar@{=}[r]    \ar_{\kappa}[d] & \T^1_{Y_{\sigma_v},E_{\sigma_v}}(v)
 \ar[d]^{\kappa_v} &\\
\T^1_{X,D}(v) \ar[r] & \T^1_{X_{\sigma_v},D_{\sigma_v}}(v)
}  
\end{equation}

Denoting by $f_{\sigma_v} \colon (\cX_{\sigma_v},
\cD_{\sigma_v} )\to S_v$ the local Altmann deformation, and by
${\kappa_v^\dagger} \colon T_{0}S_v\to \T^1_{X_{\sigma_v}, D_{\sigma_v}}$ its
Kodaira--Spencer map, we also have a commutative diagram:
\begin{equation}
  \label{eq:CD_B}
    \xymatrix{T_{0}
      S_v\ar[d]_{\kappa^\dagger}\ar[dr]^{\kappa_v^\dagger}& \\
      \T^1_{X,D}  (v)\ar[r] & \T^1_{X_{\sigma_v}, D_{\sigma_v}}(v)
      }  
\end{equation}
The two commutative diagrams~\ref{eq:CD_A} and~\ref{eq:CD_B}, together
with Lemma~\ref{lem:T1(X,D)} show that the statement that we have to
prove follows from its local version: namely, the statement that the
two maps:
    \[
\kappa_v \colon \T^1_{Y_{\sigma_v},E_{\sigma_v}}(v) \to
\T^1_{X_{\sigma_v}, D_{\sigma_v}}(v) \quad \text{and} \quad
 \kappa_v^\dagger \colon T_{0} S_v \to \T^1_{X_{\sigma_v}, D_{\sigma_v}}(v)
\]
have the same image. To prove this local statement, consider the
simultaneous partial resolution
\[
  g_{\sigma_v} \colon (\cY_{\sigma_v}, \cE_{\sigma_v})\to S_v
\]
of the local Altmann deformation provided by
Proposition~\ref{pro:admissible_decomposition_Matsushita}, and denote
by $\lambda_v^\dagger \colon T_{0}S_v \to
\T^1_{Y_{\sigma_v},E_{\sigma_v}}(v)$ its Kodaira--Spencer
map. Everything follows from the following two simple facts:
\begin{enumerate}[(a)]
\item $\lambda_v^\dagger$ is an isomorphism;
\item The partial resolution of the local Altmann deformation induces
  a factorization of Kodaira--Spencer maps
  \[
    \xymatrix{ T_{0} S_v \ar[r]^(.4){\lambda_v^\dagger} \ar[dr]_{\kappa_v^\dagger} &
      \T^1_{Y_{\sigma_v},E_{\sigma_v}}(v)\ar[d]^{\kappa_v}\\
       & \T^1_{X_{\sigma_v},D_{\sigma_v}}(v)
      }
    \]
  \end{enumerate}
     \end{proof}

  \begin{lem}
    \label{lem:Ginvariant_formal}
  \begin{enumerate}[(1)]
  \item Theorem~\ref{thm:Tequivariant_versal} with the setup
    $\frakD=\catDeff{Y}{E}$, $\T^1=\T^1_{Y,E}$, $\TT=\Spec \CC[M]$ and
    $\TT$-invariant subspace $W=\T^1(v)\subset \T^1$ gives a formal $\TT$-equivariant deformation:
     \[
    \xymatrix{ (Y,E) \ar[d]\ar@{^(->}[r]& (\widehat{\cY}_v, \widehat{\cE}_v) \ar[d]^{\widehat{g}_v}\\
    \{0\} \ar@{^(->}[r]& \widehat{\T}^1(v)}
\]
(where we denoted by $\widehat{\T}^1(v)$ the formal completion of
$\T^1(v)$ at the origin).
  \item Consider the $\TT$-equivariant blow-down diagram of formal
    $\TT$-equivariant deformations
    \[
    \xymatrix{ (Y,E)\ar[dd]\ar[dr]^\pi\ar@{^{(}->}[rrr] &   & &
      \bigl(\widehat{\cY}_v, \widehat{\cE}_v\bigr)
      \ar[dd]_(.3){\widehat{g}_v}\ar[dr]^{\widehat{\Pi}_v} & \\
      &(X,D)\ar[dl] \ar@{^(-}[rr]&  & \ar@{->}[r]&
      \bigl(\widehat{\cX}_v,\widehat{\cD}_v\bigr)
      \ar[dl]^{\widehat{f}_v} \\
      \{0\} \ar@{^{(}->}[rrr]& & &  \widehat{\T}^1(v) & }
  \]
   The formal $\TT$-equivariant deformation
 $\widehat{f}_v\colon \bigl(\widehat{\cX}_v,\widehat{\cD}_v\bigr)\to \widehat{\T}^1(v)$ is 
 formally isomorphic to the formal completion of the global Altmann family
 $f_v\colon (\cX_v,\cD_v) \to S_v=\AA^r$ of Proposition~\ref{prop:altmann_global_Fano}.      
  \end{enumerate}
  \end{lem}

  \begin{proof}
Part~(1) is a consequence of the fact that the pair $(Y,E)$
is unobstructed: Theorem~\ref{thm:Tequivariant_versal} gives a family
over a formal subscheme of $\widehat{\T}^1_{Y,E}(v)$ and we are just
saying that this subscheme is all of $\widehat{\T}^1_{Y,E}(v)$.

Let us prove Part~(2). For sake of clarity, let us denote by
\[
  \xymatrix{(X,D)\ar@{^(->}[r] \ar[d]& (\cX^\dagger_v,\cD^\dagger_v) \ar[d]^{f_v^\dagger}\\
  \{0\} \ar@{^(->}[r] & S_v}
\]
the global Altmann family of Proposition~\ref{prop:altmann_global_Fano}. We want to compare
$\widehat{f}_v$ with $\widehat{f}^\dagger_v$.


Theorem~\ref{thm:Tequivariant_versal} with the setup
$\frakD=\catDeff{X}{D}$, $\T^1=\T^1_{X,D}$, $\TT=\Spec \CC[M]$ and
$\TT$-invariant subspace $W=\kappa \left(\T^1(v)\right)=\kappa^\dagger (\CC^r)$ gives a formal $\TT$-equivariant deformation:
\[
    \xymatrix{(X,D) \ar[d]\ar@{^(->}[r]& (\widehat{\cX}_W, \widehat{\cD}_W) \ar[d]^{\widehat{f}_W}\\
    \{0\} \ar@{^(->}[r]& \widehat{M}}
\]
where, denoting by $\widehat{W}$ the formal completion at the origin,
$\widehat{M}\subset \widehat{W}$ is a formal subscheme. The families
$\widehat{f}_v$ and $\widehat{f}_v^\dagger$ are both induced by the
miniversal family $\widehat{f}_W$: it follows from this and Lemma~\ref{lem:injectivities}
that $\widehat{M}=\widehat{W}$ and that the three families are formally
isomorphic.
\end{proof}


 \begin{proof}[Proof of Proposition~\ref{pro:TdefoII}]

   \begin{claim}
   There exists a complex analytic lift $\widetilde{j}$ of $j$ as in
   the diagram:
   \[
     \xymatrix{
       & \cM_{Y,E}^{\text{an}}\ar[d]^{\Pi_\star} \\
       S_v^{\text{an}}\ar[r]_j \ar@{-->}[ru]^{\widetilde{j}}&  \cM_{X,D}^{\text{an}}}
   \]
   \end{claim}

   The claim finishes the proof: the pull-back family
   $\widetilde{j}^\star
   (\cY ^{\text{an}},\cE ^{\text{an}}) \to S_v ^{\text{an}}$ is the
   family that we want.

   \smallskip
  
   It remains to prove the claim. By~\cite[Theorem~1.5(ii)]{MR232018}
   it is enough to prove that the lift exists formally. The existence
   of a formal lift is a more-or-less immediate consequence of
   Lemma~\ref{lem:Ginvariant_formal} and the fact that
   $\Aut (Y,E)=\Aut(X,E)=\TT$. For clarity, we spell out the
   detail. Choose ideals
 \[
   \cO_{S_v}\supset \mathfrak{m}_{S_v}=J_1 \supset J_2 \supset \cdots
 \]
 such that
 \begin{enumerate}[(a)]
 \item for all $m\geq 1$ $J_m/J_{m+1}\cong \CC$;
 \item for all $n>0$ there exists $m$ such that
   $\mathfrak{m}_{S_v}^n\supset J_m$. 
 \end{enumerate}

 For all $m\geq 0$, let $S_v^m =\Spec (\cO_{S_v}/J_m)$ and denote by
 $j^m\colon S_v^m \to \cM_{X,D}^{an}$, etcetera, the induced
 morphism.

 For a fixed $n\geq 1$, assume by induction that we have constructed a
 lift to order $n$
  \[
     \xymatrix{
       & \cM_{Y,E}^{\text{an}}\ar[d]^{\Pi_\star} \\
       S_v^n \ar[r]_{j^n} \ar@{-->}[ru]^{\widetilde{j}^n}&  \cM_{X,D}^{\text{an}}}
   \]
   We construct a lift $\widetilde{j}^{n+1} \colon S_v^{n+1} \to
   \cM_{Y,E}^{\text{an}}$. This finishes the proof, since the base case $n=1$
   holds by Lemma~\ref{lem:injectivities}.

   For all $m\geq 0$, the morphism $j^m\colon S_v^m \to \cM_{X,D}^{\text{an}}$
   induces a deformation of the pair
   $(X,D)$ that we denote by
    \[
\xi^m =(j^m)^\star (\cX,\cD) \in \ob \catDeff{X}{D} \left( S_v^m \right) 
\]

    The lift $\widetilde{j}^n \colon S_v^n\to 
    \cM_{Y,E}^{\text{an}}$ induces a deformation 
    \[
\eta^n=(\widetilde{j}^n)^\star (\cY,\cE) \in \ob \catDeff{Y}{E} \left(
  S^n_v \right)
   \]
of the pair $(Y,E)$ over $S_v^n$. By construction, this deformation
blows down to the given deformation $\xi^n$ of the pair $(X,D)$, and we write this fact as $\pi_\star \eta^n=\xi^n$.

By Lemma~\ref{lem:Ginvariant_formal} there
exists a deformation
    \[
\eta^{n+1}_\sharp \in \ob\catDeff{Y}{E} (S_v^{n+1})
\]
that extends $\eta^n$ and blows down to an object
isomorphic to $\xi^{n+1}$.

By the versal property of the miniversal family there exists a
morphism $\widetilde{j}^{n+1}_\sharp \colon S_v^{n+1} \to \cM_{Y,E}^{\text{an}}$
such that $\eta^{n+1}_\sharp=(\widetilde{j}^{n+1}_{\sharp})^\star
(\cY,\cE)$. By construction,
$\Pi_\star \circ \widetilde{j}^{n+1}_\sharp \equiv j^{n+1}$~(mod~$n$)
as morphisms from $S_v^{n+1}$ to $\cM_{X,D}^{\text{an}}$.

In addition, as we said, $(\Pi_\star \circ
\widetilde{j}^{n+1}_\sharp)^\star(\cX,\cD)$ is isomorphic to
$\xi^{n+1}=(j^{n+1})^\star (\cX,\cD)$.

It follows that there exists $a\in
\cO_{S_v^1}\otimes \aut (X,D)$ such that
\begin{equation}
  \label{eq:tricky_deformations}
  j^{n+1}=a+(\Pi_\star \circ \widetilde{j}^{n+1}_\sharp) 
\end{equation}
as morphisms to $\cM_{X,D}^{\text{an}}$. But
$\Aut^0 (X,D)=\Aut^0 (Y, E)$, thus regarding 
$a$ as an element of $\cO_{S_v^1}\otimes \aut (Y,E)$ we have that
$\widetilde{j}^{n+1}=\widetilde{j}^{n+1}_\sharp+a$ is a lift of $j^{n+1}$.
\end{proof}

\section{Proof of \texorpdfstring{Theorem~\ref{thm:1}}{Theorem 1}}
\label{sec:proof-theorem}

In this section, we prove Theorem~\ref{thm:1}.

\begin{rem}
  \label{rem:only_cDV} Theorem~\ref{thm:1} states that $\cX_t$ has at
  worst ordinary double point singularities. In this section we
  only show that $\cX_t$ has at worst isolated cDV singularities; the
  full statement will be proved in \S~\ref{sec:proof-theorem-2}, see Remark~\ref{rem:ODP}.
\end{rem}

\begin{lem}
  \label{lem:fibre_dimensions}
  Let $N$ be a $3$-dimensional lattice, $P\subset N$ a reflexive
  polytope endowed with amd, $(X,D)$ be the corresponding Fano toric
  pair and $\pi\colon (Y,E)\to (X,D)$ the induced partial resolution.

Consider a complex analytic miniversal deformation space of the pair
$(Y,E)$ and the blow-down deformation
  \[
    \xymatrix{ (Y,E)\ar[dd]\ar[dr]^\pi\ar@{^{(}->}[rrr] &   & &
    \bigl(\cY^{\text{an}}, \cE^{\text{an}}\bigr)
      \ar[dd]_(.3){g^{\text{an}}}\ar[dr]^{\Pi^{\text{an}}} & \\
      &(X,D)\ar[dl] \ar@{^(-}[rr]&  & \ar@{->}[r]&
      \bigl(\cX^{\text{an}},\cD^{\text{an}}\bigr)
      \ar[dl]^{f^{\text{an}}} \\
      \{0\} \ar@{^{(}->}[rrr]& & &   \cM_{Y, E}^{\text{an}}& }
  \]
  
For $i=1,2$ consider the subsets:
\[
  Z^i=\{x\in \cX \mid \dim \pi^{-1} (x) \geq i \}
  \quad \text{and, for all $t\in \cM_{Y,E}^{\text{an}}$,} \quad
  Z^i_t= Z^i \cap \cX_t
\]

Then $Z^i\subset \cX$ is a complex analytic subspace and, for
$t\in \cM_{Y,E}^{\text{an}}$ general
\begin{enumerate}[(1)]
\item $Z^2_t=\emptyset$;
\item $Z^1_t$ is finite.
\end{enumerate}
  \end{lem}

  \begin{proof} First of all, by a theorem of Cartan--Remmert~\cite[Chap.~V, Sec.~3.3, Theorem~5]{MR1131081}
    \[
      W^i=\{y\in \cY \mid \dim \pi^{-1} \pi (y) \geq i \}
    \]
is a complex analytic subspace. Because $\pi$ is proper, then by the
Remmert proper mapping theorem $Z^i=\pi
(W^i)\subset \cX$ is also a complex analytic subspace.

Let us show that for $t\in \cM_{Y,E}^{\text{an}}$ general
$Z^2_t=\emptyset$. First of all, $Z^2_0\subset X$ is a finite set
contained in the $0$-skeleton of the toric variety $X$. It follows
from this that all of $Z^2$ is contained in a small analytic
neighbourhood of the $0$-skeleton of $X$ in $\cX$. Let us fix a vertex
$v\in Q$ and let us focus on the corresponding closed point
$x^v\in X$. The complex analytic deformation of
Proposition~\ref{pro:TdefoII} is induced by an inclusion
$i\colon S_v^{\text{an}} \to \cM_{Y,E}^{\text{an}}$. By
Proposition~\ref{pro:singularities_general_fibre_Altmann}, there is a
small analytic neighbourhood $x^v\in U \subset \cX$ such that for
$t \in S_v^{\text{an}}$ general, the singular locus of
$U_t=\cX_t\cap U$ consists of a disjoint union $\coprod C_i$ of curves
and for all $i$ $U_t$ has transverse $\text{A}_{n_i}$-singularities along $C_i$ and
therefore $Z^2\cap U_t=\emptyset$ and this must then be true for all
$t\in \cM_{Y,E}^{\text{an}}$ general. Since this is true for all
$v$, we conclude that there is a small analytic neighbourhood
$V\subset \cX$ of the $0$-skeleton of $X$ such that for
$t\in \cM_{Y,E}^{\text{an}}$ general $Z^2\cap V_t
=\emptyset$. Because, as we said, in fact $Z^2\subset V$, we have that
for $t\in \cM_{Y,E}^{\text{an}}$ general $Z^2_t=\emptyset$.

Finally we show that for $t\in \cM_{Y,E}^{\text{an}}$ general
$Z^1_t \subset \cX_t$ is finite. First of all $Z^1_0\subset X$ is contained in
the $1$-skeleton. Let us fix an edge $e\subset Q$ and let us focus on
the corresponding component $\Gamma^e\subset X$ of the
$1$-skeleton. It is enough to show that for $x\in \Gamma^e$ general,
there exists an open neighbourhood $x\in U \subset \cX$ such that for
$t\in \cM_{Y,E}^{\text{an}}$ general $U_t=\cX_t\cap U$ is smooth. For this purpose
consider the ``interior'' $\Gm\subset \Gamma^e$. The space $X$ has
transverse $\text{A}_{\ell_e-1}$-singularities along this $\Gm$ and the germ of $X$
along this $\Gm$ is isomorphic to the germ:
\[
(xy-z^{\ell_e}=0) \subset \AA^3\times \Gm
\]
and any $1$-parameter deformation of $X$ over a small disk $\DD$ induced by a general
deformation of the pair $(Y,E)$ with Kodaira--Spencer map a general section
(Lemma~\ref{lem:T1Y})
\[
  s \in \T^1_{Y,E} = \rH^0\bigl(Y, \cT^1_{Y,E}\bigr) = \rH^0
  \bigl(\Delta, \nu_\star \nu^\star (-K_Y)\bigr)
\]
can be written as:
  \[
\bigl[ xy=\prod_{i=1}^{\ell_e}(z-tA_i(u)) \quad \text{mod} \; t^2 \bigr] \subset \AA^4_{x,y,z,u}
\times \DD_t
  \]
  where $A_j(u)=s\vert_{\Delta^e_j \cap \pi\inv \left(\Gm\right)}$
  are Laurent polynomials supported on $\Gamma^e$ such that the
  multiple roots of the polynomial
\[
\prod (z-tA_j(u))
\]
are all double roots and they occur for distinct values of $u$. By the Jacobian criterion the
total space is nonsingular at the generic point of $\Gm$ and the
statement follows.
\end{proof}

\begin{proof}[Proof of Theorem~\ref{thm:1}] Part~(1) is Part~(3) of
  Theorem~\ref{thm:3}, stating that the pair $(Y,E)$, regarded as a
  qODP stack, is unobstructed and smoothable.

  Lemma~\ref{lem:blow_down_defo} states that indeed $\cX_t=\Proj
  R(\cY_t, -K_{\cY_t})$. Lemma~\ref{lem:fibre_dimensions} shows that
  the exceptional set of $\Pi_t$ consists of a finite number of curves
  that are contracted. Because $\cY_t$ has qODP, and hence terminal,
  singularities, and the morphism $\Pi_t$ is crepant and small, it
  follows that $\cX_t$ has Gorenstein terminal and hence,
  by~\cite[Theorem~(2.2)]{MR605348}, isolated $\text{cDV}$
  singularities. As anticipated this is all that we show here. The
  rest of the statement follows from Theorem~\ref{thm:topologyII}, see
  Remark~\ref{rem:ODP}. 
\end{proof}

\section{Local invariance under rearrangement}
\label{sec:flops}


The content of this section is rather technical. The goal is to prove
Theorem~\ref{thm:local_invariance}. Before getting down to business,
we discuss informally the statement and its proof.

\smallskip

\paragraph{\textbf{Statement of purpose}} The statement of
Theorem~\ref{thm:local_invariance} is very natural. The starting point
is a plane polygon with a given admissible Minkowski decomposition, a
dual tropical arrangement subordinated to the Minkowski decomposition,
and the induced partial resolution $\pi \colon (Y, E)\to (X,D)$. The
statement describes what happens to a deformation $Y\subset \cY\to (0\in \cM)$
when we change the dual tropical arrangement. The answer is
that, given a second tropical arrangement subordinated to the same
Minkowski decomposition, and corresponding induced partial resolution
$\pi^\prime \colon (Y^\prime, E^\prime) \to (X,D)$, there is a
birational transformation
\[
  \xymatrix{
    \cY \ar@{-->}[rr]^\Phi\ar[dr]& &  \cY^\prime\ar[dl]\\
           &(0\in\cM) & 
    }
\]--- a flop,
in fact, although we don't use the terminology --- and $\cY^\prime\to
(0\in \cM)$ is a deformation of $Y^\prime$.

Theorem~\ref{thm:local_invariance} is needed in the key step in the
proof of our main result Theorem~\ref{thm:topology}, where it is used
to modify the total space of a deformation into one where the
conclusions are easier to see --- all the action is happening in the proof of
Lemma~\ref{lem:topology}.

To prove Theorem~\ref{thm:local_invariance}, we manoeuvre the first
tropical arrangement into the second by a sequence of elementary moves
that we classify into types~I--IV (and inverses of moves of
type~I). In all cases, the existence of the birational transformation
$\Phi$ follows easily from the general results of~\cite{BCHM10}, but
it is not immediately clear that $\cY^\prime\to (0\in \cM)$ is a
deformation of $Y^\prime$. The natural way to prove the result is to
prove it for the miniversal deformation family of $Y$ and then prove
it for all families by pulling back from the miniversal family. For
moves of type~II and~IV, $Y$ has nonisolated singularities and the
miniversal family is an infinite-dimensional ind-scheme and in that
case the miniversal property only holds formally, see
\S~\ref{sec:some-gener-princ}. This is sufficient to show that
$\cY^\prime$ is a deformation of $Y^\prime$ once we know that
$\cY^\prime$ exists.

\medskip

\begin{thm}
  \label{thm:local_invariance}
  Let $\overline{N}$ be a rank-$2$ lattice, $F\subset \overline{N}$ a lattice polygon, and $m=(F=\sum F_j)$
  an admissible Minkowski decomposition.

  Let $N=\overline{N}\oplus \ZZ$, $\sigma=\langle F\times \{1\}
  \rangle_+$, and $(X,D)$ the corresponding affine toric pair. 
  
  Let $\sum \Gamma_j$ be a generic dual tropical arrangement and
  $\pi \colon (Y, E)\to (X,D)$ the induced partial resolution.

    Let $g\colon (\cY,\cE) \to 0\in \cM$ be a (complex analytic or
     algebraic) deformation of $(Y,E)$ over a smooth germ $0\in \cM$,
     $f\colon (\cX,\cD)\to 0 \in \cM$ the blow-down deformation of
     $(X,D)$, and $\Pi\colon \cY \to \cX$ the obvious morphism.
 
  \smallskip
  
  Let $\sum \Gamma_j^\prime$ be another generic dual tropical
  arrangement and $\pi^\prime \colon (Y^\prime, E^\prime) \to (X,D)$
  the induced partial resolution.

  \smallskip If the morphism $\Pi \colon \cY \to \cX$ is small, then
  there exists a birational transformation over $\mathcal{X}$:
  \[
    \xymatrix{
      \left( \mathcal{Y}, \mathcal{E}\right) \ar[dr]_\Pi \ar@{-->}[rr]^\Phi & & \left( \mathcal{Y}^\prime, \mathcal{E}^\prime\right) \ar[dl]^{\Pi^\prime}\\
       & \left( \mathcal{X}, \mathcal{D}\right)}
     \]
     such that:
     \begin{enumerate}[(i)]
     \item The morphism $\Pi^\prime \colon \cY^\prime \to \cX^\prime$
       is small;
     \item The morphism $f\circ \Pi^\prime \colon \mathcal{Y}^\prime \to \cM$ is flat;
     \item The rational map $(Y^\prime,E^\prime) \dasharrow
       \left(\mathcal{Y}^\prime,\cE^\prime \right)$ is a morphism and
       it sends $(Y,E)$ isomorphically to the fibre over $0\in \cM$;
     \item With the inclusion of Part~(ii), the morphism $f\circ
       \Pi^\prime \colon \left(\cY^\prime,\cE^\prime\right)\to (0\in
       \cM)$ is a 
       deformation of $(Y^\prime,E^\prime)$;
     \item The birational map $\Phi \colon (\cY, \cE) \dasharrow
       (\cY^\prime, \cE^\prime)$ is a composition of moves of type~I,
       their inverses, and type II--IV of Definition~\ref{dfn:types}
       and Lemma~\ref{lem:moves}.
     \end{enumerate}
   \end{thm}

   \begin{lem}
     \label{lem:moves}
     Let $(Y,E)$ be one of the toric pairs corresponding to the tropical
     arrangements pictured on the left of
     Figures~\ref{fig:Type_1}--\ref{fig:Type_4} and
     $\pi\colon (Y,E)\to (X,D)$ the associated contraction of the ``internal''
     proper curves and surfaces on $Y$.

     Let $g\colon (\cY,\cE) \to 0\in \cM$ be a (complex analytic or
     algebraic) deformation of $(Y,E)$ over a smooth germ $0\in \cM$,
     $f\colon (\cX,\cD)\to 0 \in \cM$ the blow-down deformation of
     $(X,D)$, and $\Pi\colon \cY \to \cX$ the obvious morphism.

     Let $(Y^\prime, E^\prime)$ be the pair given by the tropical
     arrangements pictured on the right of the same figure and
     $\pi^\prime\colon (Y^\prime, E^\prime) \to (X,D)$ the induced partial resolution.

     If the morphism $\Pi \colon \cY \to \cX$ is small, then there exists a birational transformation over $\mathcal{X}$:
  \[
    \xymatrix{
      \left( \mathcal{Y}, \mathcal{E}\right) \ar[dr]_\Pi \ar@{-->}[rr]^\Phi & & \left( \mathcal{Y}^\prime, \mathcal{E}^\prime\right) \ar[dl]^{\Pi^\prime}\\
       & \left( \mathcal{X}, \mathcal{D}\right)}
     \]
     such that:
     \begin{enumerate}[(i)]
     \item The morphism $\Pi^\prime \colon \cY^\prime \to \cX^\prime$
       is small; 
     \item The morphism $f\circ \Pi^\prime \colon \mathcal{Y}^\prime \to \cM$ is flat;
     \item The rational map $(Y^\prime,E^\prime) \dasharrow
       \left(\mathcal{Y}^\prime,\cE^\prime \right)$ is a morphism and
       it sends $(Y,E)$ isomorphically to the fibre over $0\in \cM$;
     \item With the inclusion of Part~(ii), the morphism $f\circ
       \Pi^\prime \colon \left(\cY^\prime,\cE^\prime\right)\to (0\in
       \cM)$ is a deformation of $(Y^\prime,E^\prime)$;
     \item The effect on the Kodaira--Spencer morphisms with values in
       $\T^1_{Y,E}$, $\T^1_{Y^\prime, E^\prime}$ is pictured in the four cases in Figures~\ref{fig:Type_1}--\ref{fig:Type_4} using the identification $\T^1_{Y,E}=\rH^0\bigl(\Delta, \nu_\star \nu^\star (-K_Y)\bigr)$ of Lemma~\ref{lem:toricqODPdeformations}. Here $A(x)=\sum a_ix^i$, $B(y)=\sum b_iy^i$, etc.
     \end{enumerate}
     Figure~\ref{fig:Type_1} can also be read from right to left and the appropriate statements hold.
     
     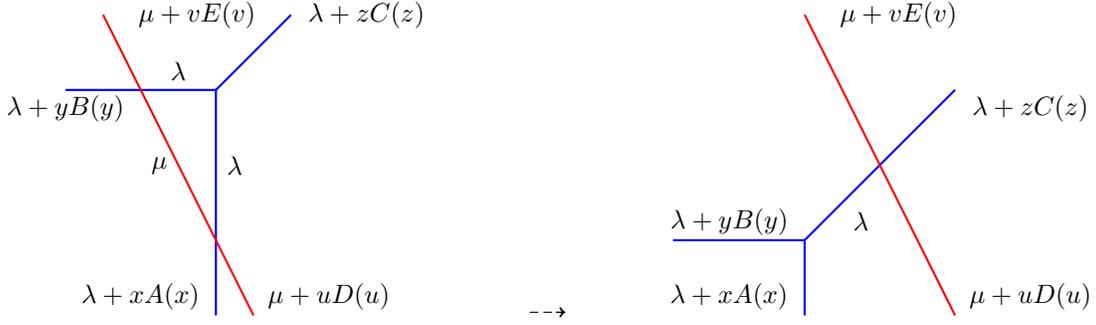
\begin{figure}[ht]
  \centering
\begin{tikzpicture}
  \draw[thick, blue] (0,0) -- (0,-3) (0,0) -- (-2,0) (0,0) -- (1,1);
  \draw[thick, red] (-1.5,1) -- (0.5,-3);
  \node at (-1,-2.75) {$\lambda + xA(x)$};
  \node at (-2,-0.25) {$\lambda + yB(y)$};
  \node at (2,1) {$\lambda + zC(z)$};
  \node at (1.5,-2.75) {$\mu + uD(u)$};
  \node at (-0.25,1) {$\mu + vE(v)$};
  \node at (-0.75,-1) {$\mu$};
  \node at (0.25,-1) {$\lambda$};
  \node at (-0.5,0.25) {$\lambda$};
\end{tikzpicture}
\hspace{1cm}
$\dasharrow$
\hspace{1cm}
\begin{tikzpicture}
  \draw[thick, blue] (-1,-1) -- (-2.75,-1) (-1,-1) -- (-1,-2) (-1,-1) -- (1,1);
  \draw[thick, red] (-1,2) -- (1,-2);
  \node at (-2,-1.75) {$\lambda + xA(x)$};
  \node at (-2,-0.75) {$\lambda + yB(y)$};
  \node at (-0.25,-0.75) {$\lambda$};
  \node at (2,0.75) {$\lambda + zC(z)$};
  \node at (2,-1.75) {$\mu + uD(u)$};
  \node at (0.25,2) {$\mu + vE(v)$};
\end{tikzpicture}
\caption{Move of Type~I.}
  \label{fig:Type_1}
\end{figure}

     \begin{figure}[ht]
  \centering
\begin{tikzpicture}
  \draw[ thick,red] (0,-2) -- (2.25,-2) (0,-2) -- (-2.5,-2);
  \draw[thick,blue] (0,-3) -- (-2.5,-0.5)
  (0,-3) -- (2.25, -3) (0, -3) -- (0,-5);
  \node at (1.25,-2.25) {$\mu+uD_1(u)$};
  \node at (1.25,-3.25) {$\lambda+uD_0(u)$};
  \node at (-0.75,-2.75) {$\lambda$};
  \node at (-1,-4.75) {$\lambda + xA(x)$};
  \node at (-2,-2.25) {$\mu + zC(z)$};
  \node at (-1,-0.75) {$\lambda +yB(y)$};
\end{tikzpicture}
\hspace{1cm}
$\dasharrow$
\hspace{1cm}
\begin{tikzpicture}
  \draw[ thick,red] (0,-3) -- (1.5,-3) (0,-3) -- (-2.5,-3);
  \draw[thick,blue] (-1,-2) -- (-2.5,-0.5)
  (-1,-2) -- (1.5, -2) (-1, -2) -- (-1,-5);
  \node at (0.5,-2.25) {$\lambda+uD_0(u)$};
  \node at (0.5,-3.25) {$\mu+uD_1(u)$};
  \node at (-1.25,-2.5) {$\lambda$};
  \node at (-2,-4.75) {$\lambda + xA(x)$};
  \node at (-2,-3.25) {$\mu + zC(z)$};
  \node at (-1,-0.75) {$\lambda +yB(y)$};
\end{tikzpicture}
\caption{Move of Type~II.}
  \label{fig:Type_2}
\end{figure}

\begin{figure}[ht]
  \centering
\begin{tikzpicture}
  \draw[thick, blue] (-0.5,-1.5) -- (-0.5,3.5) ;
  \draw[thick, red] (-2.5,0.5) -- (2,0.5) ;
  \draw[thick, green] (-2.5,3) -- (1.5,-1) ;
  \node at (-1.75,-1.25) {$\lambda_1 + xA(x)$};
  \node at (-1.75,0.25) {$\lambda_2 + zC(z)$};
  \node at (-0.25,0.25) {$\lambda_2$};
  \node at (0,1) {$\lambda_3$};
  \node at (1.5,0.75) {$\lambda_2 + uD(u)$};
  \node at (-0.75,0.75) {$\lambda_1$};
  \node at (-1.5,3.25) {$\lambda_3 + vE(v)$};
  \node at (0.75,3.25) {$\lambda_1 + yB(y)$};
  \node at (0.75,-1.25) {$\lambda_3 + wF(w)$};
 \end{tikzpicture}
 \hspace{1cm}
 $\dasharrow$
 \hspace{1cm}
 \begin{tikzpicture}
  \draw[thick, blue] (3.5,0) -- (3.5,4.75) ;
  \draw[thick, red] (1,2.5) -- (6,2.5) ;
  \draw[thick, green] (1,4.5) -- (5,0.5) ;
  \node at (2.25,0.25) {$\lambda_1 + xA(x)$};
  \node at (1.75,2.25) {$\lambda_2 + zC(z)$};
  \node at (4.75,2.75) {$\lambda_2+uD(u)$};
  \node at (4.75,4.5) {$\lambda_1+yB(y)$};
  \node at (2.25,4.5) {$\lambda_3+vE(v)$};
  \node at (3,2) {$\lambda_3$};
  \node at (3.25,2.75) {$\lambda_2$};
  \node at (3.75,2.25) {$\lambda_1$};
   \node at (4.75,0.25) {$\lambda_3 + wF(w)$};
 \end{tikzpicture}
  \caption{Move of type~III}
    \label{fig:Type_3}
  \end{figure}

\begin{figure}[ht]
  \centering
\begin{tikzpicture}
  \draw[thick, blue] (-2,0) -- (2,0) ;
  \draw[thick, red] (-2,-1) -- (2,-1) ;
  \draw[thick, green] (-0.5,1.5) -- (0.5,-2.5) ;
  \node at (-1,-1.25) {$\lambda_1 + xA(x)$};
  \node at (-1.25,0.25) {$\lambda_2 + zC(z)$};
  \node at (-0.25,-0.5) {$\lambda_3$};
  \node at (1.25,0.25) {$\lambda_2 + uD(u)$};
  \node at (-1.5,1.25) {$\lambda_3 + vE(v)$};
  \node at (1.5,-1.25) {$\lambda_1 + yB(y)$};
  \node at (1.75,-2.25) {$\lambda_3 + wF(w)$};
 \end{tikzpicture}
 \hspace{1cm}
 $\dasharrow$
 \hspace{1cm}
 \begin{tikzpicture}
  \draw[thick, red] (-2,0) -- (2,0) ;
  \draw[thick, blue] (-2,-1) -- (2,-1) ;
  \draw[thick, green] (-0.5,1.5) -- (0.5,-2.5) ;
  \node at (-1.25,0.25) {$\lambda_1 + xA(x)$};
  \node at (-1,-1.25) {$\lambda_2 + zC(z)$};
  \node at (1.25,-1.25) {$\lambda_2+uD(u)$};
  \node at (1,0.25) {$\lambda_1+yB(y)$};
  \node at (-1.5,1.25) {$\lambda_3+vE(v)$};
   \node at (-0.25,-0.5) {$\lambda_3$};
   \node at (1.75,-2.25) {$\lambda_3 + wF(w)$};
 \end{tikzpicture}
  \caption{Move of type~IV}
    \label{fig:Type_4}
  \end{figure}
  
   \end{lem}

   \begin{dfn}
     \label{dfn:types}
     The birational maps $(\cY,\cE)\dasharrow (\cY^\prime, \cE^\prime)$ of Lemma~\ref{lem:moves}, pictured in Figures~\ref{fig:Type_1}--\ref{fig:Type_4}, are called \emph{moves of type I--IV} and --- reading Fig.~\ref{fig:Type_1} from right to left --- the \emph{inverse of a move of type~I}.
   \end{dfn}
     
   \begin{proof}[Proof of Theorem~\ref{thm:local_invariance}]
     The morphism $\pi \colon Y \rightarrow X$ is projective by
     Theorem~\ref{thm:cayley_trick}. (Recall that a subdivision of a
     lattice polyhedron $P \subset \overline{N} \otimes \RR$ is
     regular iff the associated proper birational toric morphism to
     the affine toric variety
     $X=\Spec k[\Cone(P \times\{1\})^{\vee} \cap (\overline{N} \oplus
     \ZZ)^{\vee}]$ is projective.)

     We can rearrange the $\sum \Gamma_j$ into the
     $\sum \Gamma^\prime_j$ by a sequence of moves of type~I, its
     inverse, and II--IV as described in Lemma~\ref{lem:moves}.  So,
     by induction, we may assume that $\pi \colon Y \rightarrow X$ and
     $\pi' \colon Y' \rightarrow X$ factor through $Z \rightarrow X$
     such that in the toric affine chart at a $0$-stratum of $Z$, the
     morphisms $Y \rightarrow Z$ and $Y' \rightarrow Z$ are as
     described in Lemma~\ref{lem:moves}. Using
     Lemma~\ref{lem:toricqODPdeformations} and the description of the
     global infinitesimal deformations of $(Y,E)$ in terms of local
     data in \S~\ref{sec:some-gener-princ}, we may replace $X$ by the
     toric affine chart at a $0$-stratum of $Z$, so that we are in the
     situation of Lemma~\ref{lem:moves}. The statement then follows
     from Lemma~\ref{lem:moves}.

\end{proof}

\begin{proof}[Proof of Lemma~\ref{lem:moves}]

  \textsc{Step 1: existence} We construct the birational map
  $\Phi\colon \cY \dasharrow \cY^\prime$. We will see that it is a
  log-flip in the sense of the minimal model program and then its
  existence will follow from standard results in the minimal model
  program.

  Note that $Y \rightarrow X$ and $Y' \rightarrow X$ have relative
  Picard rank $1$ and we are assuming that $\Pi$ is small.  Let $\cA$ be
  a $\Pi$-ample divisor on $\cY$. Let $\cH$ be a hyperplane section of
  $\cX$ containing $\pi_*\cA$. Let $\Delta$ be the effective Cartier
  divisor $\epsilon(\pi^*\cH-\cA)$ for $0 < \epsilon \ll 1$,
  $\epsilon \in \QQ$.  Then $(\cY,\Delta)$ is klt and, since $\Pi$ is
  crepant, $-(K_{\cY}+\Delta)$ is $\Pi$-ample. We take
  $\Pi' \colon \cY' \rightarrow \cX$ to be the log terminal model over
  $\cX$ of $\Pi\colon \cY \to \cX$. In the algebraic case it exists
  by~\cite[Theorem~1.2]{BCHM10}; in the complex analytic case it
  exists by~\cite[Theorem~1.2]{2022arXiv220111315F}. Part~(i) also
  follows: since $\Pi$ is small, also $\Pi^\prime$ is small.

  \smallskip

  \textsc{Step 2} By the theorem on formal functions and the formula
$$\cY' = \Proj_{\cX} \left(\bigoplus_{n \ge 0} {\Pi}_*\cO_{\cY}\bigl(n(K_\cY+\Delta)\bigr)\right)$$
the formal completion of $\Pi'\colon \cY^\prime \to \cM$ over $0 \in
\cM$ may be computed from the formal completion of $\Pi$.

\smallskip

\textsc{Step 3: proof for type~I, its inverse, and type~III} The proof
will follow from the:

\begin{claim}
  In the set up of moves of type~I, its inverse, and type~III, let
  $\mathfrak{g}\colon (\mathfrak{Y}, \mathfrak{E}) \to \AA^r$ be the simultaneous resolution of the
  local Altmann deformation from
  Proposition~\ref{pro:admissible_decomposition_Matsushita}.

  The family $\mathfrak{g}\colon \mathfrak{Y} \to \AA^r$ is miniversal.
\end{claim}

Let us first prove the claim. It is enough to prove that the
Kodaira--Spencer map for $\mathfrak{g}\colon \mathfrak{Y} \to \AA^r$
is an isomorphism. We have an exact sequence
$$\rH^0(N_E Y) \rightarrow \T^1_{Y,E} \rightarrow \T^1_Y \rightarrow 0,$$
cf. Lemma~\ref{lem:T1(XD)and(X)}.  We have
$N_E Y=\cO_E(E)=\cO_E(-K_Y)$ and
$\T^1_{Y,E}=\rH^0(\nu^*\nu_*\cO_{\Delta}(-K_Y))$ by
Lemma~\ref{lem:toricqODPdeformations}, and the map
$\rH^0(N_E Y) \rightarrow \T^1_{Y,E}$ is the restriction map.
Moreover $E=-K_Y=\pi^*(-K_X)$ is principal; fix a trivialization
$\cO_Y(-K_Y)=\cO_Y$.  We observe that the map
$\rH^0(N_EY) \rightarrow \T^1_{Y,E}=\rH^0(\cO_{\Delta'})$ has image
$\rH^0(\cO_{\Delta}) \subset \rH^0(\cO_{\Delta'})$, and deduce that
the Kodaira--Spencer map is an isomorphism as required.

Parts~(ii--iv) of the statement follow easily from versality of the
family $Y\subset \mathfrak{Y} \to (0\in \AA^r)$. Indeed for this family the
morphism $\mathfrak{Y}^\prime \to \mathfrak{X}$ is the
simultaneous resolution of the same local Altmann family corresponding
to the second arrangement. In particular it is clear from the statement
of Proposition~\ref{pro:admissible_decomposition_Matsushita} that
$\mathfrak{Y}^\prime \to (0\in \AA^r)$ is a deformation of
$Y^\prime$. The statement follows from Step~2 and the fact that the
formal completions of $\cY\to (0\in \cM)$ and $\cY^\prime \to (0\in \cM)$
are pulled back from $\mathfrak{Y} \to (0\in \AA^r)$ and
$\mathfrak{Y}^\prime \to (0\in \AA^r)$.

Part~(v) is not difficult to prove: the induced birational map
$Y\dasharrow Y^\prime$ is an isomorphism outside the unique
exceptional divisor(s), which are mapped to the $0$-stratum $x\in
X$. The Kodaira--Spencer map is constant in the internal curves and
the claimed effect is the only solution that is consistent with
Lemma~\ref{lem:toricqODPdeformations}.

\textsc{Step 4: proof for type~II and type~IV} The proof follows the outline of Step 3.
For example in the case of moves of type~II a miniversal deformation family of
$(Y,E)$ is constructed explicitly in Example~\ref{exa:uz2}. 
The base of the miniversal family is the infinite dimensional ind-scheme $\AA^\infty = \varinjlim \AA^n$.
The miniversal property works for deformations over finite dimensional
$\CC$-algebras $A$; in practice this means that every deformation over
$\Spec A$ is the pull-back of a deformation $(\mathfrak{Y}^n, \mathfrak{E}^n)\to \AA^n$
via a morphism $\Spec A \to \AA^n$, for some finite $n$.

In Claim~\ref{claim:defo_uz2} miniversal deformation families
$(\mathfrak{Y}, \mathfrak{E})\to \AA^\infty$ of $(Y,E)$ and 
$(\mathfrak{Y}^\prime, \mathfrak{E}^\prime)\to \AA^\infty$ of
$(Y^\prime, E^\prime)$ are constructed, and the birational map
$\mathfrak{Y}\dasharrow
\mathfrak{Y}^\prime$ is also constructed. Parts~(ii--iv) follow; for Part~(v) see
Part~(iv) of Claim~\ref{claim:defo_uz2}.

The case of moves of type~IV is similar and easier. 
\end{proof}

\section{Proof of \texorpdfstring{Theorem~\ref{thm:topology}}{Theorem 1}}
\label{sec:proof-theorem-2}

We work in the analytic category throughout this section.

\begin{lem}
  \label{lem:2}
Let $P$ be a $3$-dimensional reflexive polytope endowed with amd,
$(X,D)$ the corresponding Fano toric pair and $\pi \colon (Y,E)\to
(X,D)$ the induced partial resolution.

 Let $0\in \DD$ be a small analytic disk and $g\colon \cY \to \DD$ a smoothing of $Y$. For all
  $0<|t|<\!\!< 1$ $\rH^2(\cY_t; \QQ)=\rH^2(Y; \QQ)$.
\end{lem}

\begin{proof}
  The Milnor fibre for the smoothing of a qODP deformation retracts onto a lens space $S^3/ _{\ZZ/a\ZZ}$.
\end{proof}

We prove the following more precise version of Theorem~\ref{thm:topology}.

\begin{thm}
  \label{thm:topologyII}
Let $P$ be a $3$-dimensional reflexive polytope endowed with amd,
$(X,D)$ the corresponding Fano toric pair and $\pi \colon (Y,E)\to
(X,D)$ the induced partial resolution.

Choose a $1$-parameter deformation of $(Y,E)$ over a small
analytic disk $0\in \DD$ with generic Kodaira--Spencer class, and
consider the blow-down deformation
  \[
    \xymatrix{ (Y,E)\ar[dd]\ar[dr]^\pi\ar@{^{(}->}[rrr] &   & & (\cY,\cE)\ar[dd]_(.3){g}\ar[dr]^\Pi & \\
                        &(X,D)\ar[dl] \ar@{^{(}-}[rr]&  & \ar@{->}[r]& (\cX,\cD)\ar[dl]^f\\
                        \{0\}  \ar@{^{(}->}[rrr]& & &  \DD &}
  \]

  \smallskip
  
  Note that it follows from the matching condition along dull edges
  that the Kodaira--Spencer class
  $s\in \T^1_{Y,E}=\rH^0\bigl(\Delta^\prime, \nu^\star (-K_Y)\bigr)$ of
  the deformation enjoys the following property:

   For all edges $e=[v,w]\leq P$ and all $1\leq i< j \leq \ell_e$,
    writing
    \[s_i=s\vert_{\Delta^e_i}\in \rH^0\left(\Delta^e_i, -K_Y\vert_{\Delta^e_i}\right)=
      \rH^0\left(\Gamma^e, -K_X \vert_{\Gamma^e} \right)\]
    $s_i-s_j$ has simple zeros in $\Gm =\Gamma^e\setminus
    \{x^v,x^w\}$. 

  Then for all $0<|t|<\!\!< 1$, $\cX_t$ has precisely one ODP for
  all $e=[v, w]$, $x\in
  \Gamma^e\setminus \{x^v,x^w\}$, and $1\leq i<j\leq \ell_e$ such that
  $s_i(x)=s_j(x)$, located near $x\in \Gamma^e$, and it is smooth
  everywhere else.
  
  Above all ODP in $\cX_t$, $\cY_t$ has precisely one rational curve with
  normal bundle $\cO(-1)\oplus\cO(-1)$ in the homology class of
$\sum_{i<k\leq j} C_k^e$ in $\rH^2(\cY_t;\QQ)=\rH^2(Y;\QQ)$.
\end{thm}

\begin{rem}
  \label{rem:ODP}
  \begin{enumerate}
  \item In particular, the theorem states that $\cX_t$ has ordinary double
  points: this finishes the proof of Theorem~\ref{thm:1}.
\item The statement indeed is a more precise version of
  Theorem~\ref{thm:1}. Indeed it is a simple exercise to compute the
  number of singular points on a general fibre.
  \end{enumerate}
\end{rem}

The proof, given in \S~\ref{sec:end-proof-theorem}, will be a (simple)
consequence of the following:

\begin{lem}
\label{lem:topology} 
Let $P$ be a $3$-dimensional reflexive polytope $P$ endowed with amd,
$(X,D)$ the corresponding Fano toric pair and
$\pi \colon (Y,E)\to (X,D)$ the induced partial resolution.

Consider a miniversal complex analytic deformation of $(Y,E)$ over an
analytic germ $0\in \cM$ and its blow-down:
  \[
    \xymatrix{ (Y,E)\ar[dd]\ar[dr]^\pi\ar@{^{(}->}[rrr] &   & & (\cY,\cE)\ar[dd]_(.3){g}\ar[dr]^\Pi & \\
                        &(X,D)\ar[dl] \ar@{^{(}-}[rr]&  & \ar@{->}[r]& (\cX,\cD)\ar[dl]^f\\
                        \{0\}  \ar@{^{(}->}[rrr]& & &  \cM &}
  \]

For $x\in X$ and $0<\varepsilon$, denote by $B_x(\varepsilon) \subset \cX$ the open ball with centre at $x$ and radius
$\varepsilon$.

Then for all $x^v\in X$ in the $0$-skeleton there exists
$0<\varepsilon$ such that for all sufficiently general\footnote{that
  is, outside a complex analytic strict subset $W\subsetneq \cM$} $m\in \cM$, $\cX_m\cap
B_{x^v}(\varepsilon)$ is smooth.
\end{lem}

\begin{proof}
Assume for a contradiction that the conclusion fails. Then there
exists a flat projective family of curves
$\psi \colon \cC \to z_0\in Z$ over an analytic germ $z_0\in Z$ and a
commutative diagram
 \[
   \xymatrix{
  \cC \ar[d]_\psi \ar[r]^\iota  & \cY \ar[d]^\Pi\\
   Z\ar[r]^x \ar[dr]_m & \cX \ar[d]^f  \\
      & \cM }
 \]
 such that
 \begin{enumerate}[(i)]
 \item For all $z\in Z$, $\iota \colon \cC_z\hookrightarrow
   \cY_{m(z)}$ is the inclusion of a subscheme. Put in other words,
   $\cC/Z$ is a family of curves in $\cY/\cM$.
 \item $m\colon Z \to \cM$ is proper and surjective.
 \item $x(z_0)=x^v$.
 \end{enumerate}

 Note that, because the diagram is commutative, for all $z\in Z$ the
 curve $\cC_z\in \cY_{m(z)}$ is contracted by the map $\Pi$ to the
 point $x(z)\in \cX_{m(z)}$.

Let $g_v\colon (\cY_v,\cE_v)\to S_v$ be the simultaneous resolution of
the global Altmann deformation; in other words, the complex analytic deformation whose
existence was established in Proposition~\ref{pro:TdefoII}.

 Because $g\colon (\cY,\cE) \to 0\in \cM$ is a miniversal deformation,
 there is a diagram where both squares are Cartesian
 \[
   \xymatrix{
     (\cY_v,\cE_v) \ar[d]_{\Pi_v} \ar[r] & (\cY, \cE) \ar[d]^\Pi \\
     (\cX_v,\cD_v) \ar[d]_{f_v} \ar[r] & (\cX, \cD) \ar[d]^f \\
   S_v \ar[r]_\mu & \cM }
\]

By Proposition~\ref{pro:singularities_general_fibre_Altmann}, for
general $t\in S_v$, for all $z\in Z$ such that $m=m(z)=\mu(t)$, the
singular locus of $\cX_m$ near $x^v$ consists of a disjoint union of
curves with transverse type $\text{A}$ singularities and the curve
$\cC_z\subset \cY_m \to \cX_m$ is a connected union of irreducible components of a fibre of the minimal resolution
along one of these $\text{A}$-curves.

For the rest of the proof we restrict all our families and diagrams to an affine
neighbourhood of $x^v\in X$. By Theorem~\ref{thm:local_invariance},
there exists a birational transformation over $\cX$:
\[
  \xymatrix{
    (\cY, \cE) \ar@{-->}[rr]^\Phi\ar[dr]_\Pi& & (\cY^\prime,
    \cE^\prime) \ar[dl]^{\Pi^\prime}\\
     & (\cX,\cD)&     }
 \]
 Inducing new families
    \[
    \xymatrix{ (Y^\prime,E^\prime)\ar[dd]\ar[dr]^{\pi^\prime}\ar@{^{(}->}[rrr] &   & & (\cY^\prime,\cE^\prime)\ar[dd]_(.3){g^\prime}\ar[dr]^{\Pi^\prime} & \\
                        &(X,D)\ar[dl] \ar@{^{(}-}[rr]&  & \ar@{->}[r]& (\cX,\cD)\ar[dl]^f\\
                        \{0\}  \ar@{^{(}->}[rrr]& & &  \cM &}
                      \quad \text{and} \quad
  \xymatrix{
  \cC^\prime \ar[d]_{\psi^\prime} \ar[r]^{\iota^\prime}  & \cY^\prime \ar[d]^{\Pi^\prime}\\
   Z\ar[r]^x \ar[dr]_m & \cX \ar[d]^f  \\
      & \cM }                      
   \]
   where now the (class of the) exceptional curve $\cC^\prime_z$ is
   the class of the fibre of the surface $(E_j\subset Y^\prime)\to
   (\Gamma\subset X)$ shaded in Fig.~\ref{fig:exc_curve}.
        \begin{figure}[h!]
  \centering
\begin{tikzpicture}
  \draw[thick] (0,0) -- (2,-1) -- (3,-2) 
  (3,-6) -- (2,-7) -- (0,-8)
  (2, -1) -- (6, -1) (3, -2) -- (6,-2) 
  (3,-6) -- (6,-6) (2,-7) -- (6, -7)
  (3,-3.5) -- (3, -4.5) (3, -3.5)--(6,-3.5)
  (3, -4.5) -- (6,-4.5);
  \draw[thick, dashed] (3,-2) -- (3,-3) (3,-5) -- (3, -6);
  \fill[green!20!white] (3,-3.5) -- (6,-3.5) -- (6,-4.5) -- (3, -4.5)
  -- cycle ;
  \node at (4.8,-0.7) {$\lambda + uA_{n+1}(u)$};
  \node at (4.8,-1.7) {$\lambda + uA_{n}(u)$};
  \node at (4.8,-3.2) {$\lambda + uA_{j}(u)$};
  \node at (4.8,-4.2) {$\lambda + uA_{j-1}(u)$};
    \node at (4.8,-6.7) {$\lambda + uA_{0}(u)$};
    \node at (-1,-0.5) {$\lambda + yC(y)$};
  \node at (2.75,-4) {$\lambda$};
  \node at (2,-1.5) {$\lambda$};
  \node at (2,-6.5) {$\lambda$};
  \node at (-1,-7.5) {$\lambda + xB(x)$};
\end{tikzpicture}
\caption{Exceptional curve and tangent vector \texorpdfstring{$\xi \in T_0 \cM=\rH^0\bigl(\Delta^\prime, \nu_\star
\nu^\star (-K_{Y^\prime}) \bigr)$}{in T0(M)}.} 
  \label{fig:exc_curve}
\end{figure}
It follows from this that the curve $\cC^\prime_{z_0}\subset Y^\prime$
is the internal fibre in the surface $E_j$. The figure
also depicts --- following our usual conventions, see e.g. Lemma~\ref{lem:moves} --- a tangent vector $\xi \in T_0 \cM=\rH^0\bigl(\Delta^\prime, \nu_\star
\nu^\star (-K_{Y^\prime}) \bigr)$.  

The birational map $\Phi\colon (\cY,\cE) \dasharrow (\cY^\prime,\cE^\prime)$ is a
composition of moves of type I, their inverses, and type II--IV, see
Definition~\ref{dfn:types}. By Lemma~\ref{lem:moves}, a generic
tangent vector $\xi \in T_0 \cM$ has the property --- see
Fig.~\ref{fig:exc_curve} --- that $A_{j-1}(0)\neq A_j(0)$. The
statement now follows from Lemma~\ref{lem:smoothness}, implying that
the curve $\cC^\prime_{z_0}$ does not deform to a general fibre. 
\end{proof}

\begin{lem}
  \label{lem:smoothness}
  Consider the toric pair $(X,D)$ where
  \begin{align*}
    X & =\text{A}_{n +1}\times \AA^1_u =\left(xy+z^{n+1}=0\right)
        \subset \AA^4_{x,y,z,u} \\
    D & = (uz=0) \subset X
  \end{align*}
  Note that $D$ has three components: $D=D_u+ D_x+D_y\subset X$, where
  $D_u=(u=0)$, $D_x=(z=y=0)$ and
  $D_y=(z=x=0)$.

  Let $\pi \colon Y\to X$ be the minimal resolution with chain of
  exceptional divisors $E_1,\dots, E_n$ and strict transforms
  $F=D_u^\prime$, $E_0=D_x^\prime$, $E_{n+1}=D_y^\prime$ labelled as pictured in
  Fig.~\ref{fig:KS_cA}. 

  Consider a $1$-parameter analytic deformation of $(Y,E)$ over a
  small disk and its blow-down:
     \[
    \xymatrix{ (Y,E)\ar[dd]\ar[dr]^\pi\ar@{^{(}->}[rrr] &   & & (\cY,\cE)\ar[dd]_(.3){g}\ar[dr]^\Pi & \\
                        &(X,D)\ar[dl] \ar@{^{(}-}[rr]&  & \ar@{->}[r]& (\cX,\cD)\ar[dl]^f\\
                        \{0\}  \ar@{^{(}->}[rrr]& & &  \DD &}
                    \]

Assume that the Kodaira--Spencer class $\xi \in \T^1_{Y,E}$ of the deformation, in the
notation pictured in Fig.~\ref{fig:KS_cA} with $A_j(u)\in
\CC[u]$, $B(x)\in \CC[x]$, $C(y)\in \CC[y]$, is such that the
$A_j(0)\in \CC$ are pairwise distinct. 
 
Then for all small enough $0\neq t \in \DD$ the fibre $\cX_t$ is
smooth.
\end{lem}

\begin{proof}

  Lemma~\ref{lem:T1(XD)and(X)} gives an exact sequence:
  \[
    \rH^0\left(E, N_E Y\right) \to \T^1_{Y,E} \to
    \T^1_Y \to (0) 
  \]
  From this, using Fig.~\ref{fig:KS_cA} to picture a vector $\xi \in
  \T^1_{Y,E}$, we see that
  \begin{enumerate}[(a)]
  \item   $\T^1_Y=\CC[u]^{n+1}/\CC[u]$, where $\CC[u]$
  acts by translation:
  \[
A(u) \colon \left(A_1(u),\dots, A_{n+1}(u)\right) \mapsto
\left(A(u)+A_1(u),\dots, A(u)+A_{n+1}(u)\right) 
\]
 and 
\item The map $\T^1_{Y,E}\to \T^1_Y$ takes a vector $\xi\in \T^1_{Y,E}$ as pictured in
  Fig.~\ref{fig:KS_cA} to the equivalence class of the tuple
  $\left(A_1(u),\dots, A_{n+1}(u)\right)$.
  \end{enumerate}

  The result follows from the following:
  \begin{claim}
    Let $g\colon \cY \to \DD$ be a deformation of $Y$ over a small
    disk and $f\colon \cX\to \DD$ its blow-down. If the
    Kodaira--Spencer class $\xi \in \T^1_Y$ is (the class of) a tuple
    $\left(A_1(u),\dots, A_{n+1}(u)\right)$ where the $A_j(0)\in \CC$
    are pairwise distinct, then for all $0<t\in \DD$ small enough
    $\cX_t$ is smooth.
  \end{claim}

  To show the claim, note that any $1$-parameter family $g\colon
  \cY\to \DD$ with Kodaira--Spencer class $\xi$ blows down to a
  deformation $f\colon \cX \to \DD$ given by an equation:
  \[
\bigl[ xy=\prod_{i=1}^{n+1}(z-tA_i(u)) \quad \text{mod} \; t^2 \bigr] \subset \AA^4_{x,y,z,u}
\times \DD_t
  \]
and smoothness is easily checked with the Jacobian criterion.
\end{proof}

\subsection{Proof of Theorem~\ref{thm:topologyII}}
\label{sec:end-proof-theorem}

The singular fibres occur next to the points in the $1$-strata
$\Gm=\Gamma^e\setminus \{x^v,x^w\}$ where for some $i,j$ $s_i-s_j=0$.

The formula for the homology class is an exercise starting from the
construction of the simultaneous resolution $\Pi \colon \cY \to \cX$ in Example~\ref{exa:cA}.

\appendix

\section{The Cayley polytope, mixed subdivisions, and tropical arrangements}
\label{sec:Cayley}

The notions below were (to our knowledge) first introduced and studied
in \cite{sturmfels_newton_polytope_resultant} and in
\cite{huber_rambau_santos}. For a self-contained treatment we
recommend~\cite[\S9.2]{book_triangulations} and~\cite[\S4.1]{Joswig}.

\subsection{Conventions concerning polyhedra and polytopes}
\label{sec:conv-conc-polyh}

Fix a real vector space $V$ of finite dimension.  A \emph{cone} is a
non-empty subset of $V$ which is closed under sum and multiplication
by non-negative real numbers.  If $S$ is a subset of $V$, then we
denote by $\langle S \rangle_+$ the conical hull of $S$, i.e.\ the smallest cone
containing $S$.  A subset of $V$ is a \emph{polyhedral cone} if it
coincides with the conical hull of a finite set, or equivalently if it
is the intersection of finitely many closed half-spaces whose boundary
hyperplanes contain the origin.

A \emph{polyhedron} is the intersection of finitely many closed
half-spaces, so it is always convex and closed.  There is the obvious
notion of \emph{face} of a polyhedron. The $0$-dimensional faces are
called \emph{vertices} and the top-dimensional ones \emph{facets}.  A
compact polyhedron is called \emph{polytope}; equivalently, a polytope
is the convex hull of a finite set.  The convex hull of a subset
$S\subset V$ is denoted by $\conv{S}$.

Now fix a lattice $L$ and consider the real vector space
$L_\RR := L \otimes_\ZZ \RR$.  A cone in $L_\RR$ is called
\emph{rational} if it is the conical hull of finitely many points in
$L$.  A \emph{lattice polytope} in $L$ is a subset of $L_\RR$ which is
the convex hull of finitely many points in $L$.

\begin{dfn}
  \label{dfn:polyhedral_subdivision}
  A \emph{polyhedral subdivision} of a polytope $F$ is a finite
  collection $\cS$ of polytopes with the following properties: every
  face of an element of $\cS$ is an element of $\cS$, the intersection
  of two elements of $\cS$ is either empty or a common face, and the
  union of the elements in $\cS$ is $F$. The elements of $\cS$ are
  called \emph{cells} of the subdivision.  If the cells are simplices,
  then the subdivision is called a \emph{triangulation}.  If $F$ is a
  lattice polytope in a lattice $L$, then a \emph{lattice polyhedral
    subdivision} of $F$ is a polyhedral subdivision whose cells are
  lattice polytopes.  A polyhedral subdivision of a polytope $F$ is
  called \emph{regular} it its cells are the domains of linearity of a
  convex piecewise affine-linear continuous function
  $\phiv \colon F \to \RR$.
\end{dfn}

\subsection{The Cayley construction}

\begin{dfn}
    \label{dfn:cayley_polytope}
    Let $F$ be a lattice polytope in a lattice $L$ and let
    $m=(F = F_1 + \cdots + F_k)$ 
    a Minkowski decomposition. The
    \emph{Cayley polytope} of $m$ is the
    lattice polytope 
    \begin{equation*}
    \cayley{m} = \conv{F_1 + e_1, \dots, F_k + e_k}
    \end{equation*}
        in the lattice
    $L \oplus \ZZ^k$, where $e_1, \dots, e_k$ are the vectors of the standard basis of $\ZZ^k$. 
\end{dfn}

By construction, $\cayley{m}$ is contained in the affine hyperplane
$x_1 + \cdots + x_k = 1$ of $L \oplus \ZZ^k$, where $x_1, \dots, x_k$
are the standard coordinates of $\ZZ^k$.

\begin{rem}
 The Cayley polytope does not have interior lattice points; more precisely
        \begin{equation*}
        \cayley{m} \cap (L \oplus \ZZ^k ) = \bigcup_{i=1}^k (F_i \cap L) \times \{ e_i \}.
        \end{equation*}
The polytope $F$ is an affine linear slice of the Cayley
        polytope: indeed, the map
        \begin{equation} \label{eq:cayley_map}
\iota \colon  F \to \cayley{m} \cap \left( L_\RR \times \left\{ \left( \frac{1}{k},
       \dots, \frac{1}{k} \right)\right\} \right)
        \end{equation}
        defined by $\iota (v) =\frac{1}{k} (v + e_1 + \cdots + e_k)$ is
        bijective.
        Notice that this map does not respect the lattice structure. 
\end{rem}

\begin{dfn}
   Let $F$ be a lattice polytope in a lattice $L$ and $m=(F = F_1 + \cdots + F_k)$ 
  a Minkowski decomposition.

  A cell $\widetilde{C}\subset \cayley{m}$ is necessarily the
  convex hull
  \[
\widetilde{C} = [C_j\times\{e_j\}\mid j=1,\dots, k]
\]
 of cells $C_j \times \{e_j\} \subset F_j\times \{e_j\}$, and then $C=\iota^{-1}(\widetilde{C}) =\sum C_j\subset F$ is a
  cell.
  \begin{enumerate}[(a)]
  \item   A \emph{mixed cell} is a
  cell $C\subset F$ of this form, together with a choice, for all
  $j$, of a cell $C_j\subset F_j$ such that $C=\sum
  C_j$.\footnote{This choice is not unique, and this is why we need to
  build it into the definition.}  
\item A \emph{mixed subdivision} of $F$ is a subdivision of $F$
  consisting of the mixed cells induced by the cells of a lattice
  polyhedral subdivision of $\cayley{m}$. In this case we say that the
  subdivision is \emph{subordinated} to $m$. 
\item A lattice triangulation $\cT$ of $\cayley{m}$ is \emph{fine} if
  every lattice point of $\cayley{m}$ is a $0$-stratum of some simplex
  of $\cT$. A mixed subdivision of $F$ subordinated to $m$ is
  \emph{fine} if it is induced by a fine lattice triangulation $\cT$
  of $\cayley{m}$. In other words, a mixed subdivision subordinated to
  $m$ is fine if and only if it has no non-trivial refinement among
  mixed subdivisions subordinated to $m$.
\item A mixed subdivision of $F$ subordinated to $m$ is \emph{coherent} if
it is induced by a regular lattice polyhedral subdivision
triangulation of $\cayley{m}$.
\end{enumerate}
\end{dfn}

\begin{rem}
  It obvious that a coherent mixed subdivision is itself regular.
\end{rem}

The following statement is a synthesis of (for example)~\cite[Theorem~1.4]{MR2134766},
\cite[Theorem~1.11]{Joswig},~\cite[Corollary~4.10]{Joswig},
\cite[\S~3, \S~4]{MR3751882}.

\begin{thm}
	\label{thm:cayley_trick}
	Let $F$ be a lattice polygon in a rank-$2$ lattice $N$, and $m
        = (F= F_1 + \cdots + F_k)$ a Minkowski decomposition of $F$.

        The
	following three sets are in (compatible) $1$-to-$1$
        correspondence:
	\begin{enumerate}[(i)]
		\item The set of coherent mixed subdivisions of $F$ subordinated to $m$;
		\item The set of regular subdivisions of the polytope $\cayley{m}$;
		\item The set of arrangements of tropical curves
		$\Trop (f_1)$, \dots, $\Trop (f_k)$, where $f_i$ is a tropical
		polynomial with Newton polytope $F_i$.
                    \end{enumerate}

        In this correspondence, fine mixed subdivision of $F$
        correspond to fine triangulations and to generic tropical arrangements.
      \end{thm}

      \begin{proof} The correspondence between mixed subdivisions of
        $F$ and subdivisions of $\cayley{m}$ is an elementary (almost
        tautologous), well-known and 
        well-explained fact, see for example
        in~\cite[\S~9.2]{book_triangulations}
        and~\cite{MR2134766}.

        We briefly sketch the correspondence with tropical
        arrangements (it also is an elementary fact).  
        Consider the tropical semiring $\TT$, that is, $\RR\cup
        \{\infty\}$ with operations $\min, +$:
        \[
x \oplus y = \min \{x,y\} \quad \text{and} \quad x \odot y = x+y
\]
A tropical polynomial
\[
  f(x) = \bigoplus_{\nu\in N} c_\nu \odot x^{\odot \nu} \in \TT[N]
\]
(where almost all $c_\nu =\infty$) evaluates at $a\in M_\RR=\Hom (N,\RR)$ as
\[
f(a) = \min \{c_\nu +\langle a,\nu\rangle \mid \nu \in N\}
\]
The tropical hypersurface defined by $f$ is the subset
$\Trop (f)\subset M_\RR$ consisting of those $a\in M_\RR$ where the
minimum is attained at least twice.

It is almost obvious that
\[
  \Trop (f\odot g)=\Trop (f) \cup \Trop (g) \quad \text{and}
  \Newt (f\odot g)=\Newt (f) + \Newt (g)
\]

A tropical polynomial $f$ induces a regular subdivision of its Newton
polyhedron $F=\Newt (f)$ defined as follows.
The \emph{extended Newton polyhedron} $\widetilde{F} \subset N_\RR \oplus
\RR$ --- defined to be the convex hull of the $\{(\nu,c_\nu+p) \mid
p\geq 0, \nu \in N\}$ --- is the supergraph of a PL function $\varphi\colon F \to
\RR$: the regular subdivision is the one whose maximal cells are the domains of
linearity of $\varphi$. We
denote this subdivision by $\mathcal{S} (f)$. It is an elementary
fact~\cite[Theorem~1.1]{Joswig} that $\mathcal{S} (f)$ is dual to the
tropical hypersurface $\Trop (f)$.

The key fact is this: the union of two tropical hypersurfaces is dual
to the regular mixed subdivision derived from the regular subdivisions
dual to the two hypersurfaces, see~\cite[\S~3, \S~4]{MR3751882}.
\end{proof}

\section{Deformations of pairs}
\label{sec:appendix}

In this section we develop the general theory of deformations of pairs
$(X,D)$ of a scheme $X$ and effective Cartier divisor $D$ on
$X$. This is basically a long exercise in deformation theory. The
results must be well-known to many algebraic geometers but,
surprisingly given the vast amounts of existing literature on
deformation theory, we could not find them written up anywhere. The
first \S~\ref{dfn:Q_L} treats deformations of pairs $(X,L)$ of a
scheme and line bundle; in \S~\ref{sec:defo(X,B)} we discuss
deformations of pairs $(X,D)$ of a scheme and Cartier divisor.

\smallskip

We work over a fixed base field $k$.  If $S$
is an algebraic stack of finite type over $k$, $\LL_S$ denotes the
cotangent complex of $S$ over $k$, $\Omega_S$ denotes the cotangent
sheaf of $S$ over $k$, i.e., $\Omega_S = \cH^0(\LL_S)$, and $T_S$
denotes the tangent sheaf of $S$ over $k$, i.e.,
$T_S = \cHom(\Omega_S, \cO_S)$.  Similarly, we use the notation
$\LL_f$, $\Omega_f$, $T_f$ for a morphism $f$ between two algebraic
stacks of finite type over $k$.

\subsection{Deformations of a scheme together with a line bundle}

Here we fix a separated scheme $X$ of finite type over the field $k$ and a line bundle $L$ on $X$.

\begin{nota} \label{dfn:Q_L} Write
  $Q_L = P^1(L) \otimes_{\cO_X} L^\vee$, where $P^1(L)$ is the first
  sheaf of principal parts of $L$.
\end{nota}

We refer the reader to \cite[Definition~16.3.1]{ega4-4} for the definition of the sheaves of principal parts.
There is a natural short exact sequence of coherent sheaves on $X$
\begin{equation} \label{eq:short_exact_sequence_atiyah_class}
(0) \to \Omega_X \to Q_L \to \cO_X \to (0)
\end{equation}
that is associated to the class in
$\Ext^1(\cO_X,\Omega_X) = \rH^1(X, \Omega_X)$ that is the image of
the class of $L$ in $\rH^1(X, \cO_X^*)$ via the homomorphism of
abelian groups $\rH^1(X, \cO_X^*) \to \rH^1(X, \Omega_X)$
induced on first cohomology by the homomorphism of sheaves of abelian
groups $\mathrm{dlog} \colon \cO_X^* \to \Omega_X$ given by
$g \mapsto g^{-1} \rmd g$.  We refer the reader to
\cite[\S3]{altmann_christophersen} and to
\cite{wang_deformations_XL}. The short exact
sequence~\eqref{eq:short_exact_sequence_atiyah_class} is called the
\emph{Atiyah class} of $L$, see \cite[IV.2.3]{illusie_1}.

\begin{rem} \label{rem:explicit_description_sheaf_Q_L}
It is possible to give an explicit description of the sheaf $Q_L$.
Fix an affine open cover $\cU = \{ U_i \}_i$ that trivialises the line bundle $L$.
Set $U_{ij} = U_i \cap U_j$.
Let $g_{ij} \in \cO_X(U_{ij})^*$ be transition functions of $L$ over $\cU$, i.e.\ the diagram
\begin{equation*}
	\xymatrix{
	& L \vert_{U_{ij}} \\
	\cO_{U_{ij}} \ar[ru]^{ s_i \vert_{U_{ij}}   } & & \cO_{U_{ij}} \ar[lu]_{  s_j \vert_{U_{ij}}  } \ar[ll]^{ \cdot g_{ij}   }
	}
\end{equation*}
commutes, where $s_i$ (resp.\ $s_j$) is a nowhere vanishing section of
$L$ over $U_i$ (resp. $U_j$). In other words,
$s_j \vert_{U_{ij}} = g_{ij} s_i \vert_{U_{ij}}$.  Then the line
bundle is associated to the cohomology class
$[\{ g_{ij} \} ] \in \check{\rH}^1(\cU, \cO_X^*) \subseteq
\rH^1(X,\cO_X^*)$.

The sheaf $Q_L$ is obtained by choosing isomorphisms
$\phiv_i \colon Q_L \vert_{U_i} \to \Omega_{U_i} \oplus \cO_{U_i}$ and
on the pairwise intersections we have the following commutative
diagram of isomorphisms.
\[
	\xymatrix{
	& Q_L \vert_{U_{ij}} \ar[ld]_{\phiv_i \vert_{U_{ij}} } \ar[rd]^{\phiv_j  \vert_{U_{ij}} } \\
	\Omega_{U_{ij}} \oplus \cO_{U_{ij}}  \ar[rr]_{ \begin{pmatrix}
			\mathrm{id}_{\Omega_{U_{ij}}} &  \frac{\rmd g_{ij}}{g_{ij}} \\
			0 & \mathrm{id}_{\cO_{U_{ij}}}
	\end{pmatrix}  } & & \Omega_{U_{ij}} \oplus \cO_{U_{ij}}
}
\]

\end{rem}

\begin{rem}
  The sheaf $Q_{L^\vee}$ associated to the dual line bundle $L^\vee$
  is isomorphic to $Q_L$, but the short exact sequence
  \eqref{eq:short_exact_sequence_atiyah_class} for $L^\vee$ differs
  from the one for $L$ by a minus sign in one of the two maps.
\end{rem}

\begin{rem}
	If $X$ is smooth, then the dual sequence of \eqref{eq:short_exact_sequence_atiyah_class} is
	\[
	(0) \to \cO_X \to \cD^1(L) \to T_X \to (0)
	\]
	where $\cD^1(L)$ is the sheaf of first order differential
        operators of $L$ and the homomorphism $\cD^1(L) \to T_X$ is
        the symbol map.
\end{rem}

We next see that the sheaf $Q_L$ controls the deformations of the pair $(X,L)$.
We also consider the forgetful map $\Deff{X}{L} \to \Def{X}$.

\begin{pro} \label{pro:deformations_X_L} Let $X$ be a scheme of finite
  type over the field $k$ and let $L$ be an invertible sheaf on
  $X$. Then:
	\begin{enumerate}[(i)]
		\item There exists an exact sequence
		\begin{equation} \label{eq:exact_sequence_tangent_deformation_X_L}
				\rH^1(X, \cO_X) \to \rT \Deff{X}{L} \to \rT \Def{X} \to \rH^2(X, \cO_X)
		\end{equation}
		\item If $\rH^2(X,\cO_X) = (0)$, then the forgetful
                  map $\Deff{X}{L} \to \Def{X}$ is smooth.
		\item If $X$ is reduced and $k$ is perfect, then $\Ext^1(Q_L, \cO_X)$ is the tangent space of $\Deff{X}{L}$ and the exact sequence~\eqref{eq:exact_sequence_tangent_deformation_X_L} is obtained by taking the cohomology long exact sequence of \eqref{eq:short_exact_sequence_atiyah_class}.
		\item If $X$ is normal and $k$ is perfect, then $\Ext^2(Q_L, \cO_X)$ is an obstruction space of $\Deff{X}{L}$.
	\end{enumerate}	
\end{pro}

Before giving the proof, we prove a homological algebra lemma and we make a remark on the cotangent complex of a smooth quotient stack.

\begin{lem} \label{lem:homological_algebra}
	Let $n \geq 0$ be an integer.
	Let $X$ be a noetherian scheme and let $Z \subseteq X$ be a closed subset.
	Let $A$ be a complex of $\cO_X$-modules with coherent cohomology such that:
	\begin{enumerate}
		\item for all $i>0$ $\cH^i(A) = (0)$,
		\item for all $i<0$ $\cH^i(A)$ is  set-theoretically supported on a subset of $Z$.
	\end{enumerate}
	If for every scheme-theoretic point $x \in Z$ we have
        $\depth \cO_{X,x} \geq n$, then for all $0 \leq j \leq n$ the map
	\begin{gather*}
	\Ext^j(\cH^0(A), \cO_X) \to \Ext^j(A, \cO_X)
	\end{gather*}
	induced by the truncation map $A \to \tau_{\geq 0} A$ and the
        quasi-isomorphism $\cH^0(A) \to \tau_{\geq 0} A$, is an isomorphism.
\end{lem}

\begin{rem}
	In Lemma~\ref{lem:homological_algebra} the assumption on the depth of the stalks of the sheaf $\cO_X$ at the points of $Z$ implies that every irreducible component of $Z$ has codimension $\geq n$ in $X$.
	Conversely, if $X$ is Cohen--Macaulay and every irreducible component of $Z$ has codimension $\geq n$ in $X$ then the assumption on the depth of $\cO_X$ at points of $Z$ is automatically satisfied.
\end{rem}

\begin{proof}[Proof of Lemma~\ref{lem:homological_algebra}]
  By \cite[Proposition~3.7]{hartshorne_local_cohomology} we get that
  for all $i \neq 0$ and $q<n$ $\cExt^q(\cH^i(A), \cO_X) = (0)$.  By
  the spectral sequence
	\[
	E_2^{p,q} = \rH^p (\cExt^q(\cH^i(A), \cO_X)) \Rightarrow \Ext^{p+q}(\cH^i(A), \cO_X)
      \]
      we get the vanishing
	\begin{equation} \label{eq:vanishing_exts} \text{ for all } i
          \neq 0 \text{ and } j<n \quad \Ext^j(\cH^i(A), \cO_X) = (0)
	\end{equation}
	Let $I$ be an injective resolution of $\cO_X$.  Consider the
        double complex $K$ given by $K^{i,j} = \Hom(A^{-i}, I^j)$ that
        computes $\Ext^\bullet(A,\cO_X)$.  The double complex $K$
        gives two spectral sequences; one of these has first page
        equal to $' \! E_1^{i,j} = \Hom(\cH^{-i}(A), I^j)$, and second
        page equal to $' \! E_2^{i,j} = \Ext^j(\cH^{-i}(A), \cO_X)$.
        The vanishing in \eqref{eq:vanishing_exts} implies that, for
        all $0 \leq j \leq n$,
        $' \! E_2^{0,j} = \Ext^j(\cH^0(A), \cO_X)$ is the unique
        non-zero term in $' \! E_2$ with total degree $j$. Moreover,
        for all $0 \leq j \leq n$, the differentials in $' \! E_2$
        starting from and arriving in $' \! E_2^{0,j}$ are zero. This
        implies that, for all $0 \leq j \leq n$, $' \! E_2^{0,j} $ is
        isomorphic to the limit $\Ext^j(A,\cO_X)$.
\end{proof}

\begin{rem} \label{rem:cotangent_quotient_stack} Fix a field $k$ and
  an algebraic group $G$ over $k$ acting on a smooth scheme $Y$ over
  $k$.  Then the cotangent complex $\LL_{[Y/G]}$ of the quotient stack
  $[Y/G]$ is quasi-isomorphic to the following complex of
  $G$-equivariant sheaves on $Y$ concentrated in degrees $0,1$:
	\[
	\Omega_Y \to \frakg^\vee \otimes_k \cO_Y 
	\]
	where $\frakg^\vee$ is the dual of the Lie algebra $\frakg$ of
        $G$ and the differential is dual to the adjoint action on
        tangent vectors.
\end{rem}

\begin{proof}[Proof of Proposition~\ref{pro:deformations_X_L}]
We denote by $\rB \Gm$ the quotient stack $[\Spec k \, / \, \Gm]$ that classifies line bundles, i.e.\ $\Gm$-torsors.
Let $\epsi_L \colon X \to \rB \Gm$ be the classifying morphism of $L$.
The morphisms
\[
X \overset{\epsi_L}\longrightarrow \rB \Gm \to \Spec k
\]
induce a distinguished triangle in the derived category of $X$:
\begin{equation*}
\rL \epsi_L^* \LL_{\rB \Gm} \to \LL_X \to \LL_{\epsi_L} \to \rL \epsi_L^* \LL_{\rB \Gm}[1]
\end{equation*}
By Remark~\ref{rem:cotangent_quotient_stack} we have $\LL_{\rB \Gm} = \cO_{\rB \Gm}[-1]$. Therefore $\rL \epsi_L^* \LL_{\rB \Gm} = \cO_X[-1]$. By shifting we obtain a distinguished triangle
\begin{equation}
  \label{eq:triangle_epsi_L}
\LL_X \to \LL_{\epsi_L} \to \cO_X \to \LL_X[1]
\end{equation}
This is exactly the Atiyah class of the line bundle $L$, see~\cite[IV.2.3]{illusie_1}, so
\begin{equation*}
Q_L = \cH^0(\LL_{\epsi_L})
\end{equation*}
and
\begin{equation}
  \label{eq:cohomology_sheaves_of_L_epsi_L}
\text{for all } i \neq 0\qquad \cH^i(\LL_X) \simeq \cH^i(\LL_{\epsi_L}) 
\end{equation}

Giving an infinitesimal deformation of the scheme $X$ together with the line bundle $L$ is equivalent to giving an infinitesimal deformation of the morphism $\epsi_L$ which is trivial on the target.
Therefore infinitesimal automorphisms, tangent space and obstructions of $\Deff{X}{L}$ are controlled by the cotangent complex $\LL_{\epsi_L}$ of the morphism $\epsi_L$. The distinguished triangle~\eqref{eq:triangle_epsi_L} induces the exact sequence
\begin{align*}
  (0) &\to\rH^0(\cO_X)  \to \Ext^0(\LL_{\epsi_L}, \cO_X) \to \Ext^0(\LL_X, \cO_X) \to \\
    &\to \rH^1(\cO_X)  \to  \Ext^1(\LL_{\epsi_L}, \cO_X) \to \Ext^1(\LL_X, \cO_X) \to \\
    &\to \rH^2(\cO_X)  \to  \Ext^2(\LL_{\epsi_L}, \cO_X) \to \Ext^2(\LL_X, \cO_X)
\end{align*}
where $\Ext^1(\LL_{\epsi_L}, \cO_X)$ is the tangent space of $\Deff{X}{L}$ and $\Ext^2(\LL_{\epsi_L}, \cO_X)$ is an obstruction space for $\Deff{X}{L}$. This immediately implies the well known fact that $\rH^2(\cO_X)$ is an obstruction space for the forgetful map
$
\Deff{X}{L} \to \Def{X}
$; this proves (i) and (ii).

Now we need to prove (iii) and (iv). Assume that $k$ is perfect.
Let $Z$ be the complement of the regular locus of $X$. Since $k$ is perfect, $X \setminus Z$ is smooth over $k$ and hence the negative degree cohomology sheaves of the cotangent complex $\LL_X$ are supported on subsets of $Z$. By \eqref{eq:cohomology_sheaves_of_L_epsi_L} also the negative degree cohomology sheaves of $\LL_{\epsi_L}$ are supported on subsets of $Z$.

If $X$ is reduced, then $X$ is $R_0$ (i.e.\ $\codim(Z,X) \geq 1$) and
$S_1$, hence for all $x\in X$
$\depth \cO_{X,x} \geq \min \{ 1, \dim \cO_{X,x} \}$. We apply
Lemma~\ref{lem:homological_algebra} to $\LL_{\epsi_L}$ and we obtain
an isomorphism
$\Ext^1(\LL_{\epsi_L}, \cO_X) \cong \Ext^1(Q_L, \cO_X)$. This proves
(iii).  Part~(iv) is proved in the same way, because if $X$
is normal then $X$ is $R_1$ and $S_2$.
\end{proof}

\subsection{Deformations of a scheme together with an effective Cartier divisor}
\label{sec:defo(X,B)}

Here we fix a separated scheme $X$ of finite type over the field $k$ and an effective Cartier divisor $i \colon D \into X$.

\begin{lem} \label{lem:homomorphism_from_OX-D_to_Q}	There is a natural injective homomorphism of $\cO_X$-modules
	\[
	\cO_X(-D) \into Q_{\cO_X(D)}
	\]
where $Q_{\cO_X(D)}$ is the coherent sheaf on $X$ associated to the line bundle $\cO_X(D)$ in Definition~\ref{dfn:Q_L}.
\end{lem}

\begin{proof}
  We use the notation as in
  Remark~\ref{rem:explicit_description_sheaf_Q_L}: $\cU = \{U_i\}_i$
  is an affine open cover of $X$ which trivialises $\cO_X(D)$, $s_i$
  is a nowhere vanishing section of $L$ on $U_i$, $\{g_{ij}\}$ are the
  transition functions of $\cO_X(D)$.  Let $s \in \rH^0(X, \cO_X(D))$
  be the global section which defines $D$.  Let $f_i \in \cO_X(U_i)$
  be such that $s \vert_{U_i} = f_i s_i$. Then
  $f_i \vert_{U_{ij}} = g_{ij} f_j \vert_{U_{ij}}$.
 
  The homomorphism $\cO_X(-D) \to Q_{\cO_X(D)}$ in the statement is
  given by the collection of homomorphisms
  $\cO_{U_i} \to \Omega_{U_i} \oplus \cO_{U_i}$ which maps
  $1 \in \cO_{U_i}$ to
\[
\begin{pmatrix}
\rmd f_i \\ -f_i
\end{pmatrix} \in \Omega_{U_i} \oplus \cO_{U_i}
\]
These homomorphisms glue together because the diagram
\[
\xymatrix{
\cO_{U_{ij}}
\ar[rr]^{\begin{pmatrix}	\rmd f_i \\ -f_i	\end{pmatrix}}  
   & & \Omega_{U_{ij}} \oplus \cO_{U_{ij}}  
   \ar[d]^{\begin{pmatrix}	\mathrm{id}_{\Omega_{U_{ij}}} &  \frac{\rmd g_{ij}}{g_{ij}} \\ 0 & \mathrm{id}_{\cO_{U_{ij}}}	\end{pmatrix}} 
    \\
\cO_{U_{ij}}
\ar[rr]_{\begin{pmatrix}	\rmd f_j \\ -f_j	\end{pmatrix}}
\ar[u]^{g_{ij}^{-1}} 
&   & \Omega_{U_{ij}} \oplus \cO_{U_{ij}}
 \\
}
\]
commutes. The homomorphism is injective because the $f_i$'s are not
zero divisors, since $D$ is a Cartier divisor.
\end{proof}

\begin{dfn} \label{dfn:logarithmic_differential}
We define $\Omega_X(\log D)$ to be the cokernel of the homomorphism $\cO_X(-D) \to Q_{\cO_X(D)}$ constructed in Lemma~\ref{lem:homomorphism_from_OX-D_to_Q}.
\end{dfn}

It is immediate that $\Omega_X(\log D)$ is a coherent sheaf on $X$. We
have a short exact sequence
\begin{equation} \label{eq:definition_Omega_log}
	(0) \longrightarrow \cO_X(-D) \longrightarrow Q_{\cO_X(D)}
        \longrightarrow \Omega_X(\log D) \longrightarrow (0) 
\end{equation}

\begin{rem}
	One can see that $\Omega_X(\log D)$ is the sheaf of logarithmic differentials of $X$ equipped with logarithmic structure (in the Zariski topology) induced by $D$ over the standard log point. This logarithmic structure on $X$ is obtained by the charts $\NN \to \cO_{U_i}$ given by $1 \mapsto f_i$.
\end{rem}

\begin{rem} \label{rem:usual_logarithmic_differentials}
	If $X$ and $D$ are smooth over $k$, then $\Omega_X(\log D)$ is the usual sheaf of logarithmic differentials on $X$ with respect to the divisor $D$.
	If $X$ is smooth and $D$ is a normal crossing divisor, our $\Omega_X(\log D)$ is \emph{not} the usual sheaf of logarithmic differentials on $X$ with respect to $D$.
	
	For example, if $X = \AA^2 = \Spec k[x,y]$ and $D= \Spec k[x,y]/(xy)$ is the union of the two axes then $\Omega_X(\log D)$ is generated by $\rmd x$, $\rmd y$ and $\frac{\rmd(xy)}{xy}$ and is not locally free, whereas the usual sheaf of logarithmic differentials on $X$ with respect to $D$ is a free $\cO_X$-module of rank $2$, generated by $\frac{\rmd x}{x}$, $\frac{\rmd y}{y}$.
	
	See Proposition~\ref{prop:Omegalog_for_X_smooth} for a general treatment of the case in which $X$ is smooth.
\end{rem}

\begin{lem} \label{lem:two_short_exact_sequence_with_log_differentials}
  There is a short exact sequence
\begin{equation} \label{eq:residue_sequence_Omega_log}
	(0) \longrightarrow \Omega_X \longrightarrow \Omega_X( \log D)
        \longrightarrow \cO_D \longrightarrow (0) 
\end{equation}
\end{lem}

When $X$ and $D$ are smooth over $k$, the  exact sequence \eqref{eq:residue_sequence_Omega_log}  is the residue sequence.

\begin{proof}[Proof of Lemma~\ref{lem:two_short_exact_sequence_with_log_differentials}]
The short exact sequence \eqref{eq:residue_sequence_Omega_log} can be constructed locally on affine open cover of $X$ which trivialises $\cO_X(D)$. Alternatively, one can consider the commutative diagram
\[
\xymatrix{
(0) \ar[r] & 0 \ar[r] \ar[d] & \cO_X(-D) \ar[r] \ar[d] & \cO_X(-D) \ar[r] \ar[d] & (0) \\
(0) \ar[r] & \Omega_X \ar[r] & Q_{\cO_X(D)} \ar[r] & \cO_X \ar[r] & (0)
}
\]
with exact rows and apply the snake lemma.
\end{proof}

We next see that, in many cases, the sheaf $\Omega_X(\log D)$ controls
the deformations of the pair $(X,D)$, i.e.\ of the closed embedding
$D \into X$.

\begin{pro}
  \label{pro:deformations_pairs_X_D}
  Let $X$ be a scheme of finite type over the field $k$ and let
  $D \into X$ be an effective Cartier divisor. Denote by $N_D X = \cO_D(D)$
  the normal bundle of $D$ in $X$. Then:
	\begin{enumerate}[(i)]
		\item There exists an exact sequence
		\begin{equation*}
			\rH^0(D,N_D X) \to \rT \Deff{X}{D} \to \rT \Def{X} \to \rH^1(D, N_D X)
		\end{equation*}
		\item If $\rH^1(D, N_D X) = 0$, then the forgetful map $\Deff{X}{D} \to \Def{X}$ is smooth.
		\item If $X$ is reduced and $k$ is perfect, then $\Ext^1(\Omega_X(\log D), \cO_X)$ is the tangent space of $\Deff{X}{D}$ and the exact sequence in (i) is obtained by taking the cohomology exact sequence of \eqref{eq:residue_sequence_Omega_log}.
		\item If $X$ is normal and $k$ is perfect, then $\Ext^2(\Omega_X(\log D), \cO_X)$ is an obstruction space of $\Deff{X}{D}$.
	\end{enumerate}
\end{pro}

\begin{proof}
  We consider the quotient stack $[\AA^1 / \Gm]$, where the action of
  $\Gm$ on $\AA^1$ has weight $1$. This stack classifies line bundles
  with a section, see \cite[\S10.3]{olsson_book}, so we can consider
  the classifying morphism $\eta_D \colon X \to [\AA^1 / \Gm]$
  associated to the effective Cartier divisor $D$ on $X$.  We have a
  sequence of maps of algebraic stacks of finite type over $k$
	\[
	X \overset{\eta_D}\longrightarrow [\AA^1 / \Gm] \overset{p}\longrightarrow \rB \Gm \longrightarrow \Spec k
	\]
	where $p$ is given by $\AA^1 \to \Spec k$ and
        $p \circ \eta_D = \epsi_{\cO_X(D)}$ is the classifying
        morphism of the line bundle $\cO_X(D)$.  There is a
        distinguished triangle in the derived category of $X$:
	\begin{equation*}
		\rL \eta_D^* \LL_{[\AA^1 / \Gm]} \to \LL_X \to \LL_{\eta_D} \to \rL \eta_D^* \LL_{[\AA^1 / \Gm]} [1]
	\end{equation*}
	
	Denote by $t$ the standard coordinate on $\AA^1$, and by
        $\cO_{\AA^1}(r)$ the $\Gm$-equivariant sheaf on $\AA^1$ with
        weight $r \in \ZZ$. By
        Remark~\ref{rem:cotangent_quotient_stack} we have that
        $\LL_{[\AA^1 / \Gm]}$ is quasi-isomorphic to the complex of
        $\Gm$-equivariant sheaves on $\AA^1$
	\[
	\cO_{\AA^1}(-1) \overset{t}{\longrightarrow} \cO_{\AA^1}
	\]
	in degrees $0,1$. It follows that
	$\rL \eta_D^* \LL_{[\AA^1 / \Gm]}$ is quasi-isomorphic to the $2$-term complex
	\begin{equation*}
		\cO_X(-D) \to \cO_X
	\end{equation*}
	in degrees $0,1$, where the differential is given by the
        section of $\cO_X(D)$ that defines $D$. This implies that
        $\rL \eta_D^* \LL_{[\AA^1 / \Gm]}$ is quasi-isomorphic to
        $\cO_D[-1]$. Therefore, by shifting we obtain the
        distinguished triangle
	\begin{equation} \label{eq:residue_triangle}
		\LL_X \to \LL_{\eta_D} \to \cO_D \to \LL_X[1]
	\end{equation}
	The truncation of this distinguished triangle in degree 0 is the short exact sequence \eqref{eq:residue_sequence_Omega_log}, because
	we have
	\begin{equation*}
		\Omega_X(\log D) = \cH^0 (\LL_{\eta_D})
	\end{equation*}
	and
	\begin{equation*}
		\text{for all } i \neq 0 \qquad \cH^i(\LL_X) \simeq \cH^i(\LL_{\eta_D}) 
	\end{equation*}

	Giving an infinitesimal deformation of the scheme $X$ together
        with the Cartier divisor $D$ is equivalent to giving an
        infinitesimal deformation of the morphism $\eta_D$ that is
        trivial on the target.  Therefore infinitesimal automorphisms,
        tangent space and obstructions of $\Deff{X}{D}$ are controlled
        by the cotangent complex $\LL_{\eta_D}$ of the morphism
        $\eta_D$.  Since for all $i$
        $\Ext^i(\cO_D, \cO_X) = \rH^{i-1}(N_D X)$, the distinguished
        triangle~\eqref{eq:residue_triangle} induces the exact
        sequence
	\begin{align*}
		(0)  &\to \Ext^0(\LL_{\eta_D}, \cO_X) \to \Ext^0(\LL_X, \cO_X) \to \\
		\to \rH^0(N_D X)  &\to  \Ext^1(\LL_{\eta_D}, \cO_X) \to \Ext^1(\LL_X, \cO_X) \to \\
		\to \rH^1(N_D X)  &\to  \Ext^2(\LL_{\eta_D}, \cO_X) \to \Ext^2(\LL_X, \cO_X)
	\end{align*}
	where $\Ext^1(\LL_{\eta_D}, \cO_X)$ is the tangent space of $\Deff{X}{D}$ and $\Ext^2(\LL_{\eta_D}, \cO_X)$ is an obstruction space for $\Deff{X}{D}$. This immediately implies the well known fact that $\rH^1(N_D X)$ is an obstruction space for the forgetful map
	$
	\Deff{X}{D} \to \Def{X}
	$; this proves (i) and (ii).
	
	The proof of (iii) and (iv) is very similar to the one of Proposition~\ref{pro:deformations_X_L} and is omitted.
\end{proof}

Now we make some comments on the sheaves
$\cExt^i(\Omega_X(\log D), \cO_X)$, which help to compute the
$k$-vector space $\Ext^i(\Omega_X(\log D), \cO_X)$.

\begin{nota}
  Write $T_X(- \log D) = \cH om (\Omega_X(\log D), \cO_X)$.
\end{nota}

\begin{rem}
  \label{rem:log_tangents}
  If $(X,D)$ is a toric pair over $k$, then $T_X(- \log D)$ coincides
  with the usual logarithmic tangent sheaf, i.e.\ the one associated
  to the standard log structure that exists on a toric
  variety. Moreover, under these assumptions,
  $T_X(-\log D)\cong \cO_X^n$ where $n=\dim X$. Compare these two
  results for the logarithmic tangent sheaf with similar statements
  (but different assumptions) about the logarithmic cotangent sheaf in
  Remark~\ref{rem:usual_logarithmic_differentials} and in
  Proposition~\ref{prop:Omegalog_for_X_smooth}.
	
  The two results above hold because of the following argument.  Let
  $T_{\log}$ be the usual logarithmic tangent sheaf; it is well known
  that $T_{\log}\cong \cO_X^n$. Our $T_X(- \log D)$ is reflexive
  because it is the dual of a coherent sheaf. By
  Remark~\ref{rem:usual_logarithmic_differentials}, $T_{\log}$ and
  $T_X(- \log D)$ coincide on the open subset of $X$ where the divisor
  $D$ is smooth (or empty). The complement of this open subset has
  codimension $\geq 2$ in $X$, so the two reflexive sheaves $T_{\log}$
  and $T_X(- \log D)$ coincide everywhere on $X$.
\end{rem}

\begin{rem} \label{rem:ext_sheaves_Omegalog}
Let $\cO_X(-D) \subseteq \cO_X$ be the ideal sheaf of the closed embedding $D \into X$.
Let $N_D X = \cO_D(D) = \cHom (\cO_X(-D), \cO_D)$ be the normal bundle of $D$ in $X$.	
By dualising the exact sequence \eqref{eq:residue_sequence_Omega_log} we get the exact sequence
\begin{equation*}
(0) \to T_X(- \log D) \to T_X \to N_D X \to \cExt^1(\Omega_X(\log D), \cO_X) \to \cExt^1(\Omega_X, \cO_X) \to (0)
\end{equation*}
where the homomorphism
\begin{equation} \label{eq:homomorphism_TX_N_DX}
T_X \to N_D X
\end{equation}
maps a derivation $\partial \colon \cO_X \to \cO_X$ into the
composition of the restriction
$\partial \vert_{\cO_X(-D)} \colon \cO_X(-D) \to \cO_X$ with the
surjection $\cO_X \onto \cO_D$.  This implies that $T_X(- \log D)$ is
the subsheaf of $T_X$ consisting of the derivations $\partial$ such
that $\partial h \in \cO_X(-D)$ for every section $h$ of the ideal
sheaf $\cO_X(-D)$ of $D$ in $X$.  Moreover we have
\begin{equation*}
\text{for all} \; i \geq 2 \qquad \cExt^i(\Omega_X(\log D), \cO_X) \simeq \cExt^i(\Omega_X, \cO_X) 
\end{equation*}
\end{rem}

\begin{pro} \label{prop:Omegalog_for_X_smooth}
	If $X$  is smooth over $k$, then the following statements hold.
	\begin{enumerate}
		\item $\Omega_X(\log D)$ is a locally free $\cO_X$-module if and only if $D$ is smooth over $k$.
		
		\item If $X$ has pure dimension $n$ and $Z$ is the
                  closed subscheme of $D$ defined by the
                  $(n-1)$\textsuperscript{st} Fitting ideal of the
                  $\cO_D$-module $\Omega_D$, then there is an
                  isomorphism
		\begin{equation*}
			\cExt^1(\Omega_X(\log D), \cO_X) \simeq \cO_{Z}(D)
		\end{equation*}
		
		\item For all $i$ there is an isomorphism
		\begin{equation*}
			\Ext^i(\Omega_X(\log D), \cO_X) \simeq \HH^i(T_X \to N_D X)
		\end{equation*}
	where the 2-term complex on the right is in degrees $0$ and $1$ and the differential is the homomorphism in \eqref{eq:homomorphism_TX_N_DX}.
	\end{enumerate}
\end{pro}

\begin{rem} \label{rem:Omegalog_Fitting_singular_locus_D}
	Assume that we are under the assumptions of (2).
	Since $\Omega_X \vert_D$ is a locally free $\cO_D$-module of rank $n$ and the conormal bundle $\cO_D(-D)$ is a line bundle on $D$, the conormal exact sequence
	\[
	\cO_D(-D) \longrightarrow \Omega_X \vert_D \longrightarrow \Omega_D \longrightarrow (0)
	\]
	gives that the $(n-1)$\textsuperscript{st} Fitting ideal of the $\cO_D$-module $\Omega_D$ is locally generated by the entries of the $n \times 1$ matrix which represents the homomorphism $\cO_D(-D) \to \Omega_X \vert_D$.
	More explicitly, on an open subset of $X$ where $\cO_X(D)$ and $\Omega_X$ are free this homomorphism is given by the differential of the local equation which defines $D$. Therefore, $Z$ coincides set-theoretically with the non-smooth locus of $D$.
	
	If, in addition, $D$ is a reduced normal crossing divisor on $X$, then $Z$ is the singular locus of $D$ equipped with the reduced structure.
\end{rem}

\begin{rem}
	Combining Proposition~\ref{prop:Omegalog_for_X_smooth}(3) and Proposition~\ref{pro:deformations_pairs_X_D}(iii) implies that if $X$ is smooth over $k$ then $\HH^1(T_X \to N_D X)$ is the tangent space of the deformation functor of the pair $(X,D)$. This result was proved in \cite[Proposition~8]{smith_varley} via explicit computations.
\end{rem}

\begin{proof}[Proof of Proposition~\ref{prop:Omegalog_for_X_smooth}]
	
  (1) Since the question is local, we may assume that $X$ is affine
  and $D$ is the zero-locus of a global function $f \in \cO_X(X)$.  By
  base change to the algebraic closure of $k$, we may assume that $k$
  is algebraically closed.  By \'etale descent, we may assume that $X$
  is the affine space $\AA^n = \Spec k[x_1, \dots, x_n]$.
	
  The exact sequence in \eqref{eq:definition_Omega_log} implies that
  $\Omega_X(\log D)$ is the cokernel of the injective homomorphism
  $\phiv \colon \cO_{\AA^n} \to \cO_{\AA^n}^{\oplus (n+1)}$ which maps
  $1$ to
	\[
	\begin{pmatrix}
		\frac{\partial f}{\partial x_1} \\
		\vdots \\
		\frac{\partial f}{\partial x_n} \\
		-f
	\end{pmatrix}
	\]
	For all closed points $p \in \AA^n(k)$,
	the following statements are equivalent:
	\begin{itemize}
		\item the sheaf $\Omega_X(\log D)$ is free in a neighbourhood of $p$ in $\AA^n$,
		\item the stalk $\Omega_X(\log D)_p$ is a flat $\cO_{\AA^n, p}$-module,
		\item $\mathrm{Tor}_1^{\cO_{\AA^n, p}} ( \Omega_X(\log D)_p  , k_p  ) = 0$,
		\item $\phiv \otimes \mathrm{id}_{k_p} \colon k_p \to k_p^{\oplus (n+1)}$ is injective,
		\item at least one among $\frac{\partial f}{\partial x_1}(p), \dots, 
		\frac{\partial f}{\partial x_n}(p)$, $f(p)$ is non-zero,
		\item $p$ does not lie in the non-smooth locus of $D$.
	\end{itemize}
	Since $p$ is arbitrary, this concludes the proof of (1).
	
	\medskip
	
        (2) Since $X$ is smooth, $\Omega_X$ is locally free, so
        $\cExt^1(\Omega_X, \cO_X) = 0$. By the exact sequence in
        Remark~\ref{rem:ext_sheaves_Omegalog},
        $\cExt^1(\Omega_X(\log D), \cO_X)$ is the cokernel of the
        homomorphism $T_X \to N_D X$ in
        \eqref{eq:homomorphism_TX_N_DX}.

Consider the restriction $N_D X = \cO_D(D) \onto \cO_Z(D)$ from $D$
to $Z$. If we show the equality
\begin{equation} \label{eq:equality_im_ker_in_Omegalog_for_Xsmooth}
	\mathrm{im}\left(  T_X \to N_D X  \right) = \ker \left( N_D X \to \cO_Z(D)  \right)
\end{equation}
then we are done because this will imply that $\cO_Z(D)$ is the
cokernel of $T_X \to N_D X$.  The
equality~\eqref{eq:equality_im_ker_in_Omegalog_for_Xsmooth} is a local
problem in the fppf topology, so we can assume, as we did in the proof
of (1), that $k$ is algebraically closed,
$X = \AA^n = \Spec k[x_1, \dots, x_n]$ and $D$ is the zero-locus of a
non-zero polynomial $f \in k[x_1, \dots, x_n]$. In this case, the
homomorphism $T_X \to N_D X$ is $\cO_X^{\oplus n} \to \cO_D$ given by
the row of the $n$ partial derivatives of $f$ and the homomorphism
$N_D X \to \cO_Z(D)$ is just the restriction $\cO_D \to \cO_Z$.  By
Remark~\ref{rem:Omegalog_Fitting_singular_locus_D}, $Z$ is the closed
subscheme of $\AA^n$ whose ideal is generated by $f$ and its $n$
partial derivatives.  This concludes the proof of (2).

\medskip

(3) We have the following sequence of quasi-isomorphisms of complexes
\begin{equation*}
	\rR \cHom(\Omega_X(\log D), \cO_X)
	\simeq \left[ Q_{\cO_X(D)}^\vee \to \cO_X(D) \right]
	\simeq \left[ T_X \to N_D X \right]
\end{equation*}
where the 2-term complexes in the middle and on the right are both
concentrated in degrees $0$ and $1$.  The first quasi-isomorphism
holds because, since $X$ is smooth, the short exact
sequence~\eqref{eq:definition_Omega_log} gives a locally free
resolution of $\Omega_X(\log D)$.  The second quasi-isomorphism
follows via the snake lemma applied to the following morphism of short
exact sequences:
\begin{equation*}
	\xymatrix{
		(0) \ar[r]  & \cO_X \ar@{=}[d] \ar[r] & Q_{\cO_X(D)}^\vee \ar[r] \ar[d] & T_X \ar[r] \ar[d] & (0) \\	
		(0) \ar[r]  & \cO_X \ar[r] & \cO_X(D) \ar[r] & N_D X \ar[r] & (0) \\	
	}
\end{equation*}
Via the sequence of quasi-isomorphisms above we get
\begin{align*}
	\Ext^i(\Omega_X(\log D), \cO_X) &= \cH^i \, \rR \Hom(\Omega_X(\log D), \cO_X) \\
	&= \cH^i \, \rR\Gamma \, \rR \cHom(\Omega_X(\log D), \cO_X) \\
	&= \HH^i (T_X \to N_D X)
\end{align*}
This concludes the proof of Proposition~\ref{prop:Omegalog_for_X_smooth}.
\end{proof}

\begin{exa} \label{ex:Omegalog_complete_intersection_affine_space}
Let $f,g \in k[x_1, \dots, x_{n+1}]$ be polynomials such that $g,f$ is a regular sequence.
Consider the following closed subschemes of the affine space $\AA^{n+1}$: $X = \Spec k[x_1, \dots, x_{n+1}] / (g)$ and $D = \Spec k[x_1,\dots,x_{n+1}] / (f,g)$. Clearly, $D$ is an effective Cartier divisor on $X$.
The conormal sequence of $X \into \AA^{n+1} \to \Spec k$ is the exact sequence
\begin{equation*}
	\cO_X  \xrightarrow{
	\begin{pmatrix}
\frac{\partial g}{\partial x_1} \\
\vdots \\
\frac{\partial g}{\partial x_{n+1}}
	\end{pmatrix}
	}  \cO_X^{\oplus {(n+1)}} \longrightarrow \Omega_X \longrightarrow (0)
\end{equation*}
The short exact sequence \eqref{eq:definition_Omega_log} is
\begin{equation*}
	(0) \longrightarrow \cO_X \xrightarrow{ 
	\begin{pmatrix}
		\rmd f \\
		-f
	\end{pmatrix}
	 } \Omega_X \oplus \cO_X \longrightarrow \Omega_X (\log D) \longrightarrow (0)
\end{equation*}
Combining these two exact sequences gives the exact sequence
\begin{equation*}
	\cO_X^{\oplus 2}   \xrightarrow{
	\begin{pmatrix}
		\frac{\partial g}{\partial x_1} & \frac{\partial f}{\partial x_1}\\
		\vdots & \vdots  \\
		\frac{\partial g}{\partial x_{n+1}} & \frac{\partial f}{\partial x_{n+1}} \\
		0 & -f
	\end{pmatrix}
	}
	\cO_X^{\oplus (n+2)}
	\longrightarrow
	\Omega_X(\log D)
	\longrightarrow (0)
\end{equation*}
\end{exa}

\begin{exa}
  \label{exa:ODP_deformations}
  Consider the \mbox{$3$-fold} ordinary double point
  $X = \Spec k[x,y,z,w] / (xy-zw)$ and its closed subscheme
  $D = \Spec k[x,y,z,w] / (xy, zw)$.  This is
  Example~\ref{ex:Omegalog_complete_intersection_affine_space} with
  $g = xy-zw$ and $f=xy$.  $D$ is a reduced effective Cartier divisor
  on $X$ with $4$ irreducible components, each of which is isomorphic
  to $\AA^2$.  The last exact sequence in
  Example~\ref{ex:Omegalog_complete_intersection_affine_space} is
\begin{equation*}
	(0) \longrightarrow \cO_X^{\oplus 2}   \xrightarrow{
		\begin{pmatrix}
			y & y \\
			x & x  \\
			-w & 0 \\
			-z & 0 \\
			0 & -xy
		\end{pmatrix}
	}
	\cO_X^{\oplus 5}
	\longrightarrow
	\Omega_X(\log D)
	\longrightarrow (0)
\end{equation*}
By dualising and doing some column operations and row operations, we
deduce that the group $\Ext^1(\Omega_X(\log D), \cO_X)$ is isomorphic
to the cokernel of the homomorphism
$\cO_X^{\oplus 4} \to \cO_X^{\oplus 2}$ given by the matrix
\begin{equation*}
	\begin{pmatrix}
		x & y & 0 & 0  \\
		0 & 0 & z & w 
	\end{pmatrix}
\end{equation*}
This shows that $\Ext^1(\Omega_X(\log D), \cO_X)$ is isomorphic to 
\begin{equation*}
	\frac{\cO_X}{(x,y)} \oplus \frac{\cO_X}{(z,w)} = \frac{k[x,y,z,w]}{(x,y,zw)} \oplus \frac{k[x,y,z,w]}{(xy,z,w)} =  \cO_{\Gamma_1} \oplus \cO_{\Gamma_2}
\end{equation*}
where $\Gamma_1$(resp.\ $\Gamma_2$) is the closed subscheme of $D$
(and of $X$) defined by the ideal generated by $x,y,zw$ (resp.\
$xy,z,w$). Hence the fibre of $\Ext^1(\Omega_X(\log D), \cO_X)$ at the
closed point corresponding to the maximal ideal $(x,y,z,w)$ has
dimension $2$.

Now we give a more conceptual way to describe
$\Ext^1(\Omega_X(\log D), \cO_X)$.  Consider the following closed
subscheme of $D$:
$\Gamma = \Spec k[x,y,z,w] / (xy, xz, xw, yz, yw, zw)$.  It is reduced
and has $4$ irreducible components, each of which is isomorphic to
$\AA^1$.  $\Gamma$ is the closed subscheme of $D$ defined by the $2$\textsuperscript{nd}
Fitting ideal of the $\cO_D$-module $\Omega_D$, so $\Gamma$ is the
non-smooth locus of $D$ equipped with the reduced structure.
$\Gamma_1$ and $\Gamma_2$ are the union of $2$ of the $4$ irreducible
components of $\Gamma$ and $\Gamma$ is the union of $\Gamma_1$ and
$\Gamma_2$.  Consider the disjoint union
$\Gamma' = \Gamma_1 \coprod \Gamma_2$. The closed embeddings
$\Gamma_1 \into \Gamma$ and $\Gamma_2 \into \Gamma$ give a finite
surjective morphism $\nu \colon \Gamma' \to \Gamma$ which is a partial
resolution.  Then $\Ext^1(\Omega_X(\log D), \cO_X)$ is isomorphic to
$\nu_\star \cO_{\Gamma'} = \nu_\star \nu^\star \cO_\Gamma$.

\section{Miniversal deformations of pairs}
\label{sec:exampl-miniv-deform}

\subsection{Statement of purpose}
\label{sec:statement-purpose-2}

\begin{nota}
  \label{notation:Axyz}
  In this section, $\AA^4_{x,y,z,u}$ (for example) denotes $\Spec \CC[x,y,z,u]$,
  that is, affine $4$-dimensional space with a choice of coordinate
  functions labelled $x,y,z,u$.

  We adopt a similar convention in a host of similar situations.
\end{nota}

  The purpose of this section is to write down explicit examples of miniversal deformations of
  toric pairs $(Y,E)$ --- where $Y$ is locally a hypersurface and $E$ is a
  Cartier divisor --- that are needed in the proof of
  Theorem~\ref{thm:local_invariance}. 

  \smallskip
  
  We allow $Y$ to be nonproper and to have nonisolated singularities.
  Even in the simplest case when $Y=\AA^3_{x,y,z}$
  and $E=(xyz=0)\subset Y$, the miniversal deformation is
  infinite-dimensional. Indeed in this case:
  \[
    \T^1_{Y,E}=\Bigl\{\lambda+xA(x)+yB(y) +zC(z)\mid \lambda \in \CC,\; A(x) \in
    \CC[x],\; B(y) \in \CC[y],\; C(z) \in \CC[z] \Bigr\}
  \]

  To address the infinite-dimensionality we work with ind-schemes. The
  terminology of ind-schemes is introduced in~\cite{KV04, MR3701353}
  but (fortunately) we don't actually need any of the theory. 

  In \S~\ref{sec:expl-exampl-miniv} we construct ``by hand''
  several explicit examples of deformation families of toric pairs
  $(Y,E)$ over a smooth ind-scheme --- in fact, see below, over $\AA^\infty
  =\varinjlim \AA^n$. We want to check that these
  families are miniversal. The key is that, although we construct these
  families over $\AA^\infty$, we only claim miniversality over Artin
  rings. In \S~\ref{sec:some-gener-princ} we prove some general
  principles: the key result is Corollary~\ref{cor:miniversality},
  stating that a family is miniversal if it is universal for
  first-order deformations. To prove the Corollary, first, using
  Lemma~\ref{local-to-global}, we reduce to the
  local case. Second, in our setting the local case may be treated by
  the explicit and elementary Lemma~\ref{miniversal}, since we
  may assume that $Y$ is a hypersurface.

\subsection{General principles}
\label{sec:some-gener-princ}

\begin{nota}
  \label{sec:miniv-deform-pairs}
  \begin{enumerate}[(i)]
  \item   We denote by $\AA^\infty $ the ind-scheme 
  $$\AA^\infty_{a_1,a_2,\dots} = \varinjlim_m \AA^m_{a_1,\dots, a_m}$$
  where the limit is taken over the inclusions
  $$\AA^m_{a_1,\dots,a_m}=(a_{m+1}=0)\hookrightarrow
  \AA^{m+1}_{a_1,\dots,a_{m+1}}$$
\item   Writing $\AA^m_l= \Spec
  k[a_1,\ldots,a_m]/(a_1,\ldots,a_m)^{l+1}$, we denote by
  $$\widehat{\AA^\infty} = \varinjlim_l \varinjlim_m \AA^m_l$$
  the formal completion of $\AA^\infty$, also an ind-scheme.
\item In \S~\ref{sec:expl-exampl-miniv} we often have a situation where we
  regard the ``power series'' $A(x)=\sum_{i=0}^\infty a_ix^i$ as a
  function on the ind-scheme
  $\AA^1_x\times \AA^\infty_{a_0,a_1,\dots}$.
  \end{enumerate}
\end{nota}

\begin{lem}
  \label{local-to-global}
Let $Y$ be a normal variety and $E \subset Y$ a Cartier divisor.
Assume that $\rH^i(Y, T_Y(-\log E))=0$ for $i=1,2$. 
Let $\{U_i\}_{i \in I}$ be an affine open covering of $Y$.
Let $A$ be an Artinian local $k$-algebra, finite dimensional over
$k$. For all $i \in I$, let $(\cU_i,\cE_i)/\Spec A$ be an
infinitesimal deformation of the pair $(U_i,E|_{U_i})$ over $A$.

If the induced deformations of
$(U_i \cap U_j, E|_{U_i \cap U_j})$ over $A$ are isomorphic, then
there exists a deformation $(\cY,\cE)/\Spec A$ of $(Y,E)$ over $A$
such that its restriction to $U_i$ is isomorphic to
$(\cU_i,\cE_i)/\Spec A$ for each $i \in I$, and it is uniquely
determined up to isomorphism of deformations. \qed
\end{lem}

\begin{dfn}
  \label{dfn:miniversal}
  Let $(Y,E)$ be a pair consisting of a variety $Y$ and Cartier
  divisor $E$.

  Consider deformations of the pair $(Y,E)$ over Artinian local
  $\CC$-algebras $A$, assumed finite-dimensional over $\CC$.

  A deformation $(\cY^u,\cE^u)$ of $(Y,E)$ over either $0 \in \AA^d$
  for some $d \in \NN$, or  the ind-scheme $0 \in \AA^{\infty} =
  \varinjlim \AA^m$ is \emph{miniversal} if:
  \begin{enumerate}
\item For all $A$ and deformation $(\cY,\cE)$ of $(Y,E)$ over $A$
  there is a morphism $\Spec A \rightarrow \AA^\infty$ such that
  $(\cY,\cE)/A$ is isomorphic to the pull back of
  $(\cY^u,\cE^u)/\AA^\infty$;
  \item If $A=k[t]/(t^2)$ the morphism $\Spec A \rightarrow \AA^{\infty}$ in (1) is uniquely determined.
\end{enumerate}
\end{dfn}

\begin{lem}
  \label{miniversal}
  Let $(Y,E)$ be a pair consisting of an affine hypersurface
  $Y=(f=0) \subset \AA^n$ and a Cartier divisor $E =(g=0) \subset Y$.
  Consider deformations of the pair $(Y,E)$ over Artinian local
  $\CC$-algebras $A$, assumed finite-dimensional over $\CC$.  There is
  a miniversal deformation $(\cY^u,\cE^u)$ of $(Y,E)$ over either
  $0 \in \AA^d$ for some $d \in \NN$ or the ind-scheme
  $0 \in \AA^{\infty} := \varinjlim \AA^m$.

\smallskip

Moreover, this miniversal deformation may be explicitly constructed as
follows:

\smallskip

Let $h_i  \in \CC[x_1,\ldots,x_n]$, $i \in I$, be a lift of a basis of the $\CC$-vector space
$$\T^1_Y=\coker \left( \rH^0(\AA^n, T_{\AA^n}) \rightarrow
  \rH^0(\AA^n, N_Y \AA^n)\right)
=\textstyle{\CC[x_1,\ldots,x_n]/(f,\frac{\partial f}{\partial x_1},\ldots, \frac{\partial f}{\partial x_n})}$$

Let $k_j \in \CC [x_1,\ldots,x_n]$, $j \in J$ be a lift of a basis of the $\CC$-vector space
$$\T^1_{Y,E} = \coker \left( \rH^0(Y,T_Y) \rightarrow
  \rH^0(Y, N_E Y) \right)$$
under the surjection
$$\CC [x_1,\ldots,x_n] \rightarrow \CC[x_1,\ldots,x_n]/(f,g) =
\rH^0(Y, N_E Y)$$
\noindent Then a miniversal deformation of the pair $(Y,E)$ is given by the
family
\begin{align*}
\cY^u & = (f+\sum t_i h_i = 0) \subset \AA^n_{x_1,\dots,x_n} \times \AA^{d}_{\{t_i,
        u_j\mid i\in I, \;j\in J\}} \\
\cE^u & = (g+\sum u_j k_j =0) \subset \cY  
\end{align*}
where $d$ is the cardinality of the disjoint union $I \sqcup J$ (either finite or countably infinite).
\end{lem} 

\begin{proof}
  By definition of ind-scheme, a morphism $\Spec A \rightarrow \AA^d$
  is given by an assignment $t_i \mapsto a_i$, $u_j \mapsto b_j$ for
  some $a_i, b_j \in A$ with finitely many nonzero. The assertion
  follows from e.g.~\cite{Artin76}, p.~23, Theorem 6.1 and Remark 6.1.
\end{proof}

\begin{cor}
  \label{cor:miniversality}
  Let $(Y,E)$ be a pair of a normal variety and Cartier divisor. Assume that
  $(Y,E)$ has an atlas of charts as in Lemma~\ref{miniversal} and that $\rH^i(Y, T_Y(-\log E))=0$ for $i=1,2$.

  Suppose given a deformation of $(Y,E)$ over a base $0 \in \AA^e$ for
  $e$ finite or countably infinite, the deformation is miniversal iff
  it is universal for first order deformations, equivalently, the
  Kodaira--Spencer map $T_0\AA^e \rightarrow \T^1_{Y,E}$ is an
  isomorphism. 
\end{cor}

\begin{proof} By Lemma~\ref{local-to-global} we may assume that $Y$ is
  affine. Let $(\cY^u, \cE^u)\to \AA^d$ be the miniversal family
  provided by Lemma~\ref{miniversal}. By versality of this family, for all orders $l$ we have compatible maps $\AA^e_l
  \rightarrow \AA^d_l$, where:
  \begin{itemize}
  \item if $e$ is finite, $\AA^e_l = \Spec k[s_1,\ldots
    s_e]/(s_1,..,s_e)^{l+1}$;
  \item if $e$ is countably infinite and $\AA^e=\varinjlim \AA^m$, $\AA^e_l : = \varinjlim \AA^m_l$
  \end{itemize}
  Now one checks that if the map for $l=1$ is an isomorphism then for
  all $l$ the maps $\AA^e_l \rightarrow \AA^d_l$ are isomorphisms.
\end{proof}


\subsection{Explicit examples of miniversal deformations of pairs}
\label{sec:expl-exampl-miniv}

Below we describe several deformation families of toric pairs $(Y,E)$ over
$\AA^\infty$ that are flat over a Zariski open neighborhood of the
origin, but we only claim (mini)versality for infinitesimal
deformations. In all cases the claimed miniversality follows from the description of $\T^1_{Y,E}$
given in Lemma~\ref{lem:toricqODPdeformations} and
Corollary~\ref{cor:miniversality}. In our setting, the assumption $\rH^i(Y, T_Y(-\log
E))=0$ for $i=1,2$ of the Corollary is satisfied by Theorem~\ref{thm:3}(2).
We will also need to work locally on $X$, where
$\pi \colon Y \rightarrow X$ is a crepant partial resolution of a
Gorenstein toric Fano \mbox{$3$-fold}, for which we need the assertion
$R^i\pi_*T_Y(-\log E))=0$, which is proved in the same way.


\begin{exa}
  \label{exa:A3}
  A miniversal deformation of the pair
  \[
    \Bigl(Y, E\Bigr)=\Bigl(\AA^3_{x,y,z},(xyz=0)\Bigr)
  \]
  is given by the family
  \begin{align*}
    \cY & = \AA^3_{x,y,z} \times \cM\\
    \cE & = \Bigl(xyz+ \lambda+xA(x)+yB(y) +zC(z)=0\Bigr)\subset \cY
  \end{align*}
  over $\cM= \AA^1_\lambda \times
    \AA^\infty_{a_0,a_1,\dots} \times \AA^\infty_{b_0,b_1,\dots}
    \times \AA^\infty_{c_0,c_1,\dots}$, where  $A(x) = \sum_{i=0}^\infty a_ix^i$, $B(y) = \sum_{i=0}^\infty
  b_iy^i$, $C(z) = \sum_{i=0}^\infty c_iz^i$.
\end{exa}

\begin{exa}
  \label{exa:ODP}
  We describe a miniversal deformation of the pair $(Y, E)$ where
  \begin{align*}
    Y& = (xy+zw=0) \subset \AA^4_{x,y,z,w} \\
    E& = (zw=0) \subset Y
  \end{align*}
   
  The following claim is a simple exercise based on the previous sections,
especially Proposition~\ref{pro:deformations_pairs_X_D}~(iii) and
Lemma~\ref{lem:toricqODPdeformations}.
  \begin{claim}
    \label{cla:miniversal_ODP} The family
    \[g\colon (\cY, \cE)\to \cM = \AA^1_\lambda \times \AA^1_\mu
      \times \AA^\infty_{a_0,a_1,\dots}\times
    \AA^\infty_{b_0,b_1,\dots} \times \AA^\infty_{c_0,c_1,\dots}
    \times \AA^\infty_{d_0,d_1,\dots}
    \]
  given by equations as follows is a miniversal deformation of the
  pair $(Y,E)$:
    \begin{align}
   \label{Eq_ODP}   
           \cY &  =
                              \Bigl(xy+zw=\lambda+\mu +
                              xA(x)+yB(y)+zC(z) + wD(w)\Bigr) 
                              \subset \AA^4_{x,y,z,w} \times \cM \\
    \label{Eq_ODP_E}  \cE & = \Bigl(zw=\lambda+xA(x)+yB(y)\Bigr) \subset \cY
    \end{align}
    where $A(x)=\sum_{i=0}^\infty a_ix^i$, $B(y)=\sum_{i=0}^\infty b_iy^i$, etc.
   
    The Kodaira--Spencer map of the family is shown in Fig.~\ref{fig:KS_ODP} , where
    the picture represents a vector in $\T^1_{Y,E}=\rH^0\left(\Delta,\nu_\star \nu^\star (-K_Y)\right)$,
      see Lemma~\ref{lem:toricqODPdeformations}.
      \begin{figure}[ht]
  \centering
\begin{tikzpicture}
  \draw[thick,red] (0,0) -- (0,-4); \draw[thick,blue] (-2,-2) -- (2,-2);
  \node at (1.25,-0.25) {$\lambda + xA(x)$};
  \node at (-1.25,-3.75) {$\lambda + yB(y)$};
  \node at (1.5,-2.25) {$\mu + zC(z)$};
  \node at (-1.5,-1.75) {$\mu + wD(w)$};
\end{tikzpicture}
\caption{Kodaira--Spencer map of the family of Equations~\ref{Eq_ODP}
  and~\ref{Eq_ODP_E}.}
  \label{fig:KS_ODP}
\end{figure}
  \end{claim}
\end{exa}

\begin{exa}
  \label{exa:cA}
  In this example we describe a miniversal deformation of the toric pair
  $(Y,E)$ where $Y$ is the product of the minimal resolution of the
  surface $\text{A}_n$-singularity with $\AA^1$.

  \smallskip
  
  In more detail, consider the pair $(X,D)$ where
  \begin{align*}
    X&=\text{A}_{n}\times \AA^1_u =\left(xy+z^{n+1}=0\right)
       \subset \AA^4_{x,y,z,u} \\
    D&=(zu=0) \subset X
  \end{align*}
  Note that $D$ has three irreducible components: $D=D_u+D_x+D_y$
  where $D_u=(u=0)$, $D_x=(z=y=0)$ and $D_y=(z=x=0)$.
  Let $\pi \colon Y\to X$ be the minimal resolution
  with chain of exceptional divisors $E_1,\dots, E_n$ and strict
  transforms $D_u^\prime$, $E_0=D_x^\prime$, $E_{n+1}=D_y^\prime$ as
  pictured in Fig~\ref{fig:KS_cA}.

\smallskip
  
  Our goal is to describe
  the miniversal deformation family of the pair $(Y,E)$, where
  $E=D_u^\prime + \sum_{i=0}^{n+1} E_i$.

\smallskip
  
Let
\[\cM= \left(\AA^\infty \right)^{n+1} \times \AA^1_\lambda
\times \AA^\infty_{b_0,b_1,\dots} \times \AA^\infty_{c_0,c_1,\dots} \]
and consider the family of pairs $f\colon (\cX, \cD)\to \cM$
  given by equations:
    \begin{align} \label{Eq_cA} 
      \cX & =
\Bigl[ xy +\prod_{i=1}^{n+1} \left(z-A_i (u) \right) =0 \Bigr]\subset
\AA^4_{x,y,z,u} \times \cM 
      \\ \label{Eq_cA_D}
\cD & = \Bigl[uz+ \lambda + xB(x) +yC(y) =0 \Bigr] \subset \cX
    \end{align}
  where:
  \begin{itemize}
  \item For $j=1, \dots, n+1$, denoting by $a_{j0}, a_{j1},\dots$ the
    coordinate functions on the $j$\textsuperscript{th} factor of the
    product $\left(\AA^\infty\right)^{n+1}$, $A_j(u)=\sum_{i=0}^\infty
    a_{ji}u^i$;
  \item $B(x)=\sum_{i=0}^\infty b_ix^i$, $C(y)=\sum_{i=0}^\infty c_i y^i$.
  \end{itemize}
  \begin{claim} \label{claim_mini_cA}
    There exists a simultaneous resolution $\Pi \colon (\cY ,\cE)\to
    (\cX, \cD)$ inducing a flat morphism $g=f\circ \Pi \colon (\cY,
    \cE)\to \cM$ where:
    \begin{enumerate}[(i)]
    \item The rational map $(Y,E)\dasharrow (\cY, \cE)$ is a morphism
      that identifies $(Y,E)$ with the fibre $g^\star (0)$;
    \item With the identification in Part~(i), $g\colon (\cY, \cE) \to
      0\in \cM$ is a miniversal family for the pair $(Y,E)$;
    \item The Kodaira--Spencer map of the family is shown in
      Fig.~\ref{fig:KS_cA}, where the picture represents a vector in
      $\T^1_{Y,E}=\rH^0\left(\Delta,\nu_\star \nu^\star (-K_Y)\right)$,
      see Lemma~\ref{lem:toricqODPdeformations}.
      \begin{figure}[ht]
  \centering
\begin{tikzpicture}
  \draw[thick] (0,0) -- (2,-1) -- (3,-2) -- (3,-3)
  (3,-5) -- (3,-6) -- (2,-7) -- (0,-8)
  (2, -1) -- (7, -1) (3, -2) -- (7,-2) (3,-3) -- (7,-3)
  (3,-6) -- (7,-6) (2,-7) -- (7, -7);
  \draw[thick, dashed] (3,-3.5) -- (3, -4.5);
  \node at (7.4, -8) {$E_0=D_x^\prime$};
  \node at (7, -6.5) {$E_1$};
  \node at (7.2, -2.5) {$E_{n-1}$};
  \node at (7, -1.5) {$E_n$};
  \node at (7.6, -0.5) {$E_{n+1}=D_y^\prime$};
  \node at (0,-4) {$D^\prime_u$};
  \node at (4.8,-0.7) {$\lambda + uA_{n+1}(u)$};
  \node at (4.8,-1.7) {$\lambda + uA_{n}(u)$};
  \node at (4.8,-2.7) {$\lambda + uA_{n-1}(u)$};
  \node at (4.8,-5.7) {$\lambda + uA_{2}(u)$};
  \node at (4.8,-6.7) {$\lambda + uA_{1}(u)$};
  \node at (-1,-0.5) {$\lambda + yC(y)$};
  \node at (2,-1.5) {$\lambda$};
  \node at (2.5,-2.5) {$\lambda$};
  \node at (2.5,-5.5) {$\lambda$};
  \node at (2,-6.5) {$\lambda$};
  \node at (-1,-7.5) {$\lambda + xB(x)$};
\end{tikzpicture}
\caption{Kodaira--Spencer map of the family of Equations~\ref{Eq_cA}
  and~\ref{Eq_cA_D}.}
  \label{fig:KS_cA}
\end{figure}
    \item The action of the symmetric group $\mathfrak{S}_{n+1}$ that permutes
      the factors of $(\AA^\infty)^{n+1}$ lifts uniquely to:
      \begin{itemize}
      \item a biregular action on $(\cX, \cD)$ such that
        $f\colon (\cX,\cD) \to \cM$ is equivariant, and
      \item a birational action on the pair $(\cY,\cE)$ such that
        $\Pi \colon \cY \to \cX$ is equivariant.
      \end{itemize}
    \end{enumerate}
  \end{claim}
\end{exa}
We prove the claim by induction on $n$; see
also~\cite[2.2.2.~Theorem]{MR1144527}. Consider the blow up
$\cX^\prime \to \cX$ of the divisor
\[
Z=(x=z-A_1(u)=0) \subset \cX
\]
The variety $\cX^\prime$ is given by the equation $(z^\prime x =
x^\prime A_1) \subset \cX \times \PP^1_{x^\prime:z^\prime}$.

In the chart where
$z^\prime \neq 0$, setting $\widetilde{x}=x^\prime/z^\prime$ and
substituting $x=\widetilde{x} A_1$, we get
\[
\cX^\prime \cap \{z^\prime \neq 0\}=\Bigl[ \widetilde{x}y +\prod_{i=2}^n \left(z-A_i (u) \right) =0 \Bigr]\subset
\AA^4_{\widetilde{x},y,z,u} \times \cM
\]
and we proceed to resolve singularities by induction on $n$.

In the chart
where $x^\prime \neq 0$, setting $\widetilde{z}=z^\prime/x^\prime$ and
substituting $z=x\widetilde{z}+A_1$, we get
\[
  \cX^\prime \cap \{x^\prime \neq 0\}=
  \Bigl[ y +\widetilde{z} \prod_{i=2}^{n+1} \left(z+(A_1(u)-A_i (u) \right) =0 \Bigr]\subset
\AA^4_{x,y,\widetilde{z},u} \times \cM
\]
and, solving for $y$, we see that $\cX^\prime$ is nonsingular in
this chart. Working in this chart, we can prove the statement concerning the
Kodaira--Spencer map of the family. Indeed, in this chart
\[
\cD^\prime \cap \{x^\prime \neq 0\}=\Bigl[ x\widetilde{z}u + \lambda
+uA_1(u)+xB(x) + yB(y) +\cdots =0\Bigr] \subset \cX
\]

It follows from this last expression, by induction on $n$, that the
Kodaira--Spencer map identifies $T_{0}{\cM}$ with $\T^1_{Y,E}$, which
shows that the family $g\colon (\cY,\cE)\to 0\in \cM$ is indeed a
versal family.  
\end{exa}

\begin{exa}
  \label{exa:uz2}
  In this example we describe a miniversal deformation of the pair
  $(Y,E)$ given by the tropical arrangement in
  Fig.~\ref{fig:uz2_trop}.

  \smallskip

  In other words, consider the pair $(X,D)$ where
  \begin{align*}
    X&=(xy - uz^2=0)\subset \AA^4_{x,y,z,u} \\
    D&=(uz=0) \subset X
  \end{align*}
  Note that $X$ has an $\text{A}_1$-singularity along the
  $u$-axis. Our goal is to describe a miniversal deformation of the
  pair $(Y,E)$ where $\pi \colon Y \to X$ is the blow-up of the
  divisor $Z=(x=z=0)\subset X$, and $E=\pi^{-1} (D)$.

  \smallskip

  The variety $Y$ is given by the equation
$(z^\prime x=x^\prime  z) \subset X \times \PP^1_{x^\prime : z^\prime}$.

  In the chart where $z^\prime \neq 0$, setting
  $\widetilde{x}=x^\prime/z^\prime$, the equations of the pair $(Y,E)$ are given by
  substituting $x=\widetilde{x}z$ in the equations of $(X,D)$:
  \begin{align*}
    Y & = (\widetilde{x}y-uz=0) \subset \AA^4_{\widetilde{x},y,z,u} \\
    E & = (uz = 0) \subset Y
  \end{align*}
  Note that this is the pair of a \mbox{$3$-fold} ordinary double
  point together with its toric boundary divisor, already considered
  in Example~\ref{exa:ODP}.
  
  In the chart where $x^\prime \neq 0$, setting $\widetilde{z} =
  z^\prime/x^\prime$, the equations of the pair $(Y,E)$ are given by
  substituting $z=x\widetilde{z}$ in the equations of $(X,D)$:
  \begin{align*}
    Y & = (y-xu\widetilde{z}^2=0) \subset \AA^4_{x,y,\widetilde{z},u} \\
    E & = (xu\widetilde{z}=0) \subset Y
  \end{align*}
  We represent the pair $(Y,E)$ with the generic tropical
  arrangement given in Fig.~\ref{fig:uz2_trop}, where we also label
  the ``axes'' by their respective coordinate functions. 
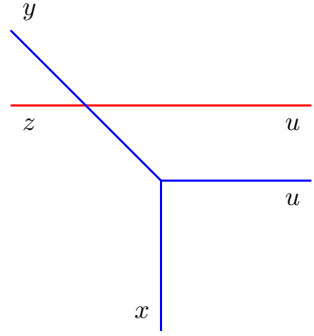
\begin{figure}[ht]
  \centering
\begin{tikzpicture}
  \draw[thick,red] (0,-2) -- (2,-2);
  \draw[thick,red] (0,-2) -- (-2,-2);
  \draw[thick,blue] (0,-3) -- (-2,-1);
  \draw[thick,blue] (0,-3) -- (2, -3);
  \draw[thick,blue] (0, -3) -- (0,-5);
  \node at (1.75,-2.25) {$u$};
  \node at (1.75,-3.25) {$u$};
  \node at (-0.25,-4.75) {$x$};
  \node at (-1.75,-2.25) {$z$};
  \node at (-1.75,-0.75) {$y$};
\end{tikzpicture}
\caption{Tropical arrangement for the pair \texorpdfstring{$(Y,E)$}{(Y,E)} of Example~\ref{exa:uz2}.}
  \label{fig:uz2_trop}
\end{figure}
 Consider the family of pairs
\[
f\colon (\cX, \cD) \to \cM = \AA^1_\lambda \times \AA^1_\mu \times
\AA^\infty_{a_0,a_1,\dots}  \times \AA^\infty_{b_0,b_1,\dots}\times
\AA^\infty_{c_0,c_1,\dots}\times
(\AA^\infty)^2
\]
given by equations as follows:
\begin{align}
  & \label{eq:mini_uz2}
    \begin{aligned}
    \cX = \Bigl[xy= \bigl(uz+\lambda+\mu + xA(x)+yB(y)+zC(z) +
  uD_1(u)\bigr)\bigl(z-D_0(u)\bigr) \Bigr] \subset
  \AA^4_{x,y,z,u}\times \cM
    \end{aligned}\\
  & \label{eq:mini_uz2D} \cD = \Bigl[uz+ \lambda + xA(x) + yB(y)  =0
    \Bigr] \subset \cX
\end{align}
where $A(x)=\sum_{i=0}^\infty a_ix^i$, etc.

\begin{claim}
  \label{claim:defo_uz2}
  Let $\Pi \colon \cY \to \cX$ be the blow-up of the
  divisor
  \[
Z=(x=z-D_0(u)=0)\subset \cX
\]
  and $\cE=\Pi^{-1}(\cD)$. Then $g=f\circ \Pi \colon (\cY, \cE)\to \cM$ is a flat morphism where
  \begin{enumerate}[(i)]
  \item The rational map $(Y,E)\dasharrow (\cY,\cE)$ is a morphism
    that identifies $(Y,E)$ with the fibre $g^\star (0)$;
  \item With the identification in Part~(i), $g\colon (\cY, \cE) \to
    0\in \cM$ is a miniversal family for the pair $(Y,E)$;
  \item   The Kodaira--Spencer map of the family is shown in
  Fig.~\ref{fig:KS_uz2}, where the picture represents a vector in
  $\T^1_{Y,E}=\rH^0\left(\Delta, \nu_\star \nu^\star (-K_Y) \right)$,
  see Lemma~\ref{lem:toricqODPdeformations}.
  \begin{figure}[ht]
  \centering
\begin{tikzpicture}
  \draw[ thick,red] (0,-2) -- (3,-2) (0,-2) -- (-3,-2);
  \draw[thick,blue] (0,-3) -- (-3,0)
  (0,-3) -- (3, -3) (0, -3) -- (0,-5);
  \node at (1.75,-2.25) {$\mu+uD_1(u)$};
  \node at (1.75,-3.25) {$\lambda+uD_0(u)$};
  \node at (-0.75,-2.75) {$\lambda$};
  \node at (-1.25,-4.75) {$\lambda + xA(x)$};
  \node at (-2.75,-2.25) {$\mu + zC(z)$};
  \node at (-1.5,-0.25) {$\lambda +yB(y)$};
\end{tikzpicture}
\caption{Kodaira--Spencer map of the family of
  Equations~\ref{eq:mini_uz2} and~\ref{eq:mini_uz2D}.}
  \label{fig:KS_uz2}
\end{figure}
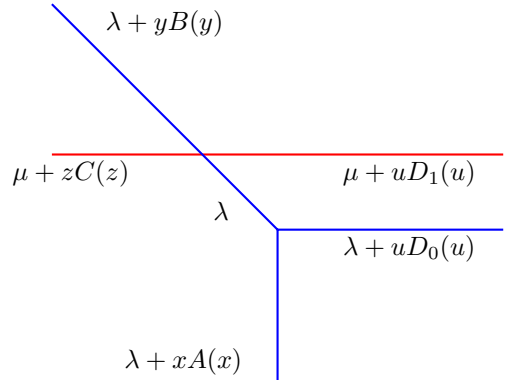
\item Let $\Pi^\prime \colon \cY^\prime \to \cX$ be the blow-up of the
  divisor
\[
Z^\prime = \Bigl(y=z-D_0(u)=0\Bigr) \subset \cX
\]
and $\cE^\prime=\Pi^{\prime\, -1}(\cD)$. Then $g^\prime=f\circ
\Pi^\prime\colon (\cY^\prime, \cE^\prime) \rightarrow \cM$ is a miniversal family for
the pair $(Y^\prime, E^\prime)$ given by the tropical arrangement in
Fig.~\ref{fig:KS_uz2prime}, also showing the Kodaira--Spencer map of
the family.
 \begin{figure}[ht]
  \centering
\begin{tikzpicture}
  \draw[ thick,red] (0,-3) -- (3,-3) (0,-3) -- (-3,-3);
  \draw[thick,blue] (-1,-2) -- (-3,0)
  (-1,-2) -- (3, -2) (-1, -2) -- (-1,-5);
  \node at (1.75,-2.25) {$\lambda+uD_0(u)$};
  \node at (1.75,-3.25) {$\mu+uD_1(u)$};
  \node at (-1.25,-2.5) {$\lambda$};
  \node at (-2.25,-4.75) {$\lambda + xA(x)$};
  \node at (-2.75,-3.25) {$\mu + zC(z)$};
  \node at (-1.5,-0.25) {$\lambda +yB(y)$};
\end{tikzpicture}
\caption{Kodaira--Spencer map of the family of Claim~\ref{claim:defo_uz2}.}
  \label{fig:KS_uz2prime}
\end{figure}
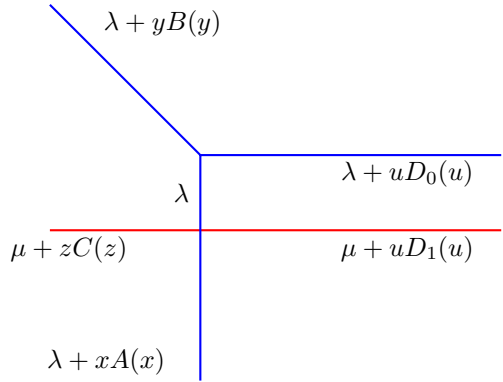
  \end{enumerate}
\end{claim}

The variety $\cY$ is given by the equation $z^\prime x = x^\prime
\bigl(z-D_0(u)\bigr)$ in $\cX \times \PP^1_{x^\prime : z^\prime}$. We
examine in turn the two charts of the blow-up.

In the chart $z^\prime \neq 0$, setting
$\widetilde{x}=x^\prime/z^\prime$ and substituting
$x=\widetilde{x}(z-D_0(u))$ in the equation for $\cX$ we obtain:
\[
\widetilde{x}y=uz+\lambda + \mu + xA(x) + yB(y) +zC(z)+uD_1 (u)
\]
with $\cE$ given by $\lambda+uz+xA(x)+yB(y)=0$. Now compare with
Example~\ref{exa:ODP}.

In the chart $x^\prime \neq 0$, setting
$\widetilde{z}=z^\prime/x^\prime$ and substituting
\[
z=\widetilde{z}x+D_0(u)
\]
in the equation for $\cX$ we obtain
\[
y=\widetilde{z}\Bigl[ u\bigl(\widetilde{z}x+D_0(u)\bigr) + \lambda +
  \mu + xA(x)+\cdots \Bigr] 
\]
where the important thing is that we can solve for $y$, that is, $\cY$
is nonsingular in this chart. The divisor $\cE$ is given by:
\[
u\bigl(\widetilde{z}x+D_0(u)\bigr)+\lambda + xA(x) + yB(y) = 0
\]
and then compare with Example~\ref{exa:A3}.

(iv) is similar to (i)--(iii).
\end{exa}

\section{Equivariant \texorpdfstring{$G$}{G}-structures on versal deformations after Rim}
\label{sec:equiv-g-struct}

\begin{setup}
  \label{sec:Rim}

  Fix an algebraically closed field $k$, and denote by $\frakB$ the
category of fat points over $k$. Consider a deformation
category 
$\frakD$ fibered in groupoids over $\frakB$ with finite dimensional
tangent space $\T^1=|\frakD (\varepsilon)|$. For example, we can take
$\frakD$ to be the category of infinitesimal deformations of a proper algebraic
variety $X$:
\[
\frakD = \catDef{X} 
\]
or a variant thereof.

Denote by $\gamma \colon \frakD \to \frakB$ the structure functor. 

From now on we assume that for some $d>0$ the $d$-dimensional torus
$\TT=\GG_m^d$ acts on $\frakD$. Then $\TT$ acts on $\T^1$ and we fix a
$\TT$-invariant subspace $W\subset \T^1$. Let $\frakB^\prime$ be the
category whose objects are pairs $(S, j)$ of a fat point $S$ over $k$
and an embedding $j\colon S \hookrightarrow W$, and let
$\frakD^\prime = \frakD | \frakB^\prime$. The group $\TT$ acts on
$\frakD^\prime$ and $\frakB^\prime$ and the functor
$\gamma \colon \frakD^\prime \to \frakB^\prime$ is $G$-equivariant.  
\end{setup}

\begin{thm}
  \label{thm:Tequivariant_versal}
  Assume Setup~\ref{sec:Rim}. There exists a formal object $\widehat{\cX} \in
  \ob \widehat{\frakD^\prime}$, equipped with a $\TT$-action, and
  characterised by the following versal property.

  For all objects $\xi \in \ob \frakD^\prime$ equipped with a
  $\TT$-action such that the morphism $\xi \to \gamma (\xi)$ is
  $\TT$-equivariant, there exists a (non-unique) $\TT$-equivariant morphism
  $i\colon \gamma (\xi) \to \gamma (\widehat{\cX})$ in
  $\widehat{\frakB^\prime}$ and a fibre square in $\widehat{\frakD^\prime}$
  \[
    \xymatrix{
    \xi \ar[r] \ar[d]& \widehat{\cX} \ar[d]\\
    \gamma(\xi) \ar[r]_i& \gamma (\widehat{\cX})}
\]
 where all morphisms are $\TT$-equivariant. \qed
\end{thm}


\end{document}